\documentclass[11pt]{amsart}
\usepackage{mathrsfs}
\usepackage[mathcal]{euscript}
\usepackage{graphicx}
\usepackage{amsthm,amscd}
\usepackage{calc}
\usepackage{mathrsfs,dsfont}
\usepackage{fancyhdr,amscd}
\usepackage{anysize}
\usepackage{amsmath,amssymb,amsfonts}
\usepackage{epsfig,enumerate}
\usepackage{CJK,CJKnumb,fancyhdr,amscd}
\usepackage{latexsym,lineno}
\usepackage{indentfirst, latexsym}
\usepackage{graphics}
\usepackage[all,poly,knot]{xy}
\usepackage[pagebackref]{hyperref}

\providecommand{\U}[1]{\protect\rule{.1in}{.1in}}

\numberwithin{equation}{section}

\def\Div{{\rm Div}}
\def\Bl{{\rm Bl}}
\def\rk{{\rm rk}}

\def\n{\nabla}

\theoremstyle{plain}

\newtheorem{thm}{Theorem}[section]
\newtheorem{mthm}[thm]{Main Theorem}
\newtheorem{lemma}[thm]{Lemma}

\newtheorem{prop}[thm]{Proposition}
\newtheorem{cor}[thm]{Corollary}
\newtheorem{quest}[thm]{Question}
\newtheorem{con}[thm]{Conjecture}

\theoremstyle{definition}
\newtheorem{defn}[thm]{Definition}
\newtheorem{ex}[thm]{Example}
\newtheorem{rem}[thm]{Remark}

\newtheorem{expe}[thm]{Expectation}

\newcommand{\comment}[1]{}

\usepackage{fancyhdr}
\begin{document}
\begin{CJK*}{GBK}{kai}
\CJKtilde
\CJKindent

\title[Invariance of plurigenera and Chow-type lemma]{Invariance of plurigenera and Chow-type lemma}


\author[Sheng Rao]{Sheng Rao}
\author[I-Hsun Tsai]{I-Hsun Tsai}

\address{Sheng Rao, School of Mathematics and Statistics, Wuhan  University,
Wuhan 430072, People's Republic of China;
Universit\'{e} de Grenoble-Alpes, Institut Fourier (Math\'{e}matiques)
UMR 5582 du C.N.R.S., 100 rue des Maths, 38610 Gi\`{e}res, France}
\email{likeanyone@whu.edu.cn, sheng.rao@univ-grenoble-alpes.fr}

\address{I-Hsun Tsai, Department of Mathematics, National Taiwan University, Taipei 10617,
Taiwan}
\email{ihtsai@math.ntu.edu.tw}
\dedicatory{In Memory of Jean-Pierre Demailly (1957-2022)}

\thanks{Rao is partially supported by NSFC (Grant No. 11671305, 11771339, 11922115) and the Fundamental Research Funds for the Central Universities  (Grant No. 2042020kf1065).}
\date{\today}

\subjclass[2010]{Primary 32G05; Secondary 32S45, 18G40, 32C35}
\keywords{Deformations of complex structures; Modifications; resolution of singularities, Spectral sequences, hypercohomology, Analytic sheaves and cohomology groups}

\begin{abstract}
This paper  answers a question of Demailly whether a smooth family of nonsingular projective varieties admits the deformation invariance of plurigenera affirmatively, and proves this more generally for a flat family of varieties with only canonical singularities and uncountable ones therein being of general type and also two Chow-type lemmata on the structure of a family of projective complex analytic spaces.
 \end{abstract}
\maketitle

\setcounter{tocdepth}{1}
\tableofcontents

\section{Introduction: main results and background}\label{Introduction}
The plurigenera are fundamental discrete invariants for the classification of complex varieties and deformation invariance of plurigenera is a central topic in deformation theory.
One main motivation of this paper is to answer one question of J.-P. Demailly \cite{Dem17} affirmatively:
\begin{quest}[Demailly]\label{quest}
Let $\pi: \mathcal{X}\rightarrow \Delta$ be a holomorphic family of compact complex manifolds over a unit disk in $\mathbb{C}$ such that each fiber $X_t:=\pi^{-1}(t)$ is projective for any $t\in \Delta$. Then for each positive integer $m$, is the $m$-genus $\dim_{\mathbb{C}} H^0(X_t, K_{X_t}^{\otimes m})$ independent of $t\in \Delta$? Here $K_{X_t}$ is the canonical bundle of ${X_t}$.
\end{quest}

In Question \ref{quest} note that there is no assumption on the existence
of an ample line bundle on $\mathcal{X}$. Concerning the difference between ``family of projective manifolds" and ``projective family of smooth manifolds", while fiberwise K\"ahler condition cannot  lead to locally K\"ahler condition \cite[Example 3.9 credited to P. Deligne]{bin}, J. Koll\'ar gives an
example of a smooth family of projective surfaces (over
an affine line) which is however not a projective family in his recent work
\cite[Example 4]{k21a}; see also our Lemma \ref{Chow-sm} and Remark \ref{4.26} for some necessary conditions.  Here, we are concerned with a generalized version: a smooth family of Moishezon manifolds; see Theorem \ref{main-thm} below.

Throughout this paper, $\pi:\mathcal{X}\rightarrow B$ will denote a proper surjective morphism from a complex variety to a complex manifold with connected (but possibly reducible) complex analytic spaces as its fibers, which is called a \emph{family}. For $t\in B\setminus \{b\}$ with some $b\in B$, $$X_t:=\pi^{-1}(t)$$  will be a \emph{general fiber} of this family and $X_b:=\pi^{-1}(b)$ is the reference fiber. Thus, by a \emph{holomorphic (or smooth) family} $\pi:\mathcal{X}\rightarrow B$ of compact complex manifolds, we mean that $\pi$ is a proper surjective holomorphic submersion between complex manifolds as in
\cite[Definition 2.8]{k}; a \emph{flat family} $\pi:\mathcal{X}\rightarrow B$ refers to a proper surjective flat morphism with connected (but possibly reducible) fibers, where \lq \emph{flat}' means for every point $x\in \mathcal{X}$ the stalk $\mathcal{O}_{\mathcal{X},x}$ is a flat $\mathcal{O}_{B,\pi(x)}$-module via the natural map $\mathcal{O}_{B,\pi(x)}\rightarrow \mathcal{O}_{\mathcal{X},x}$. The flat family can be interpreted as the extension of the notion of holomorphic (or smooth) family of complex manifolds to the singular case.

Unless otherwise explicitly mentioned, an irreducible and reduced complex  analytic space is called a \emph{complex variety} and $B$ will always be taken as an open unit disk $\Delta$ in $\mathbb{C}$ so that the family $\pi:\mathcal{X}\rightarrow \Delta$ becomes flat automatically by criterion for flatness over a regular curve in \cite[Theorem 2.13 of Chapter V]{bs},   \cite[Paragraph 1 of Introduction]{Hi} or \cite[24.4.J]{va} for example.

\subsection{Summary of the main results}
By definition, a compact complex variety (or manifold) $M$ is called a \emph{Moishezon variety (or manifold)} if $a(M)=\dim_{\mathbb{C}} M$, where the {algebraic dimension} $a(M)$ is defined as the transcendence degree of the field of meromorphic functions on $M$. A general compact complex space is called \emph{Moishezon} if all its irreducible components are. Recall that the {$m$-genus} $P_m(X)$ of a compact complex manifold $X$ is defined by
$$P_m(X):=\dim_{\mathbb{C}} H^0(X,K_X^{\otimes m})$$
 with the canonical bundle $K_X$ of $X$ for every positive integer $m$.
 It is a bimeromorphic invariant and the \emph{$m$-genus of an arbitrary compact complex variety} can be defined as the $m$-genus of its arbitrary non-singular model.
Similarly, a compact complex variety $M$ is \emph{of general type} if the Kodaira--Iitaka dimension of the canonical bundle of its arbitrary non-singular model is just the complex dimension of $M$, while a compact complex space is {of general type} if any of its irreducible components is so. Notice that $\dim_{\mathbb{C}} H^0(X,{K_X^*}^{\otimes m})$ for a compact complex manifold $X$ is \emph{not} a bimeromorphic invariant in general.
In algebraic geometry, a normal complex variety $M$ is said to have
\emph{only canonical singularities} if the canonical divisor $K_M$
is a {$\mathbb{Q}$-Cartier divisor}, i.e., some integer multiple of
$K_M$ is a Cartier divisor, and if the discrepancy divisor
$K_{\tilde{M}}-\mu^*K_M$ is effective for a (or every) resolution of singularities
$\mu:\tilde{M}\rightarrow M$. However, in complex analytic geometry, we have to adopt the analytic definition of canonical singularities in \cite[Definition 1]{sv}, where
the canonical sheaf will replace the role of the canonical divisor.

Now, we can state our answer to Demailly's Question \ref{quest} in greater generality.
\begin{mthm}\label{main-thm}
Let $\pi:\mathcal{X}\rightarrow \Delta$ be a family such that any of the following holds:
\begin{enumerate}[(i)]
    \item \label{main-thm-i}
the family $\pi$ is holomorphic and each fiber $X_t$ for $t\in \Delta$ is a Moishezon manifold;
    \item \label{main-thm-ii}
each (not necessarily projective) compact variety $X_t$ at $t\in \Delta$ has only canonical singularities and uncountable fibers therein are of general type.
\end{enumerate}
Then for each positive integer $m$, the $m$-genus $P_m({X_t})$ is independent of $t\in \Delta$.
\end{mthm}
%
All these are
proved in Subsection \ref{pf-mt}. For the big $\mathbb{Q}$-Cartier anticanonical divisor case, see Remark \ref{anti} and Theorem \ref{singular} where both $\mathcal{X}$ and $\mathcal{X}\rightarrow \Delta$ need not be smooth, for the vanishing of all plurigenera. The latter (possibly degeneration cases) seem non-trivial due to K. Nishiguchi's counter-examples \cite{n83} to the lower semi-continuity of Kodaira dimension under degeneration of surfaces; the bigness assumptions in Theorem \ref{main-thm}.\eqref{main-thm-ii}, Remark \ref{anti} and Theorem \ref{singular} are thus essential. See Remarks \ref{3.19}-\ref{3.24} for more information.

Shortly after we posted our first version on arXiv, Koll\'{a}r also can prove in the recent \cite{k21a} that small deformations of a (single) projective variety of general
type with canonical singularities are also projective varieties of general type, with the same plurigenera. Inspired by these works, Koll\'{a}r also discusses a series of questions about Moishezon morphisms, and gives partial solutions to some of them in the more recent \cite{k21b}.

Recall that a smooth positive-definite $(1, 1)$-form $\omega$ on  a complex $n$-dimensional manifold $X$ is said to be a \emph{strongly
Gauduchon metric} if the $(n, n-1)$-form $\partial\omega^{n-1}$ is $\bar\partial$-exact on $X$. This notion was introduced by D. Popovici.
As a direct corollary of Theorem \ref{main-thm}.\eqref{main-thm-i} and Theorem \ref{thm-moishezon}.\eqref{thm-moishezon-i}, one has the following result:
\begin{cor}[{}]\label{main-thm-sm}
Let $\pi: \mathcal{X}\rightarrow \Delta$ be a holomorphic family of compact complex manifolds such that uncountable fibers
therein are Moishezon (or projective in particular) and all the fibers satisfy the local deformation invariance for Hodge number of type $(0,1)$ or admit strongly Gauduchon metrics. Then for each positive integer $m$, the $m$-genus $P_m({X_t})$ is independent of $t\in \Delta$. In particular, the direct image sheaf $\pi_*K_{\mathcal{X}}^{\otimes m}$ is locally free of rank $P_m({X_0})$.
 \end{cor}

As a more direct corollary of Theorem \ref{main-thm}.\eqref{main-thm-i}, we confirm the problems considered in \cite{as} and \cite{ls} more generally, and also \cite[Conjecture on p. 85]{Ue} partially which predicts deformation upper semi-continuity of Kodaira dimension of complex manifolds:
\begin{cor}
The Kodaira dimension of Moishezon and in particular projective  manifolds is invariant in a smooth family.
\end{cor}

Much inspired by Demailly's Question \ref{quest} and (the proof of) Theorem \ref{main-thm}, we also study the structure of the family with projective fibers in two Chow-type lemmata of Section \ref{mtii}. An application of them to the invariance of plurigenera and the structure of families is given in Remark \ref{3.21}; see also Subsection \ref{subs-Chow} for other applications.
\begin{lemma}[= Lemma \ref{sing-chow}]\label{0sing-chow}
If $\pi:\mathcal{X}\rightarrow \Delta$ is a one-parameter degeneration with projective fibers and uncountable fibers are all of general type or all have big anticanonical bundles, then $\pi$ is a \emph{pseudo-projective family}, i.e., there exists a projective family
$\pi_\mathcal{P}:\mathcal{P} \rightarrow \Delta$ from a complex manifold and a bimeromorphic morphism from $\mathcal{P}$ to $\mathcal{X}$ over $\Delta$ such that it induces a biholomorphism $\pi_\mathcal{P}^{-1}(U)\rightarrow \pi^{-1}(U)$ for some $U \subset \Delta$ with $\Delta\setminus U$ being a proper analytic subset of $\Delta$.
\end{lemma}

The smooth analogue is also studied in the other Lemma \ref{Chow-sm}:
Let $\pi:\mathcal{X}\rightarrow \Delta$ be a holomorphic family of projective manifolds. Then there exists a projective morphism $\pi_\mathcal{Y}:\mathcal{Y}\rightarrow \Delta$ from a complex manifold and a bimeromorphic morphism $f: \mathcal{Y}\rightarrow \mathcal{X}$ over $\Delta$ such that  possibly after shrinking $\Delta$, $f$ is a biholomorphism between $\mathcal{Y}\setminus{\pi_\mathcal{Y}^{-1}(0)}$ and $\mathcal{X}\setminus{\pi^{-1}(0)}$.
In particular, there exists a global line bundle $\mathcal{A}$ on $\mathcal{X}$ such that $\mathcal{A}|_{X_t}$ is very ample for any $t\ne 0$ possibly after shrinking $\Delta$.

Notice that the projectiveness assumption in the Chow-type Lemmata \ref{0sing-chow} and \ref{Chow-sm} is indispensable (but not in  Theorem \ref{main-thm}) since our  proofs of them rely on the very ampleness from some line bundles on the very generic fibers and also the applications of the local invariant cycle theorem \cite{cm}. Remark that certain partial results can be obtained by \cite[Theorem 1.6]{bin} which implies, among others, that there exists a {\it locally closed} analytic subset $U$  of $\Delta$ with $\Delta\setminus U$ at most a countable set of points, such that the subfamily $\mathcal{X}_U\rightarrow U$  is {\it locally projective}.
\subsection{Strategy of proofs for the main results}\label{intro-2}
Recall that a morphism between complex analytic spaces is called a \emph{Moishezon
 morphism}, if it is bimeromorphically equivalent to a  projective
morphism, and a \emph{Moishezon family} if further it is a family. Moishezon morphism here is often called \emph{algebraic morphism} in many literatures such as \cite{[Ta],[K1]} and we will use them interchangeably. Notice that each fiber of a Moishezon morphism is Moishezon, but the converse is \emph{not} necessarily true (cf. the elliptic fibration in \cite[Remark 2.25, Definition 3.5.(2)]{cp} as a global counterexample to this).
See the definition of projective morphism of the analytic category in \cite[\S\ 1.(d) of Chapter IV]{bs} and the family here is not a priori assumed projective.

With the above notions, we first sketch a general strategy for the deformation of plurigenera for the flat family of Moishezon varieties, which is divided into three steps, and then describe how to use it to prove Lemma \ref{0sing-chow}. However, the full strength of this strategy is not applied in the proof of Theorem \ref{main-thm} but in the Chow-type Lemma \ref{0sing-chow}.

The \emph{first} step is to follow the reduction in S. Takayama's proof of the lower semi-continuity in Theorem \ref{main2T}, where no canonical singularities assumption for the fibers is needed. Actually, we assume that the total space $\mathcal{X}$ is smooth upon replacing it by a proper modification, and that $\pi: \mathcal{X}\rightarrow \Delta$ is
smooth outside $0\in \Delta$
according to the generic smoothness \cite[Corollary 10.7 of Chapter III]{Ht} by replacing $\Delta$
by a smaller disk. By the semi-stable reduction theorem in
\cite[pp. 53-54]{[KKMSD73]}, one further assumes  that the reference fiber $X_0$ has only simple normal crossings on $\mathcal{X}$ all of whose components are reduced, and obtains a semi-stable degeneration.
\begin{thm}[{\cite[Theorem 1.2]{[Ta]}}]\label{main2T}
Let $\pi:\mathcal{X}\rightarrow \Delta$ be a Moishezon family, i.e., an algebraic morphism
and a family. Let $X_0$ be a reference fiber with support
$(X_0)_{red}=\sum_{i\in I} X_i$ and $X_t$ a general fiber.
Then $\sum_{i\in I} P_m(X_i)\leq P_m(X_t)$ holds for any positive
integer $m$.
\end{thm}

Then we try to construct a global line bundle on the total space of the semi-stable (or more generally one-parameter) degeneration such that its restrictions to the general fibers are big, which is our \emph{second} step. This is to be completed in Theorem \ref{0thm-moishezon-update}, while the key difficulty is to verify the \lq surjectivity' assumption \eqref{0surj-chern} therein.
\begin{thm}[{= Theorem \ref{thm-moishezon-update}}] \label{0thm-moishezon-update}
Let $\pi: \mathcal{X}\rightarrow \Delta$ be a one-parameter degeneration such that uncountable fibers $X_t$ therein admit big
line bundles $L_t$ such that
\begin{equation}\label{0surj-chern}
\text{the first Chern class $c_1(L_t)$ of $L_t$ comes from the restriction of $H^2(\mathcal{X}, \mathbb{Z})$ on $X_t$.}
\end{equation}
Assume further that all the fibers satisfy the local deformation invariance for Hodge number of type $(0,1)$. Then there exists a (global) line bundle $L$ over $\mathcal{X}$ such that for all $t\in\Delta^*:=\Delta\setminus \{0\}$,
$L|_{X_t}$ are big and thus $X_t$ are Moishezon.
\end{thm}

Finally, the \emph{third} step is to prove that the reduced semi-stable (or more generally one-parameter) degeneration (and thus the original family $\pi: \mathcal{X}\rightarrow \Delta$) is actually a Moishezon family by the global line bundle constructed in the second step and Theorem \ref{0thm-moishezon}.
Theorem \ref{0thm-moishezon} is a modified version of Theorem \ref{thm-moishezon}.\eqref{thm-moishezon-ii} in Appendix \ref{bim-em}. Though this result can be formulated for
certain flat family case (not necessarily smooth as a family or
as a total space) since it is quickly reduced to the special case
after a process of proper modification which is harmless because
of the desired bimeromorphic property, the following (special)
version is more convenient to us.
\begin{thm}[{}]\label{0thm-moishezon}
Let $\pi: \mathcal{X}\rightarrow \Delta$ be a one-parameter degeneration of Moishezon manifolds and assume that there exists a (global) line bundle $L$ over $\mathcal{X}$ such that for all $t\in \Delta^*$,
$L|_{X_t}$ are big. Then
for some $N\in \mathbb{N}$, there exists
a bimeromorphic map (over $\Delta$)
$$\Phi:\mathcal{X}\dashrightarrow\mathcal{Y}$$
to a subvariety $\mathcal{Y}$ of $\mathbb{P}^N\times\Delta$
with every fiber $Y_t\subset\mathbb{P}^N\times\{t\}$ being a projective variety.
\end{thm}

For the degeneration case, although each of the above Moishezon manifolds
admits a big line bundle depending on fibers (cf. Lemma \ref{sm-big-moi} below),
the above assumption on the existence of a global $L$
is crucial; see Remark \ref{3.19} and Example \ref{3.22} for counterexamples
and further information.  In contrast, for the smooth family case, the existence
of such  a global $L$ can be derived; see the proof of Theorem \ref{main-thm}.\eqref{main-thm-i} in Subsection \ref{pf-mt}. For certain
semi-stable reduction cases, the deduction of the existence of such a global $L$
is possible and is discussed in Section \ref{mtii}.

Hence, one is able to obtain the deformation invariance of plurigenera of $\pi$ by its Moishezonness property and Theorem \ref{0mainT}.
Moreover, the canonical singularities assumption for the fibers is (only) used in
Nakayama--Takayama's proof of upper semi-continuity in Theorem \ref{0mainT} for the flat family.
\begin{thm}[{\cite[Theorem 1.1]{[Ta]}}]\label{0mainT}
 Let $\pi:\mathcal{X}\rightarrow \Delta$ be a Moishezon family.
Assume that every fiber $X_t$ has only canonical
singularities. Then the $m$-genus $P_m(X_t)$ is independent of $t\in
\Delta$ for any positive integer $m$.
\end{thm}

Based on the above argument and \cite[Conjecture L]{n87} on the lower semi-continuity of plurigenera under degenerations, it is reasonable to propose, without the algebraic
morphism/Moishezon family assumption on $\mathcal{X}\rightarrow \Delta$,
the following:
\begin{con}[]\label{conj-uncountable}
Let $\pi:\mathcal{X}\rightarrow \Delta$ be a flat family such that all its fibers have only canonical singularities and are Moishezon varieties or even in Fujiki's class $\mathcal{C}$ (i.e., bimeromorphic to compact K\"ahler manifolds). Then for each positive integer $m$, the $m$-genus $P_m({X_t})$ is independent of $t\in \Delta$.
\end{con}
The above conclusion may fail for cases where the assumption
on the Fujiki's class $\mathcal{C}$ is not satisfied; see Example \ref{3.22} for more.
See also Conjecture \ref{di-conj} below for a smooth version.

Let us describe how to use the strategy of deformation invariance of plurigenera to prove Lemma \ref{0sing-chow}. Upon the semi-stable reduction  in its first step and by Proposition \ref{ki-dim-limit} and Lemma \ref{Prop. B}, one assumes that the family $\pi$ is a semi-stable degeneration of projective manifolds ${X_t}$ such that $K_{X_t}$ (resp. $K_{X_t}^*$) are big for all $t \in \Delta^*$. Then by Proposition \ref{fgn} and Remark \ref{biganti}, we construct the (global) line bundles with fiberwise (very) ampleness on $\mathcal{X}$ to induce the bimeromorphic map in Theorem \ref{0thm-moishezon}, the inverse of which are modified as the desired bimeromorphic morphisms for Lemma \ref{0sing-chow}. In particular, the local invariant cycle theorem \cite{cm} and effective algebro-geometric results and techniques under the projectivity and bigness assumptions are additionally applied to verify the \lq surjectivity' assumption \eqref{0surj-chern} in Theorem \ref{0thm-moishezon-update}.
\subsection{Background and related works} We present a brief, possibly incomplete history of deformation of plurigenera, and many parts here are taken from the nice historic statement in \cite{[S98]} and \cite[Introduction]{[Ta]}.

S. Iitaka \cite{ii}
proved the smooth surface case of the deformation invariance of the plurigenera. His method works only for surfaces because it uses much
of the information from the classification of surfaces. P. Wilson \cite{w78} obtained the lower semi-continuity of the plurigenera of projective surfaces under degeneration of surfaces having only isolated
Gorenstein singularities, while Y. Kawamata \cite{k80} proved the deformation invariance of the arithmetic genera $\dim_{\mathbb{C}} H^0(X_s,K_s)$ for the dualizing sheaf $K_s$ of the reduced algebraic surfaces $X_s$ with at most simple elliptic or simple quasi-elliptic singularities. See more on the deformation behavior of arithmetic and geometric plurigenera and of Kodaira dimension on surfaces in \cite{w79}.  M. Levine \cite{Le1,Le2}
proved that for every positive integer $m$ every element of $H^0(X_0,K_{X_0}^{\otimes m})$
can be extended to the fiber of $\mathcal{X}$ over the \emph{double point} of $\Delta$  at $t=0$. So
far there is no way to continue the process to get an extension to the
fiber of $\mathcal{X}$  at $t=0$ of any finite order. N. Nakayama \cite{n86} pointed out that if the relative case of the minimal model
program can be carried out for a certain dimension, the conjecture of
the deformation invariance of the plurigenera for that dimension would follow
directly from it. Thus the deformation invariance of the plurigenera for threefolds is
a consequence of the relative case of the minimal
model program (for threefolds) by Kawamata \cite{k92} and Koll\'{a}r--S. Mori \cite{km92}.

Y.-T. Siu \cite{[S98],[S00]} proved the deformation invariance of plurigenera in case  $\pi:\mathcal{X}\rightarrow \Delta$  is a smooth projective family of complex manifolds, and in \cite{s04}
 Siu also considered, with deep insights, the case when the reference fiber is singular
and even non-reduced. See also the closely related work by H. Tsuji \cite{[Ts]}.  The works of Kawamata
\cite{[K2],[K1]}, Nakayama \cite[Chapter VI]{n04} and Takayama \cite{[Ta]} are strongly influenced by Siu.
Kawamata \cite[Theorem 6]{[K2]} obtained Theorem \ref{0mainT} with all fibers being of general type. Nakayama \cite[Chapter VI]{n04}  obtained Theorem \ref{0mainT} by assuming the so-called abundance conjecture is true
for the generic fiber, and Theorem \ref{main2T} when the abundance conjecture is
true for the generic fiber and a somewhat artificial assumption holds on a reference fiber. Nakayama \cite[Theorem 11]{n86}
 has also shown that Theorem \ref{main2T} follows from the existence of
a minimal model for the family and the abundance conjecture for the generic
fiber. Takayama \cite{[Ta]} gave a direct
proof of the above mentioned results of Nakayama without assuming the existence
of minimal models or abundance, and completed the
picture by Kawamata and Nakayama. His approach is
analytic, relying crucially on the uniform estimates provided by the
Ohsawa--Takegoshi extension theorem \cite{[OT]}. Of course, complete algebraic proofs
are certainly desirable in the case of varieties of nonnegative Kodaira dimension which are not of general type. More recently, M. P$\breve{\textrm{a}}$un \cite{[P]} discovered a very short and elegant proof of Siu's theorem on the deformation invariance of plurigenera without use of the effective global generation (see also \cite[Chapter 16]{dem} for a nice interpretation and \cite{cl}). For more recent progress, we refer the reader to \cite{bp,Hmx,cf,br,fs} as a possibly incomplete list.

 Compared with Theorem \ref{0mainT}, our main Theorem \ref{main-thm} seems to  admit more essential deformation significance since the local stability type of results often relies more on the assumption on the reference fiber rather than on the total space, such as Kodaira--Spencer's local stability theorem of K\"ahler structure \cite{KS}, its generalizations to $p$-K\"ahler structures \cite{RwZ,rwz18}, and also deformation invariance of Hodge numbers and particularly $1$-genus \cite{RZ15}, etc.
It is worth calling the reader's attention to the generalizations of Theorem
\ref{main-thm}. A classical example of Wilson \cite[Example 4.1, Remark 4.2]{w81} exhibits a family
of normal Gorenstein $3$-folds $\{V_t\}_{t\in \Delta}$, in which almost all members constructed in \cite[Example 4.1]{w81} do not have canonical singularities, and the $m$-genus
$\dim_{\mathbb{C}} H^0(V_t,mK_{V_t})$ is not constant. Kawamata \cite[Example 4.3]{[K1]} gave an example
to show that Theorem \ref{main-thm} does not hold in the larger class of
Kawamata log-terminal singularities.

In contrast, I. Nakamura \cite{N} provided a counterexample for the generalization
of the conjecture in the case where the reference fiber $X_0$ is a quotient of a $3$-dimensional simply-connected
solvable Lie group by a discrete subgroup. We note that $X_0$ is not Moishezon or even $\partial\bar\partial$-manifold
which has non-closed holomorphic $1$-forms. So we only consider the family of Moishezon varieties in this paper.

Nowadays, the following conjecture due to Siu and Nakayama is still a big challenge, cf. \cite[Conjecture 0.4]{[S00]}, \cite[Conjecture 2.1]{[S02]}, \cite[Conjecture 1.1 of Chapter VI]{n04}:
\begin{con}[{}]\label{di-conj}
Let $\pi: \mathcal{X}\rightarrow \Delta$ be a holomorphic family of compact complex manifolds such that each fiber $X_t$ is K\"ahler or even in Fujiki's class $\mathcal{C}$ (i.e., bimeromorphic to a compact K\"ahler manifold) for any $t\in \Delta$. Then for each positive integer $m$, the $m$-genus $\dim_{\mathbb{C}} H^0(X_t, K_{X_t}^{\otimes m})$ is independent of $t\in \Delta$.
\end{con}
Remark that an approach to the above conjecture in the K\"ahlerian case has recently been proposed by Demailly \cite{Dem20} and a confirmation of an important special case in an analytic way is given by J. Cao--P$\breve{\textrm{a}}$un \cite{cp20}.

\noindent
\textbf{Acknowledgement}:
The authors  would like to express their gratitude to Professor J.-P. Demailly for pointing out Question \ref{quest} to them when the first named author visited him  in May 2017 and many useful suggestions and discussions on this paper (in particular Theorem \ref{singular}), and Professors Fusheng Deng, Ya Deng, Zhengyu Hu, Junchao Shentu, Xun Yu, Dr. Xiaojun Wu, Jian Chen for many useful discussions. We also thank Professors D. Abramovich for useful advice on Lemma \ref{bu-proj}, C. Peters for explaining the local invariant cycle Theorem \ref{lict} in \cite{ps}, S. J. Kov\'{a}cs, Linquan Ma, K. Schwede on singularities, Chenyang Xu, Lei Zhang for Remark \ref{xu}, F. Campana and C. Voisin on projective/Moishezon morphisms via emails and also Junyan Cao, M. P$\breve{\textrm{a}}$un, S. Takayama for their interest. We are grateful to J. Koll\'{a}r for sending us his interesting preprint \cite{k21b} on Moishezon morphisms.
Part of this work was done during the first named author's visit to Institut Fourier, Universit\'{e} Grenoble Alpes from March 2019.
He would like to express his gratitude to the institute for their hospitality and the wonderful work environment during his visit.

\section{Preliminaries: Moishezon varieties and sheaves}
We collect some basic knowledge and notations, to be used often in the later
sections.
\subsection{Moishezon varieties}
We are mostly concerned with deformations of Moishezon varieties.
Several nice references on Moishezon varieties (or manifolds) are \cite[\S\ 3]{Ue}, \cite[\S\ 6]{pe} and \cite[Chapter 2]{mm}.

To fix the notation, let $M$ be a compact (irreducible) complex variety of dimension $n$ and $L$ a holomorphic line bundle over $M$.
Set the linear system associated to $\mathcal{L}:=H^{0}(M,L)$ as $|\mathcal{L}|=\{\Div(s): s\in \mathcal{L}\}$ and its {base point locus} $\Bl_{|\mathcal{L}|}$.
Now consider $\mathcal{L}_p:=H^{0}(M,L^{\otimes p})$ for a positive integer $p$ and the Kodaira map  $\Phi_p:=\Phi_{\mathcal{L}_p}$ associated to $L^{\otimes p}$.
Set
$$
\varrho_p=
\begin{cases}
\max\{\rk\ \Phi_p: x\in M\setminus \Bl_{|\mathcal{L}_p|}\}, &\text{if $\mathcal{L}_p\neq\{0\}$},\\
-\infty, &\text{otherwise}.
\end{cases}
$$
The \emph{Kodaira--Iitaka dimension} of $L$ is
$\max\{\varrho_p:p\in \mathbb{N}^+\}$ and denoted by $\kappa(L)$. More generally, the \emph{Kodaira--Iitaka dimension} of a line bundle $L$ over a compact complex space $X$ is the minimum of the {Kodaira--Iitaka dimensions} of the restrictions of $L$ to the irreducible components of $X$.

Since we often need to keep track of the $p$ in $\Phi_p$, the following criterion, involving a positive integer $d$, is recorded here.
\begin{thm}[{\cite[Theorem $8.1$]{Ue}}]\label{asymp}
For a  line bundle $L$ on a compact complex variety $M$, there exist positive numbers $\alpha,\beta$ and a positive integer $m_0$ such that for any integer $m\geq m_0$, the inequalities hold
\begin{equation}\label{asymp-ineq}
 \alpha m^{\kappa(L)}\leq h^0(M, L^{\otimes md})=\dim_{\mathbb{C}}H^0(M, L^{\otimes md})\leq \beta m^{\kappa(L)},
\end{equation}
where $d$ is some positive integer depending on $L$. In particular, when $h^0(M, L)\neq 0$, one can take $d=1$ in \eqref{asymp-ineq}.
\end{thm}

Recall that a line bundle on a compact complex variety $M$ is called \emph{big} if $\kappa(L)=\dim_{\mathbb{C}} M$. Then one has the characterization of a Moishezon manifold or variety.
\begin{lemma}[{\cite[Theorem 2.2.15]{mm}, \cite[Proposition 6.16]{pe}}]\label{sm-big-moi}
A compact complex manifold is Moishezon if and only if it admits a big holomorphic line bundle.
More generally, the same result still holds for
a compact complex variety.
\end{lemma}

\subsection{Some basics on sheaves}
 We have three basic properties on direct image sheaves:
\begin{lemma}[{cf. \cite[Theorem 6.2 of Chapter III, p. 176]{Iv}}]\label{2-5}
Let $f:X\rightarrow Y$ be a proper map between locally compact spaces and $\mathcal{F}$ a sheaf of abelian groups on $X$. For any point $y\in Y$ and for all $q$,
$$
(R^qf_*\mathcal{F})_y\simeq H^q(f^{-1}(y),\mathcal{F}).
$$
\end{lemma}
We can use the Leray spectral sequence to obtain an isomorphism
\begin{equation}\label{identif}\Gamma(\Delta, R^i\pi_*\mathcal{O}_{\mathcal{X}})\cong H^i(\mathcal{X},\mathcal{O}_{\mathcal{X}}),\qquad \text{ for any $i\ge 0$}.\end{equation}
\begin{thm}[]\label{leray}
Let $X,Y$ be topological spaces, $f:X\rightarrow Y$ a continuous map and $\mathcal{S}$ a sheaf of abelian groups on $X$. Then there exists the \emph{Leray spectral sequence} $(E_r)$ such that
\begin{enumerate}[$(i)$]
    \item \label{}
$E_2^{p,q}\simeq H^p(Y,R^qf_*\mathcal{S})$;
    \item \label{}
$(E_r)\Rightarrow H^*(X,\mathcal{S})$.
\end{enumerate}
So $H^k(X,\mathcal{S})=\bigoplus_{p+q=k}E_\infty^{p,q}.$
In particular, if $H^p(Y, R^qf_*\mathcal{S})=0$ for $p > 0$, there is an isomorphism
$$H^0(Y,R^qf_*\mathcal{S})\cong H^q(X,\mathcal{S}).$$
\end{thm}
\begin{proof}
This just follows from \cite[(13.8) Theorem and (10.12) Special case of Chapter IV]{Dem12}.
\end{proof}

What follows is easily available in the algebraic
category.  But since we will mostly be in the analytic category, for reader's convenience
we give some details with precise references.
\begin{thm}[{Grauert's direct image theorem \cite{Gt} or \cite[\S\ 2 of Chapter III]{bs}}]\label{gdit}
Let $f:X\rightarrow Y$ be a proper morphism of complex spaces and $\mathcal{F}$ a coherent analytic sheaf on $X$. Then for all $q\geq 0$,
the analytic sheaves $R^qf_*\mathcal{F}$ are coherent.
\end{thm}

An $\mathcal{O}_X$-sheaf $\mathcal{F}$ on a complex space $X$ is called \emph{locally free} at $x\in X$ of rank $p\geq 1$ if there is a neighborhood $U$ of $x$ such that $\mathcal{F}(U)\cong \mathcal{O}_U^p$. Such sheaves are coherent. Using Oka's theorem we get a converse: If a coherent sheaf $\mathcal{F}$ is \emph{free at $x\in X$}, i.e., if the stalk $\mathcal{F}_x$ is isomorphic to $\mathcal{O}_x^p$, then $\mathcal{F}$ is locally free at $x$ of rank $p$. In particular, the set of all points where $\mathcal{F}$ is free is open in $X$.

 For a closer study, one introduces the rank function of an $\mathcal{O}_X$-coherent sheaf $\mathcal{F}$. All $\mathbb{C}$-vector space $\bar{\mathcal{F}}_x:=\mathcal{F}_x/{\mathfrak{m}}_x\mathcal{F}_x$, $x\in X$, are of finite dimension. Here $\mathfrak{m}_x$ is the maximal ideal of $\mathcal{O}_{X,x}$.  The integer $\rk\ \mathcal{F}_x:=\dim_{\mathbb{C}} \bar{\mathcal{F}}_x\in \mathbb{N}$ is called the \emph{rank} of $\mathcal{F}$ at $x$; clearly, $\rk\ \mathcal{O}_x^p=p$. The rank function of a locally free sheaf is locally constant on a complex space $X$. Conversely, if $X$ is reduced and a sheaf  $\mathcal{F}$ is $\mathcal{O}_X$-coherent such that $\rk\ \mathcal{F}_x$ is locally constant on $X$, then $\mathcal{F}$ is a locally free sheaf on $X$.
 The set $S(\mathcal{F})$ of all points in $X$ where a coherent sheaf $\mathcal{F}$ is not free is called the \emph{singular locus} of $\mathcal{F}$. Then:
 \begin{prop}[{\cite[Proposition 7.17]{Rem}}]\label{slia}
 The singular locus $S(\mathcal{S})$, as defined precedingly, of any given $\mathcal{O}_X$-coherent sheaf $\mathcal{S}$ on a complex space $X$ is analytic in $X$. If $X$ is reduced, this set is thin in $X$.
 \end{prop}

Moreover, one has the remarkable Grauert's semi-continuity theorem:
\begin{thm}[{\cite[Theorem 4.12.(i) of Chapter III]{bs}}]\label{Upper semi-continuity}
Let $f:X\rightarrow Y$ be a proper morphism of complex spaces and $\mathcal{F}$ a coherent analytic sheaf on $X$ which is \emph{flat with respect to $Y$ (or $f$)}, which means that the $\mathcal{O}_{f(x)}$-modules $\mathcal{F}_x$ are flat for all $x\in X$. Set $\mathcal{F}(y)$ as the analytic inverse image with respect to the embedding $X_y$ of in $X$.  Then for any integers $i,d\geq 0$, the set
$$\{y\in Y: h^i(X_y,\mathcal{F}(y))\geq d\},\quad$$
is an analytic subset of $Y$.
\end{thm}

The topology in $Y$ whose closed sets are all analytic sets is called the \emph{analytic Zariski topology}. The statement of Theorem \ref{Upper semi-continuity} means an upper semi-continuity of $h^i(X_y,\mathcal{F}(y))$ with respect to this analytic Zariski topology.

\section{Deformation invariance of plurigenera}\label{mf}
We will formulate the techniques for the strategy of Subsection \ref{intro-2} in the first three subsections and then give the proof of Theorem \ref{main-thm} in Subsection \ref{pf-mt}. Remark that the full strength of these techniques is not applied in the proof of Theorem \ref{main-thm}, but in the Chow-type Lemma \ref{sing-chow}.
\subsection{Deformation invariance of $(0,q)$-type Hodge numbers}\label{di0q}
Let $\mathcal{X}$ be a complex manifold, $\Delta\subseteq \mathbb{C}$ the unit disk and $f:\mathcal{X}\rightarrow \Delta$ a flat family, smooth over the punctured disk $\Delta^*$. We say that $f$ is a \emph{one-parameter degeneration}. In general $X_0:=f^{-1}(0)$ can have arbitrarily bad singularities, but after a base change and suitable blow-ups, $X_0$ can be assumed to have only simple normal crossings on $\mathcal{X}$.
\begin{prop}[{\cite[Theorem 11.11]{ps}}]\label{reduction}
Let $f:\mathcal{X}\rightarrow \Delta$ be as above. Then there exists $m\in \mathbb{N}$ such that for the $m$-th root $f': \mathcal{X}'\rightarrow  \Delta'$ of $f$ the reference fiber has simple normal crossing divisor all of whose components are reduced.
\end{prop}
In fact,  let $\mu$ be the least common multiple of the multiplicities of the components of the divisor $X_0$ and consider the map $m: t\mapsto t^{\mu}$ sending $\Delta$ to itself. For the moment, let us denote by $\Delta'$ the source of the map and let $W$ be the normalization of the fiber product $\mathcal{X}\times_{\Delta} \Delta'$. In general, $W$ is a $V$-manifold. Blowing up the singularities we obtain a manifold $\mathcal{X}'$ and a morphism $f': \mathcal{X}'\rightarrow  \Delta'$. We call $f': \mathcal{X}'\rightarrow  \Delta'$ the \emph{$\mu$-th root fibration} of $f$. The fiber $X_0'$ has simple normal crossings, but unless $\dim_{\mathbb{C}} \mathcal{X}=3$, the components introduced in the last blow-up might not be reduced. The  semi-stable reduction theorem \cite[pp. 53-54]{[KKMSD73]} states that repeating the above procedure a finite number of times one can achieve this.
\begin{prop}[{}]\label{01-invariance}
Let $f:\mathcal{X}\rightarrow \Delta$ be a one-parameter degeneration. Suppose that $X_0$ is a reduced divisor with simple normal crossings on $\mathcal{X}$ and the irreducible components of $X_0$ are Moishezon, and that for $t\in \Delta^*$, each $X_t$ is also Moishezon.
Then for any nonnegative integers $p,q$,
$$h^{p,q}(t):=\dim_{\mathbb{C}} H^q(X_t, \Omega_{\mathcal{X}/\Delta}^p(\log X_0)\otimes \mathcal{O}_{X_t})$$ is independent of $t\in \Delta$.
In particular, $h^{0,q}(t)=\dim_{\mathbb{C}} H^q(X_t, \mathcal{O}_{X_t})$ is independent of $t\in \Delta$.
\end{prop}
\begin{proof} One owes this proposition to J. Steenbrink \cite{st76}, cf. also \cite[Corollaries 11.23, 11.24]{ps} in the K\"ahler case and \cite[Theorem 3.13]{n87} in the Fujiki class $\mathcal{C}$ case. For reader's convenience, we present a proof along the line of thought in \cite[$\S$ 11.2.7]{ps}.

The deformation invariance of Hodge numbers for the Moishezon manifolds is a standard result since a Moishezon manifold satisfies the degeneration of Fr\"olicher spectral sequence at $E_1$ or even the $\partial\bar\partial$-lemma and then we can apply \cite[Proposition 9.20]{[V]} or \cite[Theorem 1.3 or 3.1]{RZ15}.

Now the desired result is contained in the proof of the claim:
If $\varepsilon>0$ is sufficiently small, then for all $t\in \Delta$ with $|t|<\varepsilon$, the Hodge spectral sequence
$$E^{p,q}_1=H^q(X_t, \Omega_{\mathcal{X}/\Delta}^p(\log X_0)\otimes \mathcal{O}_{X_t})\Longrightarrow H^{p+q}(X_t,\mathbb{C})$$
degenerates at $E_1$. In fact, take $\varepsilon>0$ so small that $h^{p,q}(t)\leq h^{p,q}(0)$ for all $t\in \Delta^*$ with $|t|<\varepsilon$ and all $p,q\leq 0$. For $t\neq 0$, one has $\Omega_{\mathcal{X}/\Delta}^p(\log X_0)\otimes \mathcal{O}_{X_t}\cong \Omega_{X_t}^p$ and thus by the Hodge spectral sequence we have
$$\sum_{p,q}h^{p,q}(t)\geq \sum_k \dim H^k(X_t,\mathbb{C}),\quad t\neq 0,$$
with equality if and only if the Hodge spectral sequence for $X_t$ degenerates at $E_1$. One has
$$\dim H^k(X_\infty,\mathbb{C})=\sum_{p,q}h^{p,q}(0)\geq \sum_{p,q}h^{p,q}(t)\geq \sum_k \dim H^k(X_t,\mathbb{C})$$
and so equality must hold everywhere as $H^k(X_\infty,\mathbb{C})=H^k(X_t,\mathbb{C})$ and the first equality above follows from \cite[Corollary 11.23]{ps}. Here we write $H^k(X_\infty,\mathbb{C})$ for the mixed Hodge structure which the following complex puts on the hypercohomology group $\mathbb{H}^k(X_0, \psi_f\underline{\mathbb{Z}}_\mathcal{X})$:
$$\psi_f^{Hdg}:=(\psi_f\underline{\mathbb{Z}}_\mathcal{X}, (s(C^{\bullet,\bullet}),W(M)),(s(A^{\bullet,\bullet}),W(M),F))$$
with the data in \cite[(XI-22), (XI-23) and (XI-28)]{ps}.
\end{proof}

\begin{rem}\label{dubios}
The deformation invariance of Hodge numbers $h^{0,q}$ in Proposition \ref{01-invariance} holds more generally as long as the fibers in a flat family with  projective fibers have at most Du Bois singularities, cf. \cite[Example 2.6]{kv} and \cite[4.6. Th\'{e}or\`{e}me]{[DB81]}.
\end{rem}

\subsection{Deformation density of Kodaira--Iitaka dimension}\label{ddkid}
We first describe the deformation behavior of Kodaira--Iitaka dimension under the flat deformation.
\begin{prop}\label{ki-dim-limit}
Let $\pi: \mathcal{X}\rightarrow Y$ be a flat family from a complex manifold over a one-dimensional connected complex manifold $Y$ with possibly reducible fibers. If there exists a holomorphic line bundle $L$ on $\mathcal{X}$ such that the Kodaira--Iitaka dimension $\kappa(L_t)=\kappa$ for each $t$ in an uncountable set $B$ of $ Y$, then any fiber $X_t$ in $\pi$ has at least one irreducible component $C_t$ with  $\kappa(L|_{C_t})\geq \kappa$. In particular, if any fiber $X_t$ for $t\in Y$ is irreducible, then $\kappa(L_{t})\geq \kappa$. Here $L_{t}:=L|_{X_t}$ and similarly for others.
\end{prop}

\begin{proof} Compare \cite[Proposition 4.1]{rt} and we include a proof due to its difference and importance.

  The case $\kappa=-\infty$ is trivial and so one assumes that $\kappa\geq 0$.
For any two positive integers $p,q$, set
$$\label{tpql}
T_{p,q}(L)=\{t\in  Y: h^0(X_t,L_t^{\otimes p})\geq \frac{1}{q}p^\kappa\}.
$$
By the upper semi-continuity Theorem \ref{Upper semi-continuity}, it is known that $T_{p,q}(L)$ is a proper analytic subset of $Y$ or equal to $Y$. From the assumption that $\kappa(L_t)=\kappa$ for any $t$ in an uncountable subset of $B$ (still denoted by $B$) and Theorem \ref{asymp}, it holds that
\begin{equation}\label{utpq}
  \bigcup_{p,q\in \mathbb{N}^+}T_{p,q}(L)\supseteq B
\end{equation}
despite that the $d$ in Theorem \ref{asymp} is not necessarily one here.
Then since a proper analytic subset of a one-dimensional manifold is at most countable, there exists some analytic subset in the union, denoted by $T_{\tilde{p},\tilde{q}}(L)$, such that $T_{\tilde{p},\tilde{q}}(L)\supseteq B$ and so $T_{\tilde{p},\tilde{q}}(L)= Y$. That is, for any $t\in  Y$, it holds
$$h^0(X_t,L_t^{\otimes \tilde{p}})\geq \frac{1}{\tilde{q}}\tilde{p}^\kappa.$$

Now take $\tilde{L}=L^{\otimes \tilde{p}}$ and thus for any $t\in Y$,
$$h^0(X_t,\tilde{L}_t)\neq 0.$$ For any two positive integers $p,q$, we write
$$S_{p,q}=T_{p,q}(\tilde{L})\cap T_{p+1,q}(\tilde{L})\cap T_{p+2,q}(\tilde{L})\cap\cdots,$$
where
$$\label{tpqh}
T_{p,q}(\tilde{L}):=\{t\in  Y: h^0(X_t,\tilde{L}_t^{\otimes p})\geq \frac{1}{q}p^\kappa\}
$$
and similarly for $T_{p+1, q}(\tilde{L}), T_{p+2, q}(\tilde{L}),\ldots$.  By the upper semi-continuity again, $S_{p,q}$ is a proper analytic subset of $ Y$ or equal to $ Y$ (\cite[(5.5) Theorem of Chapter II]{Dem12}). From the assumption that $\kappa(\tilde{L}_t)=\kappa(L_t)=\kappa$ for each $t\in B$ and Theorem \ref{asymp} with $d=1$, one sees that
$$\bigcup_{p,q\in \mathbb{N}^+}S_{p,q}\supseteq B.$$
By the argument similar to \eqref{utpq}, there exists some $S_{i,j}=Y$ such that for any $t\in  Y$, it holds that for all $m\geq i$,
\begin{equation}\label{16-3}
h^0(X_t,\tilde{L}_t^{\otimes m})\geq \frac{1}{j}m^\kappa.
\end{equation}
This proves the proposition by Theorem \ref{asymp}.
\end{proof}
\begin{rem}
It is impossible to claim that all the irreducible components $C_t$ of any fiber $X_t$ in $\pi$ satisfy $\kappa(L|_{C_t})\geq \kappa$, thanks to \cite[Example 1]{n83}. In that example, Nishiguchi studied a family of $K3$ surfaces (with Kodaira dimension $0$) which degenerates to a surface (with Kodaira dimension $1$) and a Hopf surface (with Kodaira dimension $-\infty$).
\end{rem}

\begin{rem} Our inequality \eqref{16-3} proves more than what is asserted in the proposition. Indeed, \eqref{16-3} plays crucial roles
in many places of \cite{rt}. Moreover, Proposition \ref{ki-dim-limit} in  the irreducible fibers case is equivalent to Lieberman--Sernesi's main result \cite[Theorem on p. $77$]{ls} and \cite[Lemma in Section 5]{ti} in some sense.
See \cite[Remark 4.2]{rt} for more details.
\end{rem}

As a direct application of Proposition \ref{ki-dim-limit} and Lemma \ref{sm-big-moi}, one has the extension property of bigness:
\begin{cor}\label{unc-big}
With the notations of Proposition \ref{ki-dim-limit},  assume that there exists a holomorphic line bundle $L$ on $\mathcal{X}$ such that for each $t$ in an uncountable set of $ Y$, $L|_{X_t}$ is big. Then
any fiber $X_t$ in $\pi$ has at least one irreducible Moishezon component $\Sigma_t$ with $L|_{\Sigma_t}$ big.
\end{cor}

\subsection{Existence of line bundle over total space: Hodge number} \label{ext-lb-h}
The goal of this subsection is to prove:
\begin{thm}\label{thm-moishezon-update}
Let $\pi: \mathcal{X}\rightarrow \Delta$ be a one-parameter degeneration  such that uncountable fibers $X_t$ therein admit big
line bundles $L_t$ with the property that
\begin{equation}\label{surj-chern}
\text{$c_1(L_t)$ comes from the restriction of some cohomology class in $H^2(\mathcal{X}, \mathbb{Z})$ to $X_t$.}
\end{equation}
Assume further that all the fibers satisfy the local deformation invariance for Hodge number of type $(0,1)$. Then there exists a (global) line bundle $L$ over $\mathcal{X}$ such that for all $t\in \Delta^*$,
$L|_{X_t}$ are big and thus $X_t$ are Moishezon.
\end{thm}

First, as background material we introduce the torsion sheaf as in \cite[\S\ 7.5]{Rem}.
Let $X$ be a reduced complex space. For every $\mathcal{O}_X$-module $\mathcal{F}$ one defines
$T(\mathcal{F}):=\bigcup_{x\in X}T(\mathcal{F}_x),$
where $T(\mathcal{F}_x):=\{s_x\in \mathcal{F}_x: g_xs_x=0,\ \text{for a suitable $g_x\in \mathcal{A}_x$}\}$
with the multiplicative stalk $\mathcal{A}_x$ of the subsheaf $\mathcal{A}$ of all non-zero divisors in $\mathcal{O}_X$.
For our cases below we have $\mathcal{A}=\mathcal{O}_X$.  We obtain an  $\mathcal{O}_X$-module in $\mathcal{F}$; obviously $T(\mathcal{F})_x=T(\mathcal{F}_x)$. $T(\mathcal{F})$ is called the \emph{torsion sheaf} of $\mathcal{F}$. We call $\mathcal{F}$ \emph{torsion free at $x$} if $T(\mathcal{F})_x=0$. The sheaf $\mathcal{F}/T(\mathcal{F})$ is torsion free everywhere. Subsheaves of locally free sheaves are torsion free.

 \begin{prop}[{\cite[Proposition 4.5]{rt}}]\label{s-torsion}
For a proper morphism $\pi: \mathcal{X}\rightarrow Y$ from complex spaces to a connected one-dimensional manifold with $0\in Y$,
suppose that $R^2\pi_*\mathcal{O}_{\mathcal{X}}$ is locally free on $Y^*:=Y\setminus \{0\}$.  Let $s\in \Gamma( Y, R^2\pi_*\mathcal{O}_{\mathcal{X}})$.  Then $s|_{ Y^*}=0$ is equivalent to the germ $s_0\in T(R^2\pi_*\mathcal{O}_{\mathcal{X}})_0$.
As a consequence, if $T(R^2\pi_*\mathcal{O}_{\mathcal{X}})_0=0$ and $s|_{ Y^*}=0$, then $s=0$ in $\Gamma( Y, R^2\pi_*\mathcal{O}_{\mathcal{X}})$.
\end{prop}

In this part one always considers the proper morphism $f:X\rightarrow Y$ of complex spaces and a coherent analytic sheaf $\mathcal{F}$ on $X$, \emph{flat with respect to $f$ (or $Y$)}, which means that the $\mathcal{O}_{f(x)}$-modules $\mathcal{F}_x$ are flat for all $x\in X$.
One says that \emph{$\mathcal{F}$ is cohomologically flat in dimension $q$ at a  point $y\in Y$} if for an arbitrary $\mathcal{O}_y$-module $M$ of finite type and an integer $q$, the functor $M\mapsto R^qf_*(\mathcal{F}\otimes_{\mathcal{O}_Y} M)_y$ is exact;
 $\mathcal{F}$ is \emph{cohomologically flat in dimension $q$ over $Y$} if $\mathcal{F}$ is cohomologically flat  in dimension $q$ at any point of $Y$.
\begin{lemma}[Grauert's continuity theorem, {\cite[Theorem 4.12.(ii) of Chapter III]{bs}}]\label{gct}
If $\mathcal{F}$ is cohomologically flat in dimension $q$ over $Y$, then the function
$$y\mapsto \dim_{\mathbb{C}} H^q(X_y, \mathcal{F}(y))$$
is locally constant. Conversely, if this function is locally constant and $Y$ is a reduced space, then $\mathcal{F}$ is cohomologically flat in dimension $q$ over $Y$; in particular, the sheaf $R^qf_*\mathcal{F}$ is locally free.
\end{lemma}
\begin{lemma}[{\cite[Theorem (8.5).(iv) of Chapter I]{BHPV}}]\label{ccs-bc}
Let $X$, $Y$ be reduced complex spaces and $f : X\rightarrow Y$ a
proper holomorphic map. If $\mathcal{F}$ is any coherent sheaf on $X$, which is flat with respect to $f$, and $h^q(X_y,\mathcal{F}(y))$
is constant in $y\in Y$, then the base-change map  $$(R^qf_*\mathcal{F})_y/{\mathfrak{m}_y(R^qf_*\mathcal{F})_y}\rightarrow R^qf_*(\mathcal{F}/\hat{\mathfrak{m}}_y\mathcal{F})_y$$ is bijective.
\end{lemma}

The following will be used in the proof of Proposition \ref{global-B}.
\begin{prop}[{\cite[Proposition 4.12]{rt}}]\label{Hodge-torsionfree} For a flat family $\pi: \mathcal{X}\rightarrow Y$ over a connected manifold of dimension one,
$h^1(X_y, \mathcal{O}_{X_y})$ is independent of $y\in Y$ if and only if the sheaf $R^2\pi_*\mathcal{O}_{\mathcal{X}}$ is torsion free.
\end{prop}
Remark that a detailed proof can be found in \cite{rt} and it still works in the flat family case here, although this type of results should be known to experts, such as \cite{s04} in a different context, and the \lq only if' part for the smooth family is a special case of \cite[Proposition in $\S\ 5.5$ of Chapter 10]{Grr}.
\begin{cor}\label{h0201}
For a flat family $\pi: \mathcal{X}\rightarrow Y$ over a connected manifold of dimension one, if
$h^2(X_y, \mathcal{O}_{X_y})$ is independent of $y\in Y$, then so is $h^1(X_y, \mathcal{O}_{X_y})$.
\end{cor}
\begin{proof}
This is a direct corollary of the converse part of Lemma \ref{gct} and Proposition \ref{Hodge-torsionfree}, while locally free sheaves are torsion free. The smooth family case just follows from Kodaira-Spencer's squeeze \cite[Theorem 13]{KS}.
\end{proof}
To prove Proposition \ref{global-B}, one needs:
\begin{lemma}[]\label{section-zero}
  Let $E$ be a holomorphic vector bundle over a one-dimensional
(connected) complex manifold $M$ and $s$ a holomorphic section of $E$ on $M$.  If $s$ has uncountably
many zeros, then $s$ is identically zero.
\end{lemma}

The following is a variant of Theorem \ref{thm-moishezon-update}, in which we drop the
big line bundles assumption and conclusion.
\begin{prop}[{}]\label{global-B}
Let $\pi: \mathcal{X}\rightarrow \Delta$ be a one-parameter degeneration.
 Assume that there exists an uncountable subset $B$ of $\Delta$ such that for each $t\in B$, the fiber $X_t$ admits a
line bundle $L_t$ with the property that
\begin{equation}\label{surj-chern'}
\text{$c_1(L_t)$ comes from the restriction to $X_t$ of some cohomology class in $H^2(\mathcal{X}, \mathbb{Z})$.}
\end{equation}
Assume further that the Hodge number $h^{0,1}(X_t):=h^1(X_t, \mathcal{O}_{X_t})$ is independent of $t\in \Delta$.
Then there exists a (global) line bundle $L$ over $\mathcal{X}$ such that
$c_1(L|_{X_s}) = c_1(L_s)$
for any $s$ in some uncountable subset of $B$.
\end{prop}
\begin{proof} Compare this proposition with \cite[Proposition 4.14]{rt}.
By the assumption \eqref{surj-chern'}, there is an uncountable subset $S\subseteq B\setminus \{0\}$ such that for any $s\in S$, $c_s:=c_1(L_s)$ comes from the restriction to $X_s$ of some common cohomology class $\tilde c\in H^2(\mathcal{X}, \mathbb{Z})$ since
the union of the cohomology classes in $H^2(\mathcal{X}, \mathbb{Z})$ is at most countable and $c_1(L_t)$ live in $H^2(X_t, \mathbb{Z})$ for uncountably many fibers $X_t$.

By Proposition \ref{slia}, $R^2\pi_*\mathcal{O}_{\mathcal{X}}$ can be identified with
a vector bundle of rank $h$ on some open subset $U$ of $\Delta$
with $\Delta\setminus U$ being a proper analytic subset of $\Delta$ which is a discrete subset of $\Delta$.
But we prefer to assume further that $\mathcal{O}_{\mathcal{X}}$ is cohomologically flat
on $U$ (in dimension $2$) which in our case means that $h^2(X_t, \mathcal{O}_{X_t})$ is independent of $t\in U$ after possibly shrinking $U$;
see Lemma \ref{gct} and Theorem \ref{Upper semi-continuity}.
One sees that the intersection $\hat{S}:= S \cap U$ remains uncountable.

For any $s$ of $\hat{S} \subseteq S$,
consider the commutative diagram with $\mathcal{X}_U:=\pi^{-1}(U)$
\begin{equation}\label{les-1}
\xymatrix@C=0.5cm{
  \cdots \ar[r]^{}
  & H^1(\mathcal{X}_U, \mathcal{O}^*_{\mathcal{X}_U})\ar[d]_{} \ar[r]^{c_1}
  & H^2(\mathcal{X}_U, \mathbb{Z}) \ar[d]_{\alpha_s} \ar[r]^{\imath}
  & H^2(\mathcal{X}_U, \mathcal{O}_{\mathcal{X}_U})\ar[d]^{\Upsilon}\ar[r]^{} & \cdots \\
   \cdots \ar[r]
  &H^{1}({X_{s}},\mathcal{O}_{X_{s}}^*) \ar[r]^{c_1}
  & H^{2}(X_{s},\mathbb{Z}) \ar[r]
  & H^{2}(X_{s},\mathcal{O}_{{X_{s}}})\ar[r]&\cdots.}
\end{equation}
As argued in the first paragraph, there exists some $\tilde{c}\in H^2(\mathcal{X}, \mathbb{Z})$ such that $\alpha_s(\tilde{c})=c_s$ for the restriction map $\alpha_s$. Here we denote also by $\tilde{c}$ its restriction to $U$. Recall that $H^2(\mathcal{X}_U, \mathcal{O}_{\mathcal{X}_U})\cong \Gamma(U, R^2\pi_*\mathcal{O}_{\mathcal{X}_U})$  as obtained by the Leray spectral sequence argument in Theorem \ref{leray} with $U$ being Stein; we shall adopt them interchangeably in what follows.
Then we claim
$$\imath(\tilde{c})=0\in H^2(\mathcal{X}_U, \mathcal{O}_{\mathcal{X}_U}),$$
where $\imath$ is induced by
$\mathbb{Z}\to \mathcal{O}_{\mathcal{X}_U}$.
For this claim, first note that the image of $\imath(\tilde{c})$ under the map $\Upsilon$ to
$H^2(X_s, \mathcal{O}_{X_s})$, is zero on $X_s$
since $c_s$ is the first Chern class of the line bundle $L_s$ on $X_s$, $s\in \hat{S}$; see the second row of \eqref{les-1}.
Recall above that $R^2\pi_*\mathcal{O}_{\mathcal{X}_U}$ is a vector bundle over $U$.  By Lemmata \ref{ccs-bc}, \ref{2-5} and the cohomological flatness in our assumption, the fiber of this vector bundle at the closed point $s$ can be identified with $H^2(X_s, \mathcal{O}_{X_s})$, and the section $\imath(\tilde{c})$ is a holomorphic section of this vector bundle.   Now that the holomorphic section $\imath(\tilde{c})$ has zeros on $\hat{S}$ as just noted, and that $\hat{S}$
is uncountable, one concludes with Lemma \ref{section-zero} that
$\imath(\tilde{c})\equiv 0\in \Gamma(U, R^2\pi_*\mathcal{O}_{\mathcal{X}_U})\cong H^2(\mathcal{X}_U, \mathcal{O}_{\mathcal{X}_U})$, proving our claim above.

Fix any $p$ of $\Delta\setminus U$.  We shall prove that the preceding assertion $\imath(\tilde{c})=0$ on $U$ yields $\imath(\tilde{c})=0$ on a
neighborhood $V_p$ of $p$.  But this claim follows immediately if one applies Propositions  \ref{Hodge-torsionfree} and \ref{s-torsion} to a neighborhood $V_p$ of $p$ by using the deformation invariance of  $h^{0,1}(X_t)$ for all $t\in \Delta$. Notice that $\Delta\setminus U$ is a discrete subset of $\Delta$.

Combining these, we can conclude $\imath(\tilde{c})=0$ in $\Gamma(\Delta, R^2\pi_*\mathcal{O}_{\mathcal{X}})$.
By $\Gamma(\Delta, R^2\pi_*\mathcal{O}_{\mathcal{X}})\cong H^2(\mathcal{X}, \mathcal{O}_\mathcal{X})$ again,
one sees that by the similar long exact sequence in the diagram \eqref{les-1} with $\mathcal{X}_U$ replaced by $\mathcal{X}$, $\tilde{c}$ can be the first Chern class of a global line bundle $L$ on $\mathcal{X}$.  That is,
$L|_{X_s}$ and $L_s$ have the same first Chern class $c_s$ for any $s$ of $\hat{S}$.
This proves the proposition.
\end{proof}

\begin{cor}[{}]\label{global-B-cor}
With the similar setting to Proposition \ref{global-B}, one just assumes the deformation of $h^{0,2}(X_t)$ instead. Then the same conclusion of Proposition \ref{global-B} holds.
\end{cor}
\begin{proof}
This is a direct result of Proposition \ref{global-B} and Corollary \ref{h0201}.
\end{proof}

\begin{prop}[{cf. \cite[Proposition 4.16]{rt}}]\label{propa}
Let $M$ be a compact complex manifold with two holomorphic
line bundles $L_1$ and $L_2$ with $c_1(L_1)=c_1(L_2).$   Suppose that
$L_1$ is big.  Then $L_2$ is big too.
\end{prop}
\begin{proof}[Proof of Theorem \ref{thm-moishezon-update}]
By applications of Propositions \ref{global-B}, \ref{propa} and Corollary  \ref{unc-big}, one proves Theorem \ref{thm-moishezon-update}.
\end{proof}

\subsection{Proof of Theorem \ref{main-thm}} \label{pf-mt}
Based on the techniques developed in the first three subsections and Takayama's Theorem \ref{0mainT}, we modify the strategy in Subsection \ref{intro-2} to give:
\begin{proof}[{Proof of  Theorem \ref{main-thm}.\eqref{main-thm-i}}]
 Note that the family in this case is smooth itself and thus we can skip the first step of the strategy, that is, one doesn't need the semi-stable reduction.
 Then, just as done in \cite[Proposition 4.14]{rt}, when the family in Proposition \ref{global-B} is smooth over $\Delta$, the map $\alpha_s$ in \eqref{les-1} is an isomorphism and thus the assumption \eqref{surj-chern'} can be dropped. The case \eqref{main-thm-i} of Theorem \ref{main-thm} satisfies the assumptions of Theorem \ref{thm-moishezon-update} since the desired deformation invariance of Hodge number is just a special case of Proposition \ref{01-invariance}. So by Theorems \ref{thm-moishezon-update}, \ref{0thm-moishezon} and \ref{0mainT} (or Theorems  \ref{thm-moishezon} and \ref{0mainT}), we have actually proved Theorem \ref{main-thm} under its assumption \eqref{main-thm-i}.
\end{proof}
\begin{proof}[{Proof of  Theorem \ref{main-thm}.\eqref{main-thm-ii}}]
One assumes that the total space $\mathcal{X}$ is smooth upon replacing it by a proper modification (which would not change
the general type condition on generic fibers, in contrast to the big anticanonical bundle
case in Remark \ref{anti}), and then that $\pi: \mathcal{X}\rightarrow \Delta$ is
smooth outside a proper analytic subset $B$ of $\Delta$
according to generic smoothness \cite[Corollary 10.7 of Chapter III]{Ht}. Thus,
the general type assumption and Proposition \ref{ki-dim-limit} imply that
the line bundle $K_\mathcal{X}|_{X_t}=K_{X_t}$ is big for any $t\in \Delta$ outside $B$, and one further assumes that $\pi$ is a one-parameter degeneration with the line bundle $K_\mathcal{X}|_{X_t}=K_{X_t}$ being big for any $t\in \Delta^*$, possibly after shrinking $\Delta$. Hence, Theorem \ref{0thm-moishezon} implies that the one-parameter degeneration (and also the original family) are Moishezon and Theorem \ref{0mainT} gives the deformation invariance of plurigenera of the original family.
\end{proof}

Recall that one defines the \emph{Kodaira--Iitaka dimension} $\kappa(D)$ of a \emph{$\mathbb{Q}$-Cartier divisor} $D$ on a compact complex variety $X$ as $\kappa(\mathcal{O}_X(m_0D))$, where $m_0$ is a positive integer such that $m_0D$ is Cartier.

\begin{rem}\label{anti}
Here we consider the big anticanonical bundle analogue of Theorem \ref{main-thm}.\eqref{main-thm-ii}: Let $\pi:\mathcal{X}\rightarrow \Delta$ be a family such that any of the following holds:
\begin{enumerate}[(i)]
    \item \label{main-thm-iii}
each compact variety $X_t$ at $t\in \Delta$ has only canonical singularities and satisfies the Kodaira--Iitaka dimension $\kappa(K^*_{X_t})=\dim_{\mathbb{C}}{X_t}$, while the total space $\mathcal{X}$ is smooth;
    \item \label{main-thm-iii'}
each (not necessarily projective) compact variety $X_t$ at $t\in \Delta$ has only canonical singularities and uncountable fibers therein satisfy  $\kappa(K^*_{{X_t}})=\dim_{\mathbb{C}}{{X_t}}$, while $\mathcal{X}$ is smooth.
\end{enumerate}
Then for each positive integer $m$ and any $t\in \Delta$, the $m$-genus $P_m({X_t})$ vanishes.
Let us first prove this under the assumption \eqref{main-thm-iii}. One quick way is that $P_m(X_t)=\dim_{\mathbb{C}}H^0(X_t, K_{X_t}^{\otimes m})=0$ by the canonical singularities assumption and $\kappa(K^*_{X_t})=\dim_{\mathbb{C}}{X_t}$, while we also present a more general argument. Both will be used for a more degenerate case in the second paragraph of this remark. As the total space $\mathcal{X}$ is smooth, the generic smoothness  \cite[Corollary 10.7 of Chapter III]{Ht} implies that the family $\pi$ is smooth outside a proper analytic subset of $\Delta$ and thus one can assume that $\pi$ is a one-parameter degeneration. Then the line bundle $K^*_\mathcal{X}|_{X_t}=K^*_{X_t}$ is big for any $t\in \Delta$ and the family $\pi$ is Moishezon by Theorem \ref{0thm-moishezon}. Here the canonical singularities assumption for each $X_t$ is applied to  $K^*_\mathcal{X}|_{X_t}=K^*_{X_t}$ (as a special case of the {\it base change property}, cf. Remark \ref{3.21}  and Theorem \ref{singular} below).  Hence, Theorem \ref{0mainT} gives the deformation invariance of plurigenera of $\pi$.
The proof of deformation invariance of plurigenera under \eqref{main-thm-iii'} is quite similar. The only difference is an additional application of Proposition \ref{ki-dim-limit} to conclude that for all $t\in \Delta$, the line bundles $K^*_\mathcal{X}|_{X_t}=K^*_{X_t}$ are big, before the application of generic smoothness. Actually, we can even remove the smoothness of the total spaces in Remark \ref{3.21} and Theorem \ref{singular}.

Now let us state the more degenerate case mentioned in the first paragraph of this remark. Let $\pi: \mathcal{X}\rightarrow \Delta$ be a one-parameter degeneration and $X_t$ admit big anticanonical bundles for uncountably many $t\in \Delta^*$. Then for each positive integer $m$ and any $t\in \Delta$, the $m$-genus $P_m({X_t})$ vanishes. Here $P_m(X_0)$ for the reference fiber $X_0$ with support
$(X_0)_{red}=\sum_{i\in I} X_i$ is considered as $\sum_{i\in I} P_m(X_i)$. In fact, since $\mathcal{X}$ is smooth, the bigness assumption and Proposition \ref{ki-dim-limit} imply that for any $t\in \Delta^*$ the line bundle $K^*_\mathcal{X}|_{X_t}=K^*_{X_t}$ is  big and thus $P_m(X_t) =0$.
Notice that $X_0$ could be very singular and we need not know anything about $K_{X_0}$ for our argument. Then Theorem \ref{0thm-moishezon} gives that $\pi$ is an algebraic morphism and Theorem \ref{main2T} implies the lower-semi continuity of plurigenera at $t =0$. As any plurigenus is always nonnegative and the lower-semi continuity holds as above, all $P_m(X_0) =0.$

This type of degeneration cases seems non-trivial according to the counterexample \cite[Example 2]{n83} to Kodaira dimension under degeneration which only assumes the general fiber $X_t$ to be a rational surface. A more general case of this result is discussed in Theorem \ref{singular}.
\end{rem}

\begin{rem}\label{3.21}
Under an additional assumption that each compact variety $X_t$ at $t \in \Delta$ in  \eqref{main-thm-iii} and \eqref{main-thm-iii'} of Remark \ref{anti} is projective and using the results in Subsection \ref{subs-Chow}, one may make an attempt to relax the smoothness assumption on the total space $\mathcal X$
of Remark \ref{anti}.\eqref{main-thm-iii} and \eqref{main-thm-iii'}. Indeed, as the pluricanonical sheaf ${\omega^{[r]}_\mathcal{X}}=j_*(\Omega^n_{\mathcal{X}_{reg}})^{\otimes r}$ for some positive integer $r$ with $j: \mathcal{X}_{reg}\rightarrow \mathcal{X}$ being the inclusion of the smooth part of the normal $\mathcal{X}$
of dimension $n$ is invertible by \cite[Definition 1 and Proposition 7]{sv} with the canonical singularities assumption on each fiber, an application of the base change properties of relative
(anti)canonical sheaf \cite[Theorem 3.6.1]{Cn00}, \cite{kk} and Proposition \ref{ki-dim-limit} with the bigness and canonical singularities assumptions to some negative power of the relative canonical sheaf of the family gives the desired big line bundle on each fiber. In fact, \cite[Theorem 7.9.3]{kk} requires the assumption that the family $\mathcal{X}\rightarrow \Delta$ be projective\footnote{\label{ft}We thank Professor Linquan Ma for bringing to our attention that by \cite[Corollary 1.5 and Remark 5.6]{kk18} the projective morphism condition can be relaxed.  In our canonical singularity case, all the fibers and the total space are Cohen-Macaulay and normal (cf. \cite[proof of Theorem 1.1 on p. 16]{[Ta]}),
which thus fits into \cite[Remark 5.6]{kk18} so that the base change property above can be satisfied.  Though the works above are intricately discussed in the algebraic setting, it is expected that the similar can be said in the analytic setting.} in order for base change properties of the relative canonical sheaf to hold (see also \cite[Introduction and  and (1.0.b)]{pat} for a less technical explanation, referred to as {\it Koll\'ar's condition} there).
Using Lemma \ref{sing-chow} of Subsection \ref{subs-Chow}, this needed assumption is expected to be satisfied outside some proper analytic subset of $\Delta$ (see Remark \ref{Remark 4.30}); if so, this is sufficient for our purpose of obtaining the desired big line bundles on uncountably many fibers out of the (global) relative canonical sheaf on $\mathcal{X}$. While the above leads us to some topics for further study,
an alternative, more direct approach that allows us to remove the smoothness of $\mathcal{X}$ can
be given as follows. Observe that in Theorem \ref{singular} below, we do not need the extra assumption made here that the fibers of $\mathcal{X}$ in Remark \ref{anti}.\eqref{main-thm-iii} and \eqref{main-thm-iii'} are projective.
\end{rem}

\begin{thm}[{}]\label{singular}
Let $\pi:\mathcal{X}\rightarrow \Delta$ be a family such that each compact complex variety $X_t$ at $t\in \Delta^*$ has only canonical singularities and uncountable fibers therein satisfy $\kappa(K^*_{{X_t}})=\dim_{\mathbb{C}}{{X_t}}$. Then for each positive integer $m$ and any $t\in \Delta$, the $m$-genus $P_m({X_t})$ vanishes. The main concern here is that for the fiber $X_0$ which is possibly reducible, $P_m(X_0)$
(as defined in the last paragraph of Remark \ref{anti}) vanishes.
\end{thm}
This is obviously an adaption of Remark \ref{anti}.\eqref{main-thm-iii'}, without assuming the smoothness of the total space or the canonical singularities on the reference fiber. See Remark \ref{3.19} and Example \ref{3.22} below for nonvanishing results under certain circumstances.

\begin{proof}[Proof of Theorem \ref{singular}]
We first assume $\mathcal{X}$ to be normal. By the analytic definition of canonical singularities in \cite[Definition 1]{sv}, one knows that the pluricanonical sheaf ${\omega^{[q]}_\mathcal{X}}$ for some positive integer $q$ is actually invertible on $\pi^{-1}(\Delta\setminus\{0\})$ by the canonical singularities assumption (cf. Remark \ref{3.21} and Footnote \ref{ft}). Then one resorts to the base change
property by A. Corti\footnote{See \cite[Proposition (16.4.2)]{Ct92} for Corti's result whose proof
can be adapted to the analytic setting here by noting that the dualizing sheaf
(i.e., $Ext^1(\mathcal{O}_S, \omega_{X})$ in the notation there) in the analytic context can be found in, e.g. \cite[Proposition 2]{rr70} and \cite[Corollary 5.3 of Chapter VII]{bs}.} and the canonical singularities assumption for the general fibers
to conclude that for any $t\in \Delta^*$, the associated line bundles
\begin{equation}\label{qk}
({\omega^{[q]}_\mathcal{X}})^*|_{X_t}=({\omega^{[q]}_{X_t}})^*.
\end{equation}
Here $\mathcal{F}^* = \mathcal{H}om_{\mathcal{O}_\mathcal{X}}(\mathcal{F}, \mathcal{O}_\mathcal{X})$ for a coherent sheaf $\mathcal{F}$ on $\mathcal{X}$ (e.g. \cite[Definition 5.28]{pr}).
Inserting now the associated global line bundle to  $({\omega^{[q]}_\mathcal{X}})^*$ over $\pi^{-1}(\Delta\setminus \{0\})$ in \eqref{qk} into Proposition \ref{ki-dim-limit} over the base $\Delta^*$ and using the assumption   $\kappa(K^*_{X_t})=\dim_{\mathbb{C}}{X_t}$ for uncountable fibers, we have for any $t\in \Delta^*$,
\begin{equation}\label{big}
\kappa(K^*_{X_t})=\dim_{\mathbb{C}}{X_t}
\end{equation}
and further the canonical singularities assumption for the general fibers gives (cf. \cite[L. 1-2 on p. 17]{[Ta]})
\begin{equation}\label{pg0}
P_m(X_t)=\dim_{\mathbb{C}}H^0(X_t, K_{X_t}^{\otimes m}),
\end{equation}
which vanish by using the bigness assumption.

Take a suitable proper modification $\mu: \mathcal{Y}\rightarrow \mathcal{X}$ from a smooth manifold $\mathcal{Y}$ to $\mathcal{X}$.
Then the Weil divisor $\mu^*(rK^*_\mathcal{X})$ given by the codimension one part of the complex inverse image space\footnote{See \cite[$\S$ 2.6 of Chapter 1]{Grr} and \cite[Caution 7.12.2 of Chapter II]{Ht}. Being in the analytic setting, to avoid the possible difficulty  for the existence of a canonical divisor one may consider the dual $({\omega^{[r]}_\mathcal{X}})^*$ of the pluricanonical sheaf, where $\omega_\mathcal{X}={\omega^{[1]}_\mathcal{X}}$ is the dualizing
sheaf (\cite[Lemma and Definition 5.33]{pr})
and $r \in q\mathbb{N}$.  Take $s \in H^0(\mathcal{X}, ({\omega^{[r]}_\mathcal{X}})^*)$, which is nontrivial on the general fibers for $r\gg 1$ by using \eqref{qk} and \eqref{big} and Grauert's direct image Theorem \ref{gdit}.
Let the open subset $\mathcal{X}_{reg}$ be the smooth part of $\mathcal{X}$.
Note that the complement of $\mathcal{X}_{reg}$ in $\mathcal{X}$ is of
codimension at least two since $\mathcal{X}$ is
assumed to be normal.
Set $S'$ as the zero set of $s|_{\mathcal{X}_{reg}}$ on $\mathcal{X}_{reg}$.
Then $S'$ is an analytic subset of $\mathcal{X}_{reg}$
and its local dimension is exactly $\dim_\mathbb{C}\mathcal{X}_{reg} -1$ since the pluricanonical sheaf is
necessarily locally free on the smooth part.
Define $rK^*_\mathcal{X}$ to be the closure of $S'$ in $\mathcal{X}$ (in the sense
of complex topology).  By using the extension
theorem of Remmert--Stein \cite{rs}, $rK^*_\mathcal{X}$ is an analytic subset
of $\mathcal{X}$ and is of local dimension $\dim_\mathbb{C}\mathcal{X}-1$.
Now we define $\mu^*(rK^*_\mathcal{X})$ to be the codimension one part of the complex inverse image
 of $rK^*_\mathcal{X}$. We are not sure whether the zero set of $s$ on $\mathcal{X}$ in this case is
necessarily a closed analytic subset (as $s$ is only a section
of a coherent reflexive sheaf on $\mathcal{X}$ which is not locally free) or coincides with $rK^*_\mathcal{X}$.} by $\mu$ (by abuse of notation) is a Cartier divisor since $\mathcal{Y}$ is smooth. However, the notation $\mu^*(rK^*_\mathcal{X})$ is not to be mixed up with the usual concept
of pull-back unless $rK_\mathcal{X}$ is assumed to be Cartier.
For any $t\in \Delta^*$
the invertible sheaves by the analytic inverse image satisfy
\begin{equation}\label{qmuk}
(\mu^* {\omega^{[r]}_\mathcal{X}})|_{Y_t}=\mu^* ({\omega^{[r]}_\mathcal{X}}|_{X_t}),
\end{equation}
where $Y_t:=\mu^{-1}(X_t)$.
Since $\mathcal{Y}\rightarrow \Delta$ is
smooth outside a proper analytic subset $B$ of $\Delta$
according to generic smoothness \cite[Corollary 10.7 of Chapter III]{Ht}, the equalities \eqref{qmuk}, \eqref{qk} and \eqref{big} imply that for any $t\in \Delta$ outside $B\cup \{0\}$, $\mathcal{O}_{\mathcal{Y}}(\mu^*(rK^*_\mathcal{X}))|_{Y_t}$ is big.
One further assumes that $\mathcal{Y}\rightarrow \Delta$ is a one-parameter degeneration with the associated global line bundle to   $\mathcal{O}_{\mathcal{Y}}(\mu^*(rK^*_\mathcal{X}))$ over $\mathcal{Y}$ and $\mathcal{O}_{\mathcal{Y}}(\mu^*(rK^*_\mathcal{X}))|_{Y_t}$ being big for any $t\in \Delta^*$, possibly after shrinking $\Delta$. Theorem \ref{0thm-moishezon} implies that $\mathcal{Y}\rightarrow \Delta$  is an algebraic morphism and in turn so is $\mathcal{X}\rightarrow \Delta$. Thus
Theorem \ref{main2T} gives rise to the lower semi-continuity of plurigenera of $\mathcal{X}\rightarrow \Delta$ at $t=0$.
Hence, all $P_m(X_i)=0$ for every component $X_i$ of $X_0$ by the nonnegativity of plurigenera, \eqref{pg0} and the lower semi-continuity as above.

Next, we consider the case that $\mathcal{X}$ is not normal. Take the normalization $\nu: \mathcal{N}\rightarrow \mathcal{X}$. Since  $\mathcal{X}^*:=\pi^{-1}(\Delta^*)$ is already normal (cf. Footnote \ref{ft}), $\mathcal{N}^*:=\nu^{-1}(\mathcal{X}^*)$ is biholomorphic to $\mathcal{X}^*$ by a basic property of normalization (e.g. \cite[2.26.c)]{fis}).
So this normalization maintains all the properties of the general fibers $X_t$ involved in the preceding paragraph, up to biholomorphism, and also the Moishezonness (i.e., algebraic morphism) property of $\mathcal{X}\rightarrow \Delta$ since a normalization is already a bimeromorphic morphism. Hence, one can proceed the argument in the preceding paragraph just with the total space $\mathcal{X}$ replaced by $\mathcal{N}$.
\end{proof}

\begin{rem}\label{xu}
C. Xu provided us an algebraic counterpart for Theorem \ref{singular}:
 For a discrete valuation ring $R$ with quotient field $K$ and residue field $k$,
 let $f: \mathcal{X}\rightarrow T$ be a morphism, where $T = Spec\ R$ and $\mathcal{X}$ is normal and irreducible. If
$X_K$ is uniruled, then $X_k$ has uniruled components. See \cite[1.5.1 Corollary of Chapter IV]{k96} or \cite[Sublemma on p. 451]{sa} which is valid on $\mathbb{C}$, for a degeneration argument. Recall that if a smooth proper variety in characteristic $0$ is uniruled, then all its plurigenera vanish, while the converse is a famous conjecture (e.g. \cite[1.11 Corollary and 1.12 Conjecture of Chapter IV]{k96}). Notice that this algebraic counterpart is different from the analytic one. In Nishiguchi's (analytic) Example \ref{3.22}, there is a one-parameter degeneration whose general fibers have all vanishing plurigenera and degenerate fiber has one component with nonvanishing plurigenera.
\end{rem}
\begin{rem}\label{3.19}
The `algebraic morphism' assumption is indispensable in Theorem \ref{main2T} (= \cite[Theorem 1.2]{[Ta]}), even if the general fiber $X_t$ is projective for $t \ne 0$.
In other words, there exists a flat family $\mathcal{X}\rightarrow \Delta$ such that $X_t$ is projective for $t\ne 0$ while $\mathcal{X}\rightarrow \Delta$
is not an algebraic morphism (even after shrinking $\Delta$). This, in turn, shows the indispensability of the bigness assumption at least up to our proof of Theorems \ref{main-thm}.\eqref{main-thm-ii} and \ref{singular}. The following are two examples.
Note that the elliptic fibration in \cite[Remark 2.25, Definition 3.5.(2)]{cp} is global and is not applicable to this.
\begin{ex}\label{3.22} The following examples violate the conclusion of Theorem \ref{main2T}, hence cannot be algebraic
morphisms by the same theorem.
 One example is \cite[Example 2]{n83}, where the general fiber $X_t$ is a rational surface (with all vanishing plurigenera) for $t\ne 0$ and the singular fiber contains an Enriques surface (with even-genus $1$ and odd-genus $0$) and a Hopf surface as components. Note that a rational surface is Moishezon, K\"ahler and thus projective. The other example is also Nishiguchi's, which was announced in \cite[Note on p. 383]{n83} and explicitly presented in \cite[Example 5 in $\S$ 6]{n83a} by logarithmic transformations. In this example, the general fiber $X_t$ for $t\ne 0$ is an abelian surface  and the singular fiber contains a component of a surface with Kodaira dimension $1$.
 See also two generalizations of this example to the degenerations with the general two dimensional (possibly) non-K\"ahler fibers of non-positive Kodaira dimensions and one component of the special fiber of Kodaira dimension $1$ in \cite{u83} and \cite[Example 3.4.(1)]{Fj88}.
\end{ex}
\end{rem}
\begin{rem}\label{3.23}
 Theorem \ref{main2T} shows that the degenerations in Nishiguchi's counter-examples \cite{n83} cannot induce the bimeromorphic embedding in Theorem \ref{0thm-moishezon}, unless one makes the assumptions on high Kodaira dimensions of the fibers, such as the bigness assumption on the fibers in Theorems \ref{main-thm}.\eqref{main-thm-ii} and \ref{singular}.
\end{rem}
\begin{rem}\label{3.24}
According to the above remarks (in particular on Theorem \ref{main2T} and Remark \ref{anti}), there exist at most countable general fibers of the semi-stable degenerations in \cite[Example 2]{n83} which have big anticanonical bundles although all the general fibers (rational surfaces) have algebraic dimension $2$ and Kodaira dimension $-\infty$. This seems a non-trivial fact.
\end{rem}

\section{Family structure of projective complex spaces}\label{mtii}
Much inspired by Demailly's Question \ref{quest} and (the proof of) Theorem \ref{main-thm}, we further study the structure of the family of projective complex analytic spaces in Chow-type Lemmata \ref{Chow-sm} and \ref{sing-chow} for smooth and flat families, respectively. Besides the techniques developed in the first three subsections of Section \ref{mf}, we still resort to the local invariant cycle theorem \cite{cm} and effective algebro-geometric results and techniques.

From now on, we assume in the first three subsections of this section that \textbf{the family $\pi: \mathcal{X}\rightarrow \Delta$ is a semi-stable degeneration} as explained in Subsection \ref{intro-2} or \ref{di0q}, unless mentioned otherwise.
Furthermore, we assume that all the fibers of this reduced family are projective by Proposition \ref{reduction} and
\begin{lemma}\label{bu-proj}
Let $\pi: \mathcal{X}\rightarrow \Delta$ be a family with projective fibers  $X_t$ and  $f: \mathcal{X}' \rightarrow \mathcal{X}$ a blow-up of $\mathcal{X}$ along some closed complex subspace of $\mathcal{X}$.
Then the new fiber $f^{-1}(X_t)$ for each $t \in \Delta$ is still projective.
\end{lemma}
\begin{proof}
Notice that $f$ is a projective morphism (e.g. \cite[\S\ 2.7]{pe}) and its restriction to $X'_t:=f^{-1}(X_t)$ induces a projective
morphism $f_t: X'_t\rightarrow X_t$ by the definition of projective morphism.
So the composition of $f_t$ with $X_t\rightarrow \{t\}$ is still projective since a composition
of two projective morphisms is still projective and $X_t\rightarrow \{t\}$ is obviously projective.
Thus, $X'_t\rightarrow \{t\}$ being projective gives that $X'_t$ is projective.
 \end{proof}
\subsection{Monodromy and local invariant cycle theorem}\label{mon-lict}
 Recall that a \emph{degeneration} is a flat family $\pi: \mathcal{X}\rightarrow \Delta$ of relative dimension $n$ from the complex manifold $\mathcal{X}$ such that $X_t:=\pi^{-1}(t)$ is a smooth complex variety for $t\neq 0$.
 Let $\pi: \mathcal{X}\rightarrow \Delta$ be a degeneration, and $\pi^*: \mathcal{X}^{*}:=\pi^{-1}(\Delta^*)\rightarrow \Delta^*$ a smooth family
of complex projective manifolds over the punctured disk. It is well known
that for each $t_1, t_2 \in \Delta^*$, the fiber $X_{t_1}$ is diffeomorphic to the fiber $X_{t_2}$ by Ehresmann's theorem \cite{E}. In particular, the fibers are all homeomorphic, and
the cohomology groups $H^{\bullet}(X_t,\mathbb{C})$ are isomorphic for all $t \in \Delta^*$. Fix a base
point $\star\in\Delta^*$ and consider a path $\gamma: [0, 1]\rightarrow \Delta^*$ that generates $\pi_1(\Delta^*,\star)$.
The family of groups $H^{\bullet}(X_{\gamma(\tau)},\mathbb{C}), \tau\in [0,1]$, determines an automorphism
of $H^{\bullet}(X_\star,\mathbb{C})$. The induced homomorphism
$$\pi_1(\Delta^*, \star)\rightarrow \textrm{Aut}\ H^{\bullet}(X_\star,\mathbb{C})$$
is called the (analytic) \emph{monodromy} representation of the family.

For a family that extends to a smooth family over $\Delta$, the monodromy
representation is trivial. In this sense, monodromy is an invariant that is
meant to detect something about the singularities of the central fiber of a
degeneration. Moreover, we note that it is not the case that trivial
monodromy implies that a family can be extended to a smooth family over
the disk. See a more interesting example due to Friedman \cite{fr83}.

For the monodromy (or Picard--Lefschetz) transformation $T$ (or $T^{-1}$) that generates
the monodromy representation, we have:
\begin{thm}[Monodromy theorem by Landman \cite{lan}]\
\begin{enumerate}[$(i)$]
\item $T$ is quasi-unipotent, with index of unipotency at most $m$. In
other words, there is some $k$ such that
$(T^k-I)^{m+1} = 0$.
\item If $\pi: \mathcal{X}\rightarrow \Delta$  is semi-stable, then $T$ is unipotent ($k =1$).
\end{enumerate}
\end{thm}

Thanks to this theorem, we may define the logarithm of $T$ in the
semi-stable case by the finite sum
$$N:=\log T = (T-I)- \frac{1}{2}(T-I)^2 + \frac{1}{3}(T-I)^3- \cdots.$$
Then $N$ is nilpotent, and the index of unipotency of $T$ coincides with the
index of nilpotency of $N$; in particular, $T = I$ if and only if $N = 0$.
In Theorem \ref{inv-cyc} below, we will construct an element in the kernel of $N$.

As is evident in the previous remark, if $\pi: \mathcal{X}\rightarrow \Delta$ is a generically
smooth family of complex projective varieties, the topology of $X_0$
is related to the monodromy of the family. The Clemens--Schmid exact
sequence makes this precise.
If the components of $X_0$ are
K\"ahler and in addition there exists a class in
$H^2(X_0,\mathbb{R})$ which restricts to a K\"ahler class on each component of $X_0$, then we call $\pi: \mathcal{X}\rightarrow \Delta$  a \emph{one-parameter K\"ahler degeneration}.
The Clemens--Schmid exact sequence studies the homomorphism
$N: H^m(X_t)\rightarrow H^m(X_t)$. The first
piece of the sequence is the

\begin{thm}[{Local invariant cycle theorem, \cite{cm} or \cite[Theorem $11.43$]{ps}}]\label{lict}
Let $\pi: \mathcal{X}\rightarrow \Delta$ be a one-parameter
K\"ahler semi-stable degeneration with a fixed smooth fiber $X_t$
and the inclusion $i: X_t\subset \mathcal{X}$. Then
 the sequence
$$H^m(\mathcal{X})\xrightarrow{i^*} H^m(X_t)\xrightarrow{N} H^m(X_t)$$
is exact. In other words, all cohomology classes which are invariant under the
monodromy action come from cohomology classes in $\mathcal{X}$.
 \end{thm}

 This theorem will soon be applied to our family $\pi: \mathcal{X}\rightarrow \Delta$
set up in the beginning of this section, as it is easily seen to be
a one-parameter K\"ahler degeneration for the theorem.

\subsection{Constructing local invariant cycle}\label{clic}
Write $T$ for the monodromy transformation and still consider the semi-stable K\"ahler degeneration
$$\pi: \mathcal{X}\rightarrow \Delta$$
with projective fibers by Lemma \ref{bu-proj}.  Recall that $T$ acts on the (singular)  cohomology ring of the fiber, after going around a circle centered at $0$. Let $A_t$ be an ample line bundle of $X_t$, for $t\in \Delta$.  Then we have $c_1(A_t)$ and $T(c_1(A_t))$, as singular cohomology classes (on $X_t$).
Recall that $c_1(A_t)$ is a positive class since $A_t$ is ample and we  expect that $T(c_1(A_t))$ is also a positive class. A cohomology class $c$ in $H^2(M, \mathbb{Z})$ on a compact complex variety $M$, is
said to be a \emph{positive class} if there is an ample line bundle $L$ on $M$  such that $c_1(L) = c$.

Write the union $\Delta^* (:=\Delta\setminus \{0\})= U_0 \cup U_1$,
where $U_i, i=0,1,$ are simply-connected open subsets of $\Delta^*$.  Further, we assume that if we write the intersection $(U_0 \cap U_1)$ as a union $(V_0\cup V_1)$,
then $V_0$ and $V_1$ are simply-connected open subsets and disjoint from each other.
\begin{prop}\label{t-rep-m}
For $i=0,1$, there exist (global) line bundles $L_i$ on $\mathcal{X}_{U_i}$ and two thin subsets $Z_i\subset U_i$ such that
for any $\tau\in V_0\setminus \{Z_0\cup Z_1\}$, ${L_0}|_{X_\tau}$ is  very ample and
${T(c_1({L_0}|_{X_\tau}))}$  can also be represented by a very ample line bundle ${L_1}|_{X_\tau}$ on $X_\tau$. That is, $$c_1({L_1}|_{X_\tau}) ={T(c_1({L_0}|_{X_\tau}))}$$ on $X_\tau$. Here, we recall the assumption that for any $t$, $X_t$ is projective.
\end{prop}

Here a \emph{thin} subset of a complex space will mean a countable union of proper analytic subsets in this complex space.
This proposition connects algebraic properties (positivity or ampleness) with purely topological ones given via $T$. A generalization of our methods with the weaker assumption that
$X_t$ is projective for uncountably many $t$, is expected to be possible.
\begin{proof}[{Proof of Proposition \ref{t-rep-m}}]
We first deal with the ample case, i.e., find (global) line bundles $L_i$ on $\mathcal{X}_{U_i}$ and two thin subsets $Z_i\subset U_i$ such that
for any $\tau\in V_0\setminus \{Z_0\cup Z_1\}$, ${L_0}|_{X_\tau}$ is ample and
${T(c_1({L_0}|_{X_\tau}))}$ can also be represented by an ample line bundle ${L_1}|_{X_\tau}$ on $X_\tau$.

We start by:
\begin{lemma}\label{li-coin}
For $i=0,1$, there exist (global) line bundles $L_i$ on $\mathcal{X}_{U_i}$ and some thin subset  $Z_{0}$ of $U_0$ such that for any $s \in V_1\setminus Z_{0}$, the restrictions of $L_0$ and $L_1$ on $X_s$ have the same positive first Chern class (so both of them on $X_s$ are ample).
\end{lemma}
\begin{proof}[{Proof of Lemma \ref{li-coin}}]
To get Lemma \ref{li-coin}, we need:
 \begin{lemma}\label{l'lxs}
Let $\pi_U:\mathcal{X}_U\rightarrow U$ be a smooth subfamily of $\pi$ over a simply-connected open subset $U$ of $\Delta^*$.  Let $V \subset U$ be a  simply-connected open subset with a line bundle $L_V$ on $\mathcal{X}_V$.   Then for any $s \in V$, there exists a line bundle $L_U$ on $\mathcal{X}_U$ with $c_1(L_U|_{X_s}) =c_1(L_V|_{X_s})$.
\end{lemma}
\begin{proof} We use an idea similar to that of \cite[Lemma 3.6, Proposition 4.14]{rt}.
Use the commutative diagram:
\begin{equation}\label{les}
\xymatrix@C=0.5cm{
  \cdots \ar[r]^{}
  & H^1(\mathcal{X}_U, \mathcal{O}^*_{\mathcal{X}_U})\ar[d]_{i} \ar[r]^{c_1}
  & H^2(\mathcal{X}_U, \mathbb{Z}) \ar[d]_{j} \ar[r]^{\iota_{U}}
  & H^2(\mathcal{X}_U, \mathcal{O}_{\mathcal{X}_U})\ar[d]^{k}\ar[r]^{} & \cdots \\
   \cdots \ar[r]
  &H^{1}(\mathcal{X}_V,\mathcal{O}_{\mathcal{X}_V}^*) \ar[r]^{c_1}
  & H^{2}(\mathcal{X}_V,\mathbb{Z}) \ar[r]^{\iota_{V}}
  & H^{2}(\mathcal{X}_V,\mathcal{O}_{\mathcal{X}_V})\ar[r]&\cdots,}
\end{equation}
where the vertical arrows $i, j, k$ are restriction maps and $j$ is an isomorphism since the families here are over simply-connected domains.

With $c_1(L_V) \in H^2(\mathcal{X}_V, \mathbb{Z}) = H^2(\mathcal{X}_U, \mathbb{Z})$, we only need to prove $\iota_{U}(c_1(L_V)) =0$ by the exact sequence.  By the assumption on $L_V$, we have $\iota_{V}(c_1(L_V)) =0$ so that $\iota_{U}(c_1(L_V))$ is zero because $\iota_{U}(c_1(L_V))$ can be identified with a holomorphic section of the vector bundle $R^2{\pi_U}_*\mathcal{O}_{\mathcal{X}_U}$ on $U$ and vanishes identically if it vanishes on an open subset $V$.  Here the local freeness of  $R^2{\pi_U}_*\mathcal{O}_{\mathcal{X}_U}$ is implied by the deformation invariance of $(0,2)$-Hodge numbers (cf. Proposition \ref{01-invariance}, or for this smooth family case see \cite[Proposition 9.20]{[V]} or \cite[Theorem 1.3 or 3.1]{RZ15}) and Grauert's continuity Lemma \ref{gct}.
We can now complete the proof of the lemma by adding one row on $X_s$ to the bottom of \eqref{les} or the following remark with more information.
\end{proof}
\begin{rem}\label{rem-4.7}
For any fixed $s\in V$, one can use Wehler's argument (such as in \cite[Remark 3.7]{rt}) to find some line bundle $L_U$ on $\mathcal{X}_U$ such that $L_U|_{X_s}=L_V|_{X_s}$. If $L_V|_{X_s}$ is ample, this gives that
$c_1(L_U|_{X_s})$ is a positive class on $X_s$.
\end{rem}

\noindent\emph{Proof of Lemma \ref{li-coin} (continued).} Let $V$ be any simply-connected
domain in $\Delta^*$. By the argument of \cite[Lemma 3.6]{rt}, there is indeed a line bundle $L_V$ on $\mathcal{X}_V$ such that for any $s \in V\setminus Z_V$ with a countable union $Z_V$ of analytic subsets in $V$, ${L_V}|_{X_s}$ is ample on $X_s$. In fact, this argument can be divided into two steps. For the first step,
consider an ample line bundle $L_{t_0}$ on $X_{t_0}$ with some $t_0\in V$ and its first Chern class $c:=c_{1}(L_{t_0})\in H^2(X,\mathbb{Z})$, where $X$ is the underlying differentiable manifold of $X_t$ for each $t\in V$. Then by the exact sequence
\begin{equation}
\label{eshdis}
 \cdots\rightarrow  H^1(\mathcal{X}_V, \mathcal{O}^*_{\mathcal{X}_V})\rightarrow H^2(\mathcal{X}_V, \mathbb{Z})
\rightarrow H^2(\mathcal{X}_V, \mathcal{O}_{\mathcal{X}_V})\rightarrow\cdots
\end{equation}
and  the deformation invariance
of $h^{0,2}(X_t)$ over the whole $V$, we reach an analytic subset $Z_c$,
the zero set in $V$ of the holomorphic section $s_c\in \Gamma(V,R^2\pi_*\mathcal{O}_{\mathcal{X}_V})$ induced by $c$.  Now on $V$, the projectiveness of $X_t$ gives rise to
$$\bigcup_c Z_c\supseteq V$$
with the union taken over all the integral classes $c\in H^2(X,\mathbb{Z})$ (countably many) satisfying that $c=c_1(H_t)$ for some ample line bundle $H_t$ on $X_{t}, t\in V$.  Since a countable union of proper analytic subsets is Lebesgue negligible, in the union above there should be some $\tilde{c}\in H^2(\mathcal{X}_V, \mathbb{Z})$ induced by some ample line bundle $L_{\tilde{t}_0}$ on $X_{\tilde{t}_0}$ with some $\tilde{t}_0\in V$  satisfying $Z_{\tilde{c}}\supseteq V$, i.e., $Z_{\tilde{c}}=V$.  (Here $\tilde{t}_0$ could be different from $t_0$ in the beginning.) That is, $0 \equiv s_{\tilde{c}}\in\Gamma(V, R^2\pi_*\mathcal{O}_{\mathcal{X}_V})$. By using the identification \eqref{identif} preceding Theorem \ref{leray} and the long exact sequence \eqref{eshdis}, this
implies that $\tilde c$ is the image of some element in $H^1(\mathcal{X}_V,\mathcal{O}^*_{\mathcal{X}_V})$. This element is the desired global holomorphic line bundle ${L_V}$ on the total space $\mathcal{X}_V$ because its restriction ${L_V}|_{X_{\tilde{t}_0}}$ to $X_{\tilde{t}_0}$ is $L_{\tilde{t}_0}$ which is ample as mentioned earlier. Now we come to the second step: one uses the standard fact that the ampleness condition is open with respect to the countable analytic Zariski topology of $V$, which follows from the Nakai--Moishezon criterion for ampleness and the Barlet theory of cycle spaces \cite{bar}.

The above argument applies to $U_0$ and gives a holomorphic line bundle $L_0$ over $\mathcal{X}_{U_0}$ and a thin subset $Z_{0}$ of $U_0$ such that for any $s \in U_0\setminus Z_{0}$ the restriction $L_0|_{X_s}$ is ample.
Then, by Lemma \ref{l'lxs} (applied to $V=V_1$ and $U=U_1$) there exists a  (global) line bundle $L_1$ on $\mathcal{X}_{U_1}$ such that for any $s \in V_1\setminus Z_{0}$, the restrictions of $L_0$ and $L_1$ on $X_s$ have the same positive first Chern class. We have completed the proof of Lemma  \ref{li-coin}.
\end{proof}

\noindent\emph{Proof of Proposition \ref{t-rep-m} (continued).}
Despite of Lemma \ref{li-coin}, due to the monodromy effect, the first Chern classes of the restrictions of $L_0$ and $L_1$ could be different on $X_t$ for  $t$ of $V_0$. Using monodromy transformations, we may write
\begin{equation}\label{mono}
  c_1(L_1|_{X_t}) = T(c_1(L_0|_{X_t}))
\end{equation}
for the monodromy transformation $T$ (or $T^{-1}$). Again, as in the proof of Lemma \ref{li-coin}, one can find some thin subset $Z_1\subset U_1$, such that for any $\tau\in V_0\setminus \{Z_0\cup Z_1\}$,   $L_0|_{X_\tau}$ and $L_1|_{X_\tau}$ are ample line bundles on $X_\tau$ and satisfies \eqref{mono}. Thus, we have proved the ample case of Proposition \ref{t-rep-m}.

By the effective very ampleness Theorem \ref{Dem-fuj} below, one can obtain the desired very ample case by modifying the above argument and thus completes the proof of Proposition \ref{t-rep-m}.
\end{proof}
\begin{thm}[{\cite[Corollary 2]{Dem93} and also \cite{s96}}]\label{Dem-fuj}
If $A$ is an ample line bundle over an $n$-dimensional projective manifold $X$, then $K_X^{\otimes 2}\otimes A^{\otimes k}$ is very ample for $k>2C(n):=4C_nn^n$ with $C_n<3$ depending only on $n$.
\end{thm}
Moreover, one has:
\begin{cor}\label{tq}
For $i=0,1$, there exist two thin subsets $Z^i\subset U_i$ and some very ample line bundle $L_\tau$ on the fiber $X_\tau$ for any $\tau\in V_0\setminus \{Z^0\cup Z^1\}$ such that for every $k \in \mathbb{N}$, $T^k(c_1(L_\tau))$ is the first Chern class of some very ample line bundle $L_{\tau,k}$ on $X_\tau$.
\end{cor}
\begin{proof}
 It suffices to apply (the proof of) Proposition \ref{t-rep-m} inductively. In fact, the induction argument on the construction of $L_\bullet$ therein will give rise to the desired  line bundles $L_k,$ $k\in \mathbb{N},$ over $\mathcal{X}_{U_{m}}$ (with $m=0$ or $1$ as $k$ is even or odd), the countable union $Z^0:=\cup_{i\geq 0}Z_{2i}$ of thin subsets in $U_0$ and the countable union $Z^1:=\cup_{i\geq 0}Z_{2i+1}$ of thin subsets in $U_1$ such that for any $\tau\in V_0\setminus \{Z^0\cup Z^1\}$,
each $L_k|_{X_\tau}$ is very ample  and
$$  c_1(L_{k+1}|_{X_{\tau}}) = T(c_1(L_k|_{X_{\tau}})).$$
By choosing $L_\tau$ as $L_0|_{X_\tau}$ and $L_{\tau,k}$ as $L_k|_{X_\tau}$, one completes the proof.
\end{proof}

Recall that $T$ is a monodromy transformation, so that $T$ is a ring isomorphism.   In particular, it preserves the cup product. Thus,  we have
\begin{lemma}\label{cup-p}
For the ample divisor $A_t$ on $X_t$,
$${T(c_1(A_t))}^n = c_1(A_t)^n,$$
where $n$ is the dimension of the fiber $X_t$.
\end{lemma}

Next, we need the following important result, to be proved in the more concrete  Corollary \ref{deg-cor} and Proposition \ref{fgn} in Subsection \ref{pro-gt}.
\begin{prop}\label{finite}
Suppose that $\{L_i\}_{i\in J}$ is a set of very ample line bundles
     on an $n$-dimensional smooth projective variety of general type, and that $\sup_i \{L_i^n\}$ is finite. Then, $\{L_i\}_{i\in J}$ forms a bounded family.  In particular, $\{c_1(L_i)\}_{i \in J}$
       is a finite set.
\end{prop}
Finally, we come to the first main result of this section:
 \begin{thm}\label{inv-cyc} Consider a semi-stable degeneration $\pi: \mathcal{X}\rightarrow \Delta$ of projective manifolds of general type. With the notations before Proposition \ref{t-rep-m},
 for any $\tau\in V_0\setminus \{Z^0\cup Z^1\}$ with two thin subsets $Z^i\subset U_i$ and $i=0,1$, there exists a very ample line bundle $L_{X_\tau}$ over $X_\tau$, such that its singular cohomology class $c_1(L_{X_\tau})$  can be the image of
 $$H^2 (\mathcal{X}, \mathbb{Z})\rightarrow H^2 (X_\tau, \mathbb{Z}),$$
 and also a (global) line bundle $L$ over $\mathcal{X}$ such that for all $t\in \Delta^*$,
$L|_{X_t}$ are big.
\end{thm}
\begin{proof}
Consider $L_\tau$ for any $\tau$ in Corollary \ref{tq} (uncountably many as given there) with
  $$S_1:= \{c_1(L_\tau), T(c_1(L_\tau)), T^2(c_1(L_\tau)),\ldots\},$$
$S_2:=T(S_1)$ and similarly for $S_3$, etc.
Then one has $$S_1 \supseteq S_2\supseteq S_3\supseteq\cdots.$$
 By using Proposition \ref{finite} with the boundedness condition obtained in Lemma \ref{cup-p}, one knows that $S_1$ is finite, and thus that there exists an $m$ such that
             $$S_m = S_{m+1} = S_{m+2}=\cdots.$$
Thus, we get
              $$S_m = T(S_m) = T^2(S_m) =\cdots$$
by the definition of $S_i$, which implies that $S_m$ is a $T$-invariant set. Without loss of generality, one assumes
$$S_m=\{T^{m-1}(c_1(L_\tau)),\ldots,T^{m'}(c_1(L_\tau))\}$$
for some integer $m'\geq m-1$. Then set $$L_{X_\tau}:= L_{\tau,m-1} \otimes\cdots\otimes L_{\tau,m'},$$
 where $L_{\tau,\bullet}$ is constructed as in Corollary \ref{tq} and $\{c_1(L_{\tau,\bullet})\}$ run over $S_m$.
Since each class in $S_m$ is  represented by some (very) ample line bundle by Corollary \ref{tq} and $c_1(L_{X_\tau})$ is invariant under the monodromy transform $T$, the desired first assertion becomes a direct result of the local invariant cycle Theorem \ref{lict} and together with Theorem \ref{thm-moishezon-update} gives the second one.
\end{proof}

\subsection{Bounded family of very ample line bundles}\label{pro-gt}
To prove Proposition \ref{finite}, one hopes to get:
\begin{expe}\label{expect}
Let $\{A_i\}_{i\in \mathbb{N}}$ be a set of very ample divisors in a fixed $n$-dimensional  projective variety $M$.
Suppose that $\{A_i^n\}$ are uniformly bounded.
Then, $\{A_i\}$ forms a \emph{bounded family} (in $M$), i.e., there are only
finitely many Hilbert polynomials of $\{A_i\}$ with respect to a fixed ample divisor in
$M$.
\end{expe}
Unfortunately, this expectation is wrong in general and we give a  counter-example below.
\begin{ex}\label{k3}
By use of \cite[Theorem 4.1]{og15}, $g$ is an automorphism of a certain $K3$ surface $S$ there (see the second half of Page $8$ in \cite{og15}).
Now $g$ induces $\tilde f$ on $H^2(S)$ and $f$ on the Neron-Severi lattice $\mathcal{NS}$ of $S$.  $\mathcal{NS}$ is of rank $2$ in this
example. Here, $f$ has two positive eigenvalues, $u>1$ and $v<1$, and $uv=1$.
Let $D, E$ be eigenvectors of $f$ associated with $u, v$.
Let $L$ be a very ample divisor of projective $S$.
As $\mathcal{NS}$ is of rank $2$, write
                   $$L= aD + bE,\quad a,b>0.$$
Then for each $k\in \mathbb{N}$,
$$f^k(L) = au^kD + bv^k E,$$
denoted by $L_k$.
Each of $\{L_k\}_k$ is very ample because they are actually
$g^k(L)$ with the above $g$ in $Aut(S)$.
Since $f$, induced by $g$,  is an isometry on $\mathcal{NS}$ (preserving the
inner product induced by Poincar\'{e} pairing), these
$L_k$ satisfy the condition of Expectation \ref{expect}.
But we claim below that $\{L_k\}$ cannot form a
bounded family in the usual sense.  This gives a counterexample
to the above Expectation \ref{expect} for $n=2$.
Let's compute $(F\cdot L_k)$, where $F =sD+ tE$ is any fixed
ample divisor.
Note that, $D^2 = E^2 =0$ since $f$ is an isometry and $u>1, v<1$.
It follows that, if $L$ (resp. $F$) are ample, then
$a, b$ (resp. $s, t$) are both non-zero because $L^2= ab(D\cdot E)>0$, and that
$$(F\cdot L_k)= (atu^k+bsv^k) (D\cdot E).$$
This is unbounded by $u>1, 0<v<1$. This means that, whatever ample divisor $F$ is used,
$(L_k\cdot F)$ cannot be bounded.  So, there are no finitely many
Hilbert polynomials of $\{L_k\}$ with respect to a fixed ample divisor.
Now that $L_k^2>0$ are independent of $k$ as mentioned above, combined with
the above unboundedness this gives a counter-example to Expectation \ref{expect}.
\end{ex}

Let $\mathcal{U}$ denote the universal cover of the punctured disk $\Delta^*$.
Recall that our family is smooth over $\Delta^*$. Now pull back the family to $\mathcal{U}$, so we get a smooth family over a simply-connected space. It is somewhat easier to deal with a simply-connected base, \emph{without} consideration of monodromy.
According to Subsection \ref{clic}, working on $\mathcal{U}$ simplifies the picture.
For instance, we will have a global line bundle $L$ on $\mathcal{X}_\mathcal{U}$ (cf. Proposition \ref{t-rep-m}), such that $L$ restricts
to an ample line bundle $L_t$ for  $t$ outside a countable union $B$ of proper analytic subsets in  $\mathcal{U}$.
Also $H^2(X_t, \mathbb{Z})$ for all $t$ of $\mathcal{U}$ can be naturally identified under this picture.

To start with, $\mathcal{X}_\mathcal{U}$ is differentiably  $X\times \mathcal{U}$ with the underlying differentiable manifold $X$ of $X_t$ since $\mathcal{U}$ is simply connected.
Let $\{p_0,p_1, p_2, \ldots \}$ of $\mathcal{U}$ denote the pre-image of some point $p$ of $\Delta^*$.
Also for each $i=0,1,\ldots$, $X_i$ denotes the pull-back fiber $f^*X_p$ at $p_i$.

Let us now employ an argument similar to that of \cite[p.\ 308]{Griffith}.
We get a monodromy operator $T^k$ (similar to the \lq multiplier' there).  The difference is that our consideration here is fiber bundle (differentiable), while that of \cite{Griffith} is vector bundle (holomorphic).
But, the underlying ideas are similar for either case.  Let us skip the details.

Now the global line bundle $L$ mentioned above restricts to an ample line bundle on each $X_0, X_1, \ldots$ for the \emph{very generic} $p$ (more precisely, for $p$ outside a thin subset $B$ of $\mathcal{U}$). By the effective very ampleness Theorem \ref{Dem-fuj}, we replace $L$ by $aL +b K_{\mathcal{X}_\mathcal{U}}$ with
fixed constants $a >0$ and $b>0$ depending only on the dimension of the fibers such that it is very ample as long as $L$ is ample.
So,  we may further assume that actually $L$ restricts to a very ample divisor on each $X_0, X_1,\ldots$, denoted by $L_0, L_1,\ldots$.
Back to $X_p$ the above implies that there is a sequence of very ample divisors $D_0, D_1, \ldots$
on a fixed fiber $X_p$, and that $c_1(D_k) = T^k(c_1(D_0))$.

Note that $L_k^n$ is independent of $k =0, 1, 2, \ldots $, so is the intersection
number $(L_k^{n-1}\cdot K_{X_k})$ on each $X_k$.   Thus, we set
$$D_i^n=d_D,\quad (D_i^{n-1}\cdot K_{X_p})=d_K,\quad \text{for all $i=0, 1, 2, \ldots$}.$$
By \cite{kma}, one has an effective estimate
of $|qD_i |$ in terms of $d_D$ and $d_K$.  Roughly, its dimension is about the polynomial in $q$
\begin{equation}\label{ee}
\frac{1}{n!} d_D\cdot q^n+\text{(lower-order terms in $q$ with coefficients in polynomials of $d_D$ and $d_K$)}.
\end{equation}
See Remark \ref{aee} for an application of \eqref{ee} to our problem.

Assuming that each fiber of the family $\pi:\mathcal{X}\rightarrow \Delta$ is projective,
of general type and its restriction $\pi^*:\mathcal{X}^*\rightarrow \Delta^*:=\Delta\setminus \{0\}$ is a smooth family, we want to show that there exists an effective constant $m$, independent of $i$, with the property that for each $i=0, 1, 2, \ldots$,
\begin{equation}\label{(6.1)}
 m K_{X_p} = D_i + E_i,\quad  \text{with the divisor $E_i$ being effective.}
\end{equation}
\begin{cor}\label{deg-cor}
The degrees $\deg_D D_i:=(D^{n-1}\cdot D_i)$ of $D_i$ with respect to $D:=D_0$ are bounded above by $md_K$.
\end{cor}
\begin{proof}
By \eqref{(6.1)},    $\deg_D D_i \leq \deg_D m K_{X_p} = md_K.$
\end{proof}

By the proof of Kodaira's lemma (e.g. \cite[Proposition 2.2.6]{[La]}), to prove \eqref{(6.1)} is reduced to proving
that
\begin{equation}\label{(7.1)}
h^0(m K_{X_p}) > h^0( m K_{X_p}|_{D_i}).
\end{equation}
Note that $D_i$ can be
smooth and irreducible because it is very ample. We want to show that the lower bound estimate $(LBE)$ on the left-hand side of \eqref{(7.1)}
and the upper bound estimate $(UBE)$ on the right-hand side of \eqref{(7.1)}  satisfy
\begin{equation}\label{(8.1)}
(LBE) > (UBE).
\end{equation}
The difficulty of \eqref{(8.1)} lies in the trouble with effective bounds
for \eqref{(8.1)}, especially $(LBE)$ of \eqref{(8.1)}.  To get around this difficulty, we are going to
consider a modified version of \eqref{(6.1)} and \eqref{(8.1)}: For any ineffective $d\gg 0$ and $i$,
\begin{equation}\label{(9.1)}
                   md K_{X_p} = d D_i + F_i,
\end{equation}
such that the divisor $F_i$ is effective.
By the same token above, \eqref{(9.1)} is reduced to proving
that
\begin{equation}\label{(10.1)}
                   h^0(md K_{X_p}) > h^0( mdK_{X_p}|_{dD_i}),\quad  \text{for $d\gg 0$.}
\end{equation}
Here $dD_i$ is represented by a smooth irreducible divisor.
It suffices to prove that the lower bound estimate $(LBE)'$ on the left-hand side of \eqref{(10.1)}
and the upper bound estimate $(UBE)'$ on the right-hand side of \eqref{(10.1)} meet
\begin{equation}\label{(11.1)}
                   \text{$(LBE)'$ $>$ $(UBE)'$,\quad for $d\gg 0$.}
\end{equation}
Fortunately, the introduction of an ineffective $d$ is immaterial
to Corollary \ref{deg-cor}, which we rephrase:
\begin{cor}\label{eff-cor}
An effective version of Corollary \ref{deg-cor} holds
if an effective $m$ in \eqref{(10.1)} can be found (while $d$ could be ineffective).
\end{cor}
\begin{proof}
By $\deg_D dD_i \leq \deg_D md K_{X_p}$, the (ineffective) $d$ can be divided out
on both sides.  Thus, if $m$ is effective, we are done.
\end{proof}
From now on we focus ourselves on the proof of \eqref{(11.1)}. One needs the remarkable:
\begin{thm}[]\label{takthm}
For every positive integer $n$, there exists a positive constant
$\nu_n$ depending only on $n$ such that, for every $n$-dimensional smooth projective
variety $M$ of general type,
$$\dim_{\mathbb{C}} H^0(M,K_M^{\otimes k}) \geq \frac{\nu_n}{n!}k^n + o(k^n)$$
holds as $k\rightarrow \infty$.
\end{thm}
\begin{proof}
See \cite[Corollary 1.3]{Hm}, \cite[Theorem 1.2]{t06} and \cite[Theorem 1.2]{Ts06}. As for the effectiveness of $\nu_n$ in low dimensions, we refer the reader to the nice survey \cite{cc}.
\end{proof}
By Theorem \ref{takthm}, one has
\begin{equation}\label{(12.1)}
(LBE)' >  \frac{\nu_n}{n!}(md)^n + \text{lower-order terms},\quad \text{for $d\gg 0$}.
\end{equation}
Next, for $(UBE)'$ one can use
the $Q$-estimate of Matsusaka in \cite[Proposition 2.3]{lm}.
Recall that the $Q$-estimate of Matsusaka gives that for a normal projective variety $M$ of dimension $n$, a hyperplane section $H$ of $M$ and a torsion-free rank $1$ sheaf $\mathcal{F}$ on $M$, it holds
$$h^0(M,\mathcal{F})\leq
\begin{pmatrix} [\delta]+n\\n
\end{pmatrix}
\gamma+
\begin{pmatrix} [\delta]+n-1\\n-1
\end{pmatrix},$$
where with respect to $H$, $\delta=\frac{\deg F}{\deg M}$ and $\gamma=\deg M$. Let $C$ be an ample divisor and $E$ a one-codimensional cycle on $M$ such that
the Kodaira map associated to $kE$ is birational for $k\gg 0$. Then one has
\begin{equation}\label{(2.7.1)}
  h^0(M,kE)\leq \frac{1}{n!}(C^n)\cdot\left(\frac{(C^{n-1}\cdot kE)}{C^n}\right)^n
\end{equation}
as given in the proof of \cite[Proposition 2.2]{lm} (or \cite[(2.7.1)]{t97}).
By noting that $dD_i$ is of dimension $(n-1)$
and that the ample divisor $C$ above can be taken to be the restriction of $D_i$ to
a smooth divisor linearly equivalent to $dD_i$, one evaluates the following intersection numbers
\begin{equation}\label{(13.1)}
(mdK_{X_p}. C^{n-2})_{dD_i}
             = (mdK_{X_p}\cdot  D_i^{n-2}\cdot  (dD_i))_{X_p}
             = md^2 \deg_{D_i}K_{X_p} = md^2 d_K,
\end{equation}
where $(mdK_{X_p}\cdot  C^{n-2})_{dD_i}$ means the intersection number on ${dD_i}$ and similarly for the ${X_p}$-version.
So \eqref{(2.7.1)} and \eqref{(13.1)} imply that for all $d$, there hold
\begin{equation}\label{(13.2)}
(UBE)' \leq \frac{1}{(n-1)!}\frac{(md^2d_K)^{n-1}}{(dd_D)^{n-2}}\leq  \frac{1}{(n-1)!}(md_K)^{n-1} d^n.
 \end{equation}
Here we note that $d_D\geq 1$, so it can be dropped.
To reach \eqref{(11.1)}, by \eqref{(12.1)} and \eqref{(13.2)} it suffices to choose
\begin{equation}\label{(14.1)}
                m^n > \frac{n}{\nu_n}m^{n-1} d_K^{n-1}.
 \end{equation}
In sum, any positive integer $m$ satisfying
$$\label{(14.2)}
         m > \frac{n}{\nu_n}d_K^{n-1}
$$ completes the proof of \eqref{(9.1)}, as to be shown.

Corollary \ref{eff-cor} implies that the set of divisors $\{D_i\}_{i=0, 1 ,2,\ldots}$ on $X_p$
forms a bounded family since their degrees with respect to  $D=D_0$
are bounded above.

\begin{rem}
The above bounded family result restricts ourselves to the general type case if there
are singular fibers.  If $\mathcal{X}\rightarrow \Delta$ is a smooth family, it is seen that all the $D_i$ are linearly equivalent
to one another and in this case, there is no need to impose the assumption of general type.
It is not clear whether the boundedness result holds other than the general type case, due to
$K3$ surfaces in Example \ref{k3}.
\end{rem}

Let $\mathcal{F}_n(\alpha, \beta)$ denote the set of the following objects up to
biholomorphic equivalence:
Let $\pi^*:\mathcal{X}^*\rightarrow \Delta^*$ be a smooth, holomorphic family
over the punctured disc, written as $\mathcal{X}^*$ for short,  such that every fiber $X_t$ is an $n$-dimensional projective variety
of nonnegative Kodaira dimension.
Further, assume that every fiber $X_t$ except
for $t$ in a countable union of proper analytic subsets in $\Delta^*$ admits an ample divisor $D_t$
(not necessarily continuously varying with $t$) such that
$$\text{$D_t^n \leq \alpha$ and
$(D_t^{n-1}\cdot K_{X_t}) \leq \beta$}.$$
Two such families  are in the same \emph{equivalence class} if they are
biholomorphic when restricted to a possibly smaller punctured disc.
Then the equivalence classes of $\mathcal{X}^*$ form the set $\mathcal{F}_n(\alpha, \beta)$.

Let $\mathcal{FG}_n(\alpha, \beta)$ denote the subset of $\mathcal{F}_n(\alpha, \beta)$ consisting of those families  $\mathcal{X}^*$
such that every fiber $X_t$ is of general type.
\begin{lemma}\label{fn}
Let $\mathcal{X}^*\rightarrow\Delta^*$ be a smooth (or holomorphic) family of $n$-dimensional projective manifolds of nonnegative Kodaira dimension
over the punctured disc.  Then $\mathcal{X}^*\rightarrow\Delta^*$ belongs to $\mathcal{F}_n(\alpha, \beta)$ for some $\alpha$ and $\beta$.
\end{lemma}
\begin{proof}[Proof of Lemma \ref{fn} (sketched)]  According to the Lebesgue negligibility argument in the proof of Lemma \ref{li-coin}, we
can produce a global line bundle $L$ on $\widetilde{\mathcal{X}^*}$, which is the family pulled back from $\mathcal{X}^*$ to
the universal cover $\widetilde{\Delta^*}$ of $\Delta^*$, with the property that $L$ restricts to an ample divisor $L_t$ in the fiber $\widetilde{X_t}$ of the family $\widetilde{\mathcal{X}^*}$ except
for $t$ in a countable union $B$ of proper analytic subsets in $\widetilde{\Delta^*}$.    Write $\alpha = L_t^n$ and $\beta = (L_t^{n-1}\cdot K_{\widetilde{X_t}})$.  We are done.
\end{proof}
\begin{rem}\label{deg-est}
By the effective very ampleness Theorem \ref{Dem-fuj}, we can
assume $L_t$ on $\widetilde{X_t}$ with $t$ outside $B$ to be very ample by replacing the original $L$ by $D=aL + bK_{\widetilde{\mathcal{X}^*}}$
for some universal constants $a, b$.  Then $D_t^n$ and $(D_t^{n-1}\cdot K_{\widetilde{X_t}})$ can be explicitly bounded
in terms of $\alpha, \beta$ and $a ,b$.  See \cite[Proposition 2.2]{lm} for $D_t^n$.
Here
$0 \leq (D_t^{n-1}\cdot b K_{\widetilde{X_t}})  = (D_t^{n-1}\cdot D_t)- (D_t^{n-1}\cdot aL_t)$
and  $(D_t^{n-1}\cdot L_t)$  can be explicitly bounded in terms of  $(L_t^{n-1}\cdot D_t) =a\alpha+b\beta$
(see \cite[2.15.8 Exercise of Chapter VI]{k96}).  Or, one may simply take
 $(D_t^{n-1}\cdot b K_{\widetilde{X_t}})  <  (D_t^{n-1}\cdot D_t)=D_t^n$
 for a cruder estimate.
\end{rem}

For Lemma \ref{fn} and Remark \ref{deg-est} (and the definition of $\mathcal{F}_n(\alpha, \beta)$), it is not strictly necessary to restrict oneself to the case of nonnegative Kodaira dimension; see Remark \ref{biganti}
 for the big anticanonical bundle case.

\begin{prop}\label{fgn}
There exist explicitly computable functions
$f(x, y, z, w)$ and $g(x, y, z, w)$ with the following property: For any semi-stable degeneration $\mathcal{X}\rightarrow \Delta$ with projective fibers and general fibers being of general type and its induced punctured family $\mathcal{X}^*\in \mathcal{FG}_n(\alpha, \beta)$ by Lemma \ref{fn}, there exists a countable union $B$ of proper analytic subsets in $\Delta$ such that for any $t$ outside it,
\begin{enumerate}[$(i)$]
\item \label{fgn-i}
 one has very ample divisors $D_0, D_1, \ldots $ on each  $X_t$ and the cardinality $|\{c_1(D_i)\}_i|$ is explicitly bounded in terms of $\alpha,\beta,\nu_n, n$, which is independent of $t\in B$;
\item \label{fgn-iii} there exists a (global) line bundle $L$ on $\mathcal{X}$ such that $L_t:=L|_{X_t}$ is very ample and satisfies
the universal estimates
$$L_t^n < f(\alpha, \beta, \nu_n, n)\ and\ ({L_t^{n-1}}\cdot K_{X_t}) < g(\alpha, \beta, \nu_n, n),$$
where  $\nu_n$ is the universal constant in Theorem \ref{takthm}.
\end{enumerate}
\end{prop}
\begin{proof}
We are going to combine the effectivity results for Chow varieties with our previous discussions.
By Remark \ref{deg-est} and Corollaries \ref{deg-cor}, \ref{eff-cor}, we can reach
very ample divisors $D_0, D_1, \ldots$ on $X_p$ for any $p$  outside  a countable union of proper analytic subsets in $\Delta^*$,
with degrees with respect to $D:=D_0$, explicitly bounded
in terms of $\nu_n$ (in Theorem \ref{takthm}) and $\alpha$, $\beta$.

Write $C^{n, k, d}(X_p)$ for the Chow variety associated with $k$-dimensional
irreducible subvarieties of $X_p$ with degree $d$ with respect to  the ample divisor $D$ of $X_p$,
or simply $\text{Chow}(X_p)$ if we omit the numerical datum.
There exist explicit (upper) bounds (independent of $X_p$) for the number of connected components of
$C^{n, k, d}(X_p)$ in terms of $n, k, d$, and thus only of $d$ if we drop $n$ and $k=n-1$, as shown by \cite[3.28 Exercise of Chapter I]{k96} or \cite[Proposition 3.1]{t98}.    Moreover, we have:
\begin{lemma}\label{c1}
  Two divisors $E$ and $F$ in a smooth projective variety $M$ have the same first Chern class
 if they are seated in the same connected component of $\text{Chow}(M)$.
\end{lemma}
\begin{proof}
Note that $c_1(E) = c_1(F)$ if and only if $E$ and $F$ are algebraically
equivalent. See  \cite[\S\ 19.3.1]{ful}.  Moreover,
if $E$ and $F$ are in the same connected component of $\text{Chow}(M)$,
then they are algebraically equivalent.  See \cite[(4.1.2.3) and (4.1.3) of Chapter II]{k96}.
\end{proof}

Continuing with the proof of Proposition \ref{fgn},
by Lemma \ref{c1} and the bounds on Chow varieties, we conclude that the
cardinality $|\{c_1(D_i)\}_i|$ is explicitly bounded in terms of the bound on degrees $\deg_D D_i$ of $D_i$ with respect to the (very) ample divisor $D$ of $X_p$ and thus in terms of $\alpha,\beta,\nu_n, n$. Using this, we are going to prove the existence of the desired global line bundle on $\mathcal{X}$.

Now $\{c_1(D_i)\}_i$ on each $X_t$ (for $t$ outside $B$ of $\Delta$) is finite, and
recall that we have obtained
a monodromy-invariant subset $S$ of $\{c_1(D_i)\}_i$, as shown in the proof of Theorem \ref{inv-cyc}.  Consider
an integral class formed by summing all elements in $S$, denoted by $\mathbf{C}$, which is then invariant under the monodromy action.
It follows from the local invariant cycle Theorem \ref{lict} that $\mathbf{C}$ arises from a class in $H^2(\mathcal{X}, \mathbb{Z})$,
denoted by $\mathcal{C}$. The remaining argument here is similar to that in the proof of Proposition \ref{global-B}. Thus, a class $\tilde{\mathcal{C}} \in H^2(\mathcal{X}, \mathbb{Z})$ which is zero when restricted to at least uncountable fibers, can be reached. By identity
theorem and the long exact sequence in \eqref{les-1} with $\mathcal{X}_U$ just being $\mathcal{X}$,
one claims that this $\tilde{\mathcal{C}}$ goes to $0$ in $H^2(\mathcal{X}, \mathcal{O}_\mathcal{X})$; for this argument, the needed deformation invariance
of Hodge numbers is justified in Proposition \ref{01-invariance}.
This yields, by the same exact sequence,
the existence of a global line bundle ${L}$ on $\mathcal{X}$ with $c_1({L})=\tilde{\mathcal{C}}$.
By Nakai--Moishezon's criterion of ampleness and Barlet's theory of cycle spaces \cite{bar}, the ampleness property of ${L}$ on the fibers $X_t$ is readily obtained, except possibly
for $t$ in a countable union of proper analytic subsets in $\Delta$. So is the very ampleness for the modification of $L$, still denoted by $L$, according to Theorem \ref{Dem-fuj}.

It is easy to estimate the degree of ${L}_t$.  By \cite[Proposition 2.2]{lm} (or
\cite[Lemma 1.4]{t97}), one estimates ${L}_t^n$ in terms of
$({L}_t\cdot  D^{n-1})$ which is bounded above by the product of the cardinality $|\{c_1(D_i)\}_i|$ and the bound on degrees of $D_i$ with respect to $D$.
It follows that ${L}_t^n$ is explicitly bounded in terms of the bound on $\deg_D D_i$.
Another explicit bound $({L}_t^{n-1}\cdot K_{X_t})$ can be obtained in a way similar to
that used in Remark \ref{deg-est}.   The asserted numerical results can now
be easily verified.

The proof of Proposition \ref{fgn} is completed.
\end{proof}
\begin{rem}\label{aee}
The global line bundle in the proof of Lemma \ref{fn} may not give the
desired bundle for Proposition \ref{fgn}, because it is defined on $\widetilde{\mathcal{X}^*}$ rather than
on $\mathcal{X}$. Moreover, together with Theorem \ref{0thm-moishezon}, Proposition \ref{fgn} also implies that the linear system $|L|$ gives rise to a bimeromorphic embedding $\Phi$ of $\mathcal{X}$ into $\mathbb{P}^N\times\Delta$ for some positive integer $N$, and
the restriction $\Phi|_{X_t}$ on $X_t$ is a biholomorphic embedding
for $t$ possibly outside a (countable) proper analytic subset of $\Delta$ (since the ramification divisor and the indeterminacies in $\mathcal{X}$ are jointly projected to a proper analytic subset of $\Delta$).
 Here $N$ can be
explicitly estimated via \eqref{ee} (see \cite[comments on p. 230]{kma}).
\end{rem}

\begin{rem}\label{biganti}
For the big anticanonical bundle case, one can also prove an analogue
of Proposition \ref{fgn} by the similar argument as above, as long as one replaces Theorem \ref{takthm} by \eqref{16-3} (using $K^*_{X_t}$ in place of $\tilde L_t$) whose estimate is comparatively not an effective one.
\end{rem}
\subsection{Two Chow-type lemmata} \label{subs-Chow}
Inspired by Demailly's Question \ref{quest}, we study the structure of the family of projective complex analytic spaces, to obtain two Chow-type lemmata. The full strength of the techniques developed in
previous subsections finds an application here; see also Remark \ref{3.21} for another application.
Recall that Chow's lemma is a multi-purpose tool to extend
results from the projective situation to the proper situation in general, and one particular case of it says that a proper $\mathbf{k}$-variety over a field $\mathbf{k}$ admits a projective birational morphism from a projective $\mathbf{k}$-variety, cf. \cite[\S\ 18.9]{va} or \cite[$\S$\  18.3]{ful}.

Let us start with the smooth case.
\begin{lemma}\label{Chow-sm}
Let $\pi:\mathcal{X}\rightarrow \Delta$ be a holomorphic family of $n$-dimensional projective manifolds. Then there exists a projective morphism $\pi_\mathcal{Y}:\mathcal{Y}\rightarrow \Delta$ from a complex manifold and a bimeromorphic morphism $\rho: \mathcal{Y}\rightarrow \mathcal{X}$ over $\Delta$ such that  possibly after shrinking $\Delta$, $\rho$ is a biholomorphism between $\mathcal{Y}\setminus{Y_0}$ and $\mathcal{X}\setminus{X_0}$ with $Y_0:=\pi_\mathcal{Y}^{-1}(0)$ and $X_0:=\pi^{-1}(0)$.
In particular, there exists a global line bundle $\mathcal{A}$ on $\mathcal{X}$ such that $\mathcal{A}_t:=\mathcal{A}|_{X_t}$ is very ample for any $t\ne 0$ possibly after shrinking $\Delta$.
\end{lemma}
\begin{proof} We have obtained a global line bundle $L$ on $\mathcal{X}$ such that  $L_t:=L|_{X_t}$ is ample for the very generic $t\in \Delta$ by the argument of \cite[Lemma 3.6]{rt} or in the proof of Lemma \ref{li-coin} here.  Assume that $L_0:=L|_{X_0}$ is non-ample; if $L_0$ is ample, then $L_t$ are ample for all $|t|$ small and it follows that $\mathcal{X}\rightarrow \Delta'$ is projective on a smaller disc $\Delta'$.
Theorem \ref{Dem-fuj} implies that for the very  generic $t\in \Delta$, there exist a universal constant $k:=k(n)$ such that $K_{X_t}\otimes  L^{\otimes k}_t$ is ample and $K_{X_t}^{\otimes 2}\otimes  L^{\otimes k}_t$ is very ample,  and thus the vanishing
$$H^i(X_t,K_{X_t}^{\otimes 2}\otimes  L^{\otimes k})=0,\ \text{for $i>0$}.$$
So the invariance of the Euler--Poincar\'{e} characteristic \cite[Theorem 4.12.(iii) of Chapter III]{bs} gives that of $h^0(X_t, K_{X_t}^{\otimes 2}\otimes  L^{\otimes k})$ for the very  generic $t\in \Delta$, denoted by $d$, and Grauert's semi-continuity implies that $$\{t\in \Delta: h^0(X_t, K_{X_t}^{\otimes 2}\otimes  L^{\otimes k}) \ge d+1\}$$ is a proper analytic subset of $\Delta$. Take $\mathcal{A}$ as $K_{\mathcal X}^{\otimes 2}\otimes  L^{\otimes k}$. Shrinking $\Delta$ if necessarily, we can now assume that on $\Delta^*$, $h^0(X_t, \mathcal{A}_t) =d$ and $\mathcal{A}$ is cohomologically flat in dimension $0$ by Lemma \ref{gct}.

By Theorem \ref{0thm-moishezon},  $\Phi$ associated to $\mathcal{A}$ gives a bimeromorphic embedding of $\mathcal{X}$ to $\mathbb{P}^N\times \Delta$ (over $\Delta$).
That $\Phi$ is bimeromorphic means that $\Phi$ is biholomorphic on $\mathcal{X}\setminus \mathcal{S}$ for some subvariety $\mathcal{S}$.
 Here $\pi(\mathcal{S})$ is a proper analytic subset of $\Delta$ since $\Phi$ is biholomorphic
on $X_b$, where $b$ is as in Step \ref{step 1} of Theorem \ref{0thm-moishezon}, thanks to the very ampleness, and thus on a neighborhood of $X_b$ with respect to the complex topology. Though $\mathcal{A}_t$, being very ample for very  generic $t$, readily gives
some biholomorphic embedding $\phi_t$ of $X_t$ (fiberwise), the $\Phi$ above is
to connect these $\phi_t$ together.  So $\pi(\mathcal{S}) \subset \{t=0\}$  possibly after shrinking $\Delta$, that is, $\Phi$ is biholomorphic at least on $\mathcal{X}^*$ by the definition of $\mathcal{S}$.
Then
$\mathcal{W}:=\Phi(\mathcal{X})$ is smooth outside $W_0:=\pi^{-1}_{\mathcal{W}}(0)$ for $\pi_{\mathcal{W}}:\mathcal{W}\rightarrow \Delta$ since the biholomorphism holds there and $\mathcal{X}$ is smooth. Conversely, $\mathcal{W}$ is bimeromorphic to $\mathcal{X}$ and $\pi_{\mathcal{W}}:\mathcal{W}\rightarrow \Delta$ is a projective morphism. By Remmert's elimination of indeterminacy \cite{re} (or \cite[Theorem 1.9]{pe}) and Hironaka's Chow lemma \cite[Corollaries 2 and 1]{Hi}, one modifies $\mathcal{W}$ as a complex manifold  $\mathcal{Y}$ such that the composition $\rho: \mathcal{Y}\rightarrow \mathcal{X}$ over $\Delta$ is a bimeromorphic morphism with $\mathcal{Y}\setminus{Y_0}\cong \mathcal{W}\setminus{W_0}$ as $\mathcal{W}\setminus{W_0} \cong \mathcal{X}\setminus{X_0}$ and $\pi_\mathcal{Y}: \mathcal{Y}\rightarrow \Delta$ is the desired projective morphism. This completes the first part of the lemma.

Finally, we discuss the second half of the lemma.  By Lemmata \ref{ccs-bc}, \ref{2-5} and the cohomological flatness of $\mathcal{A}$, $\pi_* \mathcal{A}$ is a holomorphic vector bundle on $\Delta^*$ and  its fibers are identified with $H^0(X_t,\mathcal{A}_t)$, while $\Phi|_{X_t}$ is given by  the section of $\mathcal{A}_t$ on $X_t$ for $t\ne 0$. This implies that $\mathcal{A}_t$ gives a biholomorphism on $X_t$.  Thus for $t\ne 0$, $\mathcal{A}_t$ is very ample, proving the second half. We can also prove it as follows.
In the proof of the first part, $\mathcal{Y} \subset \mathbb{P}^N\times \Delta$, so $H\times \Delta|_\mathcal{Y}$  is a hyperplane section on $\mathcal{Y}$ where $H = \mathcal{O}_{\mathbb{P}^N}(1)$.  Then this divisor $\rho(H\times \Delta) \subset \mathcal{X}$ gives a global line bundle $\mathcal{A}$ on $\mathcal{X}$ with $\mathcal{A}_t$ very ample  for $t \ne 0$ since $\rho$ induces a biholomorphism $Y_t\rightarrow X_t$. Remark that the advantage of the first way is to give an estimate of $N$ for $\mathbb{P}^N$.
\end{proof}
\begin{rem}\label{4.26}
Without shrinking $\Delta$, the similar proof
above shows that there exists a global line bundle $\mathcal{A}$ on $\mathcal{X}$ such that
$\mathcal{A}_t:=\mathcal{A}|_{X_t}$ is very ample for $t$ outside a proper analytic subset of
$\Delta$. Compare this with \cite[Theorem 1.2.17]{[La]} which in the algebraic category asserts that given the existence of $\mathcal{A}$ on $\mathcal{X}$, if $\mathcal{A}_s$ is ample for some $s\in \Delta$ then $\mathcal{A}_t$
is ample for any $t$ in some Zariski open neighborhood of $s$. If one uses Barlet's theory of
cycle spaces \cite{bar}, $\mathcal{A}_t$ is ample only for $t$ outside a countable union
of proper analytic subsets of $\Delta$.
\end{rem}
The following is an application of Lemma \ref{Chow-sm} (without using
Barlet's theory \cite{bar}):
\begin{cor}
With the notations of Lemma \ref{Chow-sm}, let $\mathcal{E}$ be any line bundle on $\mathcal{X}$ such that $\mathcal{E}_{t_0}:=\mathcal{E}|_{X_{t_0}}$ is
nef for some $t_0\in \Delta$.  Then $\mathcal{E}_t$ is nef for $t$ outside some thin
subset $T$ of $\Delta$ (i.e., $T$ is a countable union of proper analytic subsets of $\Delta$).
\end{cor}
\begin{proof}
See \cite[Proposition 1.4.14]{[La]} for a version in the algebraic category (i.e., $\mathcal{X}$ is a complete variety over a complete base as a family,
or $\mathcal{X}$ is dominated by some quasi-projective variety,  cf. \cite[Definition 18.3]{ful}).
Let $\mathcal{A}$ be the line bundle on $\mathcal{X}$ as in Lemma \ref{Chow-sm} and Remark \ref{4.26} such that
$\mathcal{A}_t$ is (very) ample for $t$ outside a proper analytic subset $Z$  of $\Delta$.
First suppose $\mathcal{E}$ with $\mathcal{E}_{t_0}$ being nef for some $t_0\in \Delta\setminus Z$.
Now that $(\mathcal{E}+\frac{1}{k} \mathcal{A})_{t_0}\ (k\in \mathbb{N})$ is ample, $(\mathcal{E}+\frac{1}{k} \mathcal{A})_{t}$ is then ample for $t$ outside a proper analytic subset $S_k$
of $\Delta$ by invoking \cite[Theorem 1.2.17]{[La]} (whose techniques can
be easily adapted to the analytic category here).  Clearly, for
$t$ outside $T:=\cup_{k\in \mathbb{N}} S_k$, $\mathcal{E}_t$ is nef as claimed. Next, if $t_0\in Z$, we use the projective morphism
$\mathcal{Y}\rightarrow \Delta$ of Lemma \ref{Chow-sm} and Remark \ref{4.26} to arrive at the desired conclusion
in a similar way, since with $\rho: \mathcal{Y} \rightarrow \mathcal{X}$ the nefness of $\rho^*L$ on a fiber of $\mathcal{Y}$
is equivalent to that of a line bundle $L$ on the corresponding fiber of $\mathcal{X}$ (cf. \cite[Example 1.4.4]{[La]}).
\end{proof}

Then we come to the degeneration case and first introduce a new notion stronger than the algebraic morphism (or Moishezon) family.
\begin{defn}\label{su-proj}
A family $\pi_\mathcal{W}:\mathcal{W}\rightarrow \Delta$ is a \emph{pseudo-projective family} if there exists a projective family
$\pi_\mathcal{P}:\mathcal{P} \rightarrow \Delta$ from a complex manifold and a bimeromorphic morphism from $\mathcal{P}$ to $\mathcal{W}$ over $\Delta$ such that it induces a biholomorphism $\pi_\mathcal{P}^{-1}(U)\rightarrow \pi_\mathcal{W}^{-1}(U)$ for some $U \subset \Delta$ with $\Delta\setminus U$ being a proper analytic subset of $\Delta$.
\end{defn}

\begin{lemma}\label{Prop. B}
Let $\pi:\mathcal{X}\rightarrow \Delta$ be a one-parameter degeneration and $\pi_{\mathcal{S}}:\mathcal{S}\rightarrow \Delta$ a semi-stable reduction of $\pi:\mathcal{X}\rightarrow \Delta$. Suppose that $\pi_{\mathcal{S}}:\mathcal{S}\rightarrow \Delta$ is pseudo-projective.
Then so is $\pi:\mathcal{X}\rightarrow \Delta$.
\end{lemma}
\begin{proof}
By Definition \ref{su-proj} of pseudo-projective family, $\pi_{\mathcal{S}}:\mathcal{S}\rightarrow \Delta$ has a projective family model $\pi_{\mathcal{S}'}:\mathcal{S}'\rightarrow \Delta$ in $\mathbb{P}^N\times \Delta$ for some $N\in \mathbb{N}$ and a bimeromorphic morphism $\nu:\mathcal{S}'\rightarrow\mathcal{S}$ over $\Delta$.
Take the hyperplane $H$ of $\mathbb{P}^N$ and $H\times \Delta$ cuts $\mathcal{S}'$ to get $H'\subseteq \mathcal{S}'$. Obtain a hyperplane section
of $\pi_{\mathcal{S}'}^{-1}(t)$ for $t$ outside a proper analytic subset $B_1\in \Delta$.
Now choose generic $H$ and the above gives $S_t$ of $\mathcal{S}$ a
hyperplane section $D_t\subset S_t$ for $t$ outside $B:=B_1\cup B_2$. Here
$B_2$ is the proper analytic subset of $\Delta$ outside which the induced biholomorphism
exists between the open subsets of $\mathcal{S}$ and $\mathcal{S}'$ as in Definition \ref{su-proj}.

Consider $\mu_*{D_t}$ (cf. \cite{ful} for this push-forward map) for the semi-stable reduction $\mu: \mathcal{S} \rightarrow \mathcal{X}$ over $\Delta$, say an $m$-to-$1$ sheeted fibration.  Since $\mu$ is biholomorphic for any general $t\in \Delta$ and $D_t$ for any $t\in \Delta\setminus B$ is very ample,
$\mu_*{D_t}$ for any nonzero $t = t_1, t_2,\ldots,$ $t_m$ with $t_i^m =\tilde{t}$ of
$\triangle\setminus (B\cup \{0\})$ is actually the sum of $m$ very ample divisors $F_1, F_2, \ldots, F_m$
on $\mu(S_t) = X_{\tilde{t}}$ with
$S_t =\pi^{-1}_\mathcal{S}(\tilde{t})$ and $F_i = \mu_*D_{t_i}$.
 It is well-known that a sum of very ample divisors is still very ample.

As above, we already have the global $H'\subseteq \mathcal{S}'$ and thus a global divisor $(\mu\circ\nu)_*H'$ on $\mathcal{X}$.
So $(\mu\circ\nu)_*H'$ restricts to be very ample on any $X_t$ for $t\in \Delta\setminus (B\cup\{0\})$, and induces a bimeromorphic embedding of $\mathcal{X}$ into some $\mathbb{P}^N$ (as in the proof of Lemma \ref{Chow-sm}), whose image is $\mathcal{X}'$. Then the remaining  is similar to the proof of Lemma \ref{Chow-sm}, to construct the desired bimeromorphic morphism $\mathcal{X}'\rightarrow \mathcal{X}$.
\end{proof}
\begin{rem}\label{Prop B'}
One can obtain a similar result for algebraic morphism version of Lemma \ref{Prop. B}, for which one formulates Theorem \ref{0thm-moishezon} up to flat family case.
The proof just needs that the sum of big divisors is still big by the characterization of big line bundle
using strictly positive curvature current.
\end{rem}

\begin{lemma}\label{sing-chow}
If $\pi:\mathcal{X}\rightarrow \Delta$ is a one-parameter degeneration with projective fibers and uncountable fibers are all of general type or all have big anticanonical bundles, then $\pi$ is also a pseudo-projective family.
\end{lemma}
\begin{proof}  We first prove the general type case. It follows from Proposition \ref{ki-dim-limit} that the general smooth projective fibers of $\pi$ are of general type. Then by Lemmata \ref{bu-proj}, \ref{Prop. B} and upon a semi-stable reduction, one assumes that the one-parameter degeneration $\pi$ is a semi-stable degeneration with projective fibers such that the canonical bundles  $K_{X_t}$ of the projective manifolds $X_t$ are big for all $t \in \Delta^*$.
So by Proposition \ref{fgn} and Remark \ref{4.26}, there exists a (global) line bundle $\mathcal{L}$ on $\mathcal{X}$ such that  $\mathcal{L}|_{X_t}$ for the very  generic $t\in \Delta$ is (very) ample, where the local invariant cycle Theorem \ref{lict} plays an essential role.
Then we proceed almost the same as the smooth family case in Lemma \ref{Chow-sm} with $L$ there replaced by $\mathcal{L}$, due to (use of the proof of) Theorem \ref{0thm-moishezon} (which is originally a smooth family version) adapted to the one-parameter degeneration case.

Next, we come to the big anticanonical bundle case. This is seen to be similar to the general type case by Proposition \ref{fgn} and Remark \ref{biganti}, and we omit the details.
\end{proof}

\begin{rem}\label{Remark 4.30}
 The total space of the family $\mathcal{X}\rightarrow \Delta$ is so far
in this subsection assumed to be smooth, although a degeneration
of the family is allowed.  It is expected that a generalization of the Chow-lemma
type result to certain singular $\mathcal{X}$ (as total space) is possible, e.g. the space of the family studied in \cite{kk}. We hope to come to this in future publication.
\end{rem}

\begin{rem}\label{Remark 4.33}  To compare our results here with the classical Chow's lemma in the
algebraic setting, let $\mathcal{X}\rightarrow \Delta$ be given as in Lemma \ref{Chow-sm} or \ref{sing-chow}.
A completely analogous ``analytic Chow's lemma", if exists, may assert the existence
of a projective morphism $\mathcal{Y}\rightarrow \Delta$ with a bimeromorphism $\rho: \mathcal{Y}\rightarrow\mathcal{X}$ over $\Delta$, such that for every fiber $X_t$ there is a bimeromorphic map
$Y_t\dashrightarrow X_t$ for some subvariety $Y_t$ of $\mathcal{Y}$ (cf. \cite[Definition 18.3]{ful}).   In contrast, the construction of a projective morphism
$\mathcal{Y}\rightarrow \Delta$ in this subsection gives that $\mathcal{Y}$ satisfies $Y_t \rightarrow X_t$
biholomorphically for $t$ outside a proper analytic subset of $\Delta$.
\end{rem}

\appendix
\section{Bimeromorphic embedding}\label{bim-em}
In the proof for Theorem \ref{0thm-moishezon} on bimeromorphic embedding, we will often apply the following preliminaries.
The first ones are proper modification and meromorphic map as shown in the nice reference \cite[$\S$ 2]{Ue} on bimeromorphic geometry or \cite[\S\ 9.3]{DG94}. Notice that a complex variety here is not necessarily compact.
\begin{defn}[{\cite[Definition 2.1]{Ue}}]\label{modification}
A morphism $\pi: \tilde{X}\rightarrow X$ of two complex varieties is called a \emph{proper modification}, if it satisfies:
\begin{enumerate}
  \item [$(i)$] $\pi$ is proper and surjective;
  \item [$(ii)$] there exist nowhere dense analytic subsets $\tilde E\subseteq \tilde{X}$ and $E \subseteq X$ such that
                  $$
                  \pi:\tilde{X}-\tilde E\rightarrow X-E
                  $$
                  is a biholomorphism, where $\tilde E:=\pi^{-1}(E)$ is called the \emph{exceptional space of the modification}.
\end{enumerate}
If $\tilde X$ and $X$ are compact, a proper modification $\pi: \tilde{X}\rightarrow X$ is often called simply a \emph{modification}.
\end{defn}

More generally, we have the following definition.
\begin{defn}[{\cite[Definition 2.2]{Ue}}]\label{bimero}
Let $X$ and $Y$ be two complex varieties.
A map $\varphi$ of $X$ into the power set of $Y$ is a \emph{meromorphic map} of $X$ into $Y$,
denoted by $\varphi: X\dashrightarrow Y$, if $X$ satisfies the following conditions:
\begin{enumerate}
  \item [$(i)$] The graph ${\mathcal{G}}({\varphi}):=\{(x,y)\in X\times Y: y\in \varphi(x)\}$ of $\varphi$ is an irreducible analytic subset in $X\times Y$;
  \item [$(ii)$] The projection map $p_X:{\mathcal{G}}({\varphi})\rightarrow X$ is a proper modification.
\end{enumerate}

A meromorphic map $\varphi: X\dashrightarrow Y$ of complex varieties is called a \emph{bimeromorphic map} if $p_Y:{\mathcal{G}}({\varphi})\rightarrow Y$ is also a proper modification.

If $\varphi$ is a bimeromorphic map, the analytic set
$$
\{(y,x)\in Y\times X: (x,y)\in {\mathcal{G}}(\varphi)\}\subseteq Y\times X
$$
defines a meromorphic map $\varphi^{-1}:Y\dashrightarrow X$ such that $\varphi\circ\varphi^{-1}=id_Y$ and $\varphi^{-1}\circ\varphi=id_X$.

Two compact complex varieties $X$ and $Y$ are called \emph{bimeromorphically equivalent} (or \emph{bimeromorphic}) if there exists a bimeromorphic map $\varphi: X\dashrightarrow Y$.

Obviously, the notion of bimeromorphic map is naturally valid for general complex spaces (e.g. \cite[\S\ $2$ and $3$]{st}). More specially, see also the notions of rational and  birational maps  on
\cite[pp. 490-493]{Griffith} in the algebraic setting.
\end{defn}

Then we  need two more remarkable theorems.
\begin{thm}[{Theorem A of Cartan, \cite[Theorem 7.2.8]{h90}}]\label{cartan-a}
Let $\Omega$ be a Stein manifold and $\mathcal{F}$ a coherent analytic
sheaf on $\Omega$. For every $z\in \Omega$, the $\mathcal{O}_z$-module $\mathcal{F}_z$ is then generated by the
germs at $z$ of the sections in $\Gamma(\Omega, \mathcal{F})$.
\end{thm}

\begin{thm}[{Remmert's proper mapping theorem, \cite[Theorem 2.11 of Chapter III]{bs}}] \label{remmert}
The image of a closed analytic set by a proper morphism
is a closed analytic set.
\end{thm}

Finally, we will also be much concerned with questions of irreducibility and properness in the proof for Theorem \ref{0thm-moishezon} on bimeromorphic embedding and
refer the readers to \cite[\S\ $1$ and $2$ of Chapter $9$]{Grr} for a nice introduction.

Recall that in \cite{rt}, we proved the following theorem with four steps, the first three of which can be exactly used to prove the desired bimeromorphic embedding in Theorem \ref{0thm-moishezon}. For reader's convenience, we include a slightly modified version here.

\begin{thm}[{\cite[Theorems 1.4+4.27]{rt}}]\label{thm-moishezon}
Let $\pi: \mathcal{X}\rightarrow \Delta$ be a holomorphic family of compact complex manifolds. If the fiber $X_t$  is Moishezon for each nonzero $t$ in an uncountable subset $B$ of $\Delta$,
with $0$ not necessarily being a limit point of $B$, and every fiber $X_t$ satisfies the local deformation invariance for Hodge number of type $(0,1)$ or admits a strongly Gauduchon metric, then:
\begin{enumerate}[$(i)$]
\item \label{thm-moishezon-i}
$X_t$ is still Moishezon for any $t\in \Delta$.
\item  \label{thm-moishezon-ii}
For some $N\in \mathbb{N}$, there exists
a bimeromorphic map
$$\Phi:\mathcal{X}\dashrightarrow\mathcal{Y}$$
to a subvariety $\mathcal{Y}$ of $\mathbb{P}^N\times\Delta$
with every fiber $Y_t\subset\mathbb{P}^N\times\{t\}$ being a projective
 variety of dimension $n$, and also a proper analytic set $\Sigma\subset\Delta$, such that $\Phi$ induces a bimeromorphic map
 $$\Phi|_{X_t}:X_t\dashrightarrow Y_t$$
 for every
$t\in\Delta\setminus \Sigma$.

\end{enumerate}
\end{thm}

\begin{proof}[Proof of Theorem \ref{0thm-moishezon}]

\newtheorem{step}{Step}
\renewcommand{\thestep}{$($\Roman{step}$)$}
\begin{step}\label{step 1}
Global line bundle $L$ and Kodaira map $\Phi$
\end{step}
By assumption, we have arrived at
a global line bundle $L$ on $\mathcal{X}$ such that $L|_{X_t}$ is big
for every $t\in \Delta^*$. Moreover,
the uniform estimate similar to \eqref{16-3} holds for $L|_{X_t}$ with every $t\in\Delta$.
To proceed further, a difficulty arises. On each fiber $X_t$, there is some $q_t\in\mathbb{N}$ such
that $H^0(X_t, L|_{X_t}^{\otimes q_t})$ can induce a bimeromorphic embedding of $X_t$; $q_t$ may depend on $t$,
however.  We do not know how to control it even though the uniform estimate \eqref{16-3} holds here.
Our methods of overcoming the difficulty consist in the study of bimeromorphic geometry of
$\mathcal{X}\to\Delta$ as a family; the complex analytic desingularization \cite{ahv} also plays a useful role
in the process.

We can choose a point $b\in \Delta^*$ such that  for every $q\in \mathbb{N}$,
$h^0(X_t, L^{\otimes q}|_{X_t})$ is locally constant in some neighborhood (dependent on $q$) of $b$.
This is because for a given $m\in \mathbb{N}$,
the set of points in $\Delta$ where $h^0(X_t, L^{\otimes m}|_{X_t})$ fails to be locally constant,
is at most countable (cf. Theorem \ref{Upper semi-continuity}).
We fix such a $b\in \Delta$.  For this fiber $X_b$ at $b$, there exists a $\tilde{q}\in \mathbb{N}$
such that the Kodaira map associated with the complete linear system
$|L^{\otimes \tilde{q}}|_{X_b}|$ gives a bimeromorphic embedding of $X_b$ since $L|_{X_b}$ is big.
Now the preceding local invariance of $h^0(X_t, L^{\otimes \tilde{q}}|_{X_t})$ at $b$ yields
that the natural map
$$(\pi_*L^{\otimes \tilde{q}})_b\to H^0(X_b, L^{\otimes \tilde{q}}|_{X_b})$$
is surjective;
see Lemma \ref{ccs-bc}.  By Theorem A of Cartan (= Theorem \ref{cartan-a}),
one can choose $E$ linearly spanned by $$\{s_0, s_1, s_2,\ldots, s_N\}\subset \pi_*L^{\otimes \tilde{q}}(\Delta)$$ whose germs generate
$(\pi_*L^{\otimes \tilde{q}})_b$.   To sum up, by the identification $\pi_*L^{\otimes \tilde{q}}(\Delta)\cong
H^0(\mathcal{X}, L^{\otimes \tilde{q}})$,  the Kodaira map associated with the above $E$, denoted by
$$\Phi: \mathcal{X}\dashrightarrow \mathbb{P}^N\times\Delta: x\mapsto\big([s_0(x):s_1(x):\cdots :s_N(x)], \pi(x)\big),$$
is meromorphic
on $\mathcal{X}$ and bimeromorphic on $X_b$.  Here as usual, the concept of meromorphic (or bimeromorphic) maps is understood in the sense of Remmert; e.g. see Definition \ref{bimero} for more.   For this claim it is
stated in \cite[Example 2.4.2, p. 15]{Ue} for compact varieties.  If the target is a projective space,
it is treated in \cite[pp. 490-493]{Griffith} or \cite{s75};
a variant of which for our need ($\mathcal{X}$ is noncompact and  the target  is not a projective space
 but $\mathbb{P}^N\times\Delta$) is given.
Let $i$ denote the composite map $$\mathbb{P}^N\times\Delta\subset\mathbb{P}^N\times\mathbb{P}^1\to \mathbb{P}^{2N+1}$$
where the second map is the Segr\'e embedding $$([x_0:x_1:\cdots :x_N], [1:t])\mapsto [x_0:x_0t:x_1:x_1t:\cdots :x_N:x_Nt]$$
 (cf. \cite[p. 192]{Griffith}), and
 $$\tilde i:\mathcal{X}\times (\mathbb{P}^N\times\Delta)\to \mathcal{X}\times\mathbb{P}^{2N+1}:(z^1, z^2)\mapsto (z^1, i(z^2)).$$
Associated with the sections $$\{s_0, s_0\tilde t, s_1, s_1\tilde t, \ldots, s_N, s_N\tilde t\}\subset H^0(\mathcal{X}, L^{\otimes \tilde{q}}\otimes_{\mathcal{O}_\mathcal{X}}\pi^*\mathcal{O}_\Delta)$$
where $\tilde t\in H^0(\mathcal{X}, \pi^*\mathcal{O}_\Delta)$ denotes $\pi^*t$,
the composite map
$$\Phi_1:=i\circ \Phi:\mathcal{X}\dashrightarrow \mathbb{P}^{2N+1}$$
is a meromorphic map by \cite[pp. 490-492]{Griffith}
so that  with the graph  ${\mathcal{G}}_1\subset \mathcal{X}\times\mathbb{P}^{2N+1}$  of $\Phi_1$, one expects
$$\mathcal{G}:={\tilde i}^{-1}({\mathcal{G}}_1\cap \tilde i (\mathcal{X}\times(\mathbb{P}^N\times\Delta)))={\tilde i}^{-1}({\mathcal{G}}_1)\subset \mathcal{X}\times(\mathbb{P}^N\times\Delta)$$ to serve
as the graph of $\Phi$.  Indeed, this $\mathcal{G}$ is a complex space in $\mathcal{X}\times(\mathbb{P}^N\times\Delta)$.
The other conditions required for $\mathcal{G}$ as a graph variety of
a meromorphic map in Definition \ref{bimero} can
also be verified via ${\mathcal{G}}_1$.  It follows that
$$\Phi:\mathcal{X}\dashrightarrow \mathbb{P}^N\times\Delta$$
is a meromorphic map, as claimed.

Denote the respective projections by $$q_1:{\mathcal{G}}\to \mathcal{X}$$ and
$$q_2:=\pi_{\mathbb{P}^N\times\Delta}|_\mathcal{G}:{\mathcal{G}}\to \mathbb{P}^N\times\Delta$$
where $\pi_{\mathbb{P}^N\times\Delta}:\mathcal{X}\times (\mathbb{P}^N\times\Delta)\to \mathbb{P}^N\times\Delta$, and
set $$\mathcal{Y}:=q_2(\mathcal{G})=\Phi(\mathcal{X})\subset {\mathbb{P}^N\times\Delta}$$ (cf. \cite[p. 14]{Ue}).   The projection $\pi_{\mathbb{P}^N\times\Delta}$
is not proper.
But by Lemma \ref{proper} below, $\mathcal{Y}$ is a closed subvariety of $\mathbb{P}^N\times\Delta$
and as such $\mathcal{Y}\to\Delta$ is proper;
$$Y_t\subset\mathbb{P}^N\times\{t\}$$
 denotes the corresponding projective subvariety in $\mathcal{Y}$ seated at $t$.
 Since $\mathcal{X}$ is irreducible, so is $\mathcal{G}$ (cf. \cite[p. 13]{Ue}), and
$\mathcal{Y}=q_2(\mathcal{G})$ is thus irreducible.
Clearly $\mathcal{Y}$ is of dimension $n+1$.

Remark that some refinements of the above construction will be made in Step \ref{step 3} for our need in due course.

Some tools in what follows have counterparts in algebraic category, however we work mostly within analytic category.
The above $\Phi$ is actually a morphism outside a subvariety $\mathcal{S}(\Phi)$ of codimension at least two in
$\mathcal{X}$ (cf. \cite[the third paragraph on p. 333]{re}), giving that for every $t\in \Delta^*$, $X_t\not\subset \mathcal{S}(\Phi)$ by dimension reason.
Outside the analytic set $S_t:=\mathcal{S}(\Phi)\cap X_t$ of codimension at least one in $X_t$, the restriction
$$\Phi_t:=\Phi|_{X_t\setminus S_t}$$ is a morphism.   By \cite[pp. 35-36]{st},
$\Phi_t$ is still a meromorphic map on $X_t$.  For later references, we can do it in the following way.
Recall that the projection $$q_1:\mathcal{G}(\Phi)={\mathcal{G}}\to\mathcal{X}$$ from the graph
of $\Phi$, is a proper modification
(cf. Definition \ref{bimero}) so that $q_1^{-1}(X_t)\subset \mathcal{G}$ is an analytic subset of dimension
$n$.   Let  $\mathcal{C}$ be the unique irreducible component of $q_1^{-1}(X_t)$ which contains the graph of
$\Phi|_{X_t\setminus S_t}$, so that $\mathcal{C}\setminus q_1^{-1}(S_t)\cong X_t\setminus S_t$
biholomorphically and thus that $q_1|_{\mathcal{C}}:\mathcal{C}\to X_t$
is a proper modification (cf. Definition \ref{modification}).
Moreover, as for any subset $T\subset\mathcal{X}$,
$q_1^{-1}(T)=\mathcal{G}\cap(T\times \mathcal{Y})$, one has $\mathcal{C}\subset
q_1^{-1}(X_t)\subset X_t\times\mathcal{Y}$
hence that $\mathcal{C}$ induces a meromorphic map $\phi_{\mathcal{C}}:X_t\dashrightarrow\mathcal{Y}$
whose graph is precisely $\mathcal{C}$, since a meromorphic map is uniquely determined by an analytic subset $\mathcal{M}$ of $X_1\times X_2$ with two complex spaces $X_1,X_2$ which satisfies the condition $(M)$ that the projection $p_{X_1}:\mathcal{M}\rightarrow X_1$ is a proper modification as argued on \cite[p. 14]{Ue}.   Clearly $\phi_{\mathcal{C}}$ coincides
with $\Phi_t$ on the open subset $X_t\setminus S_t$.  Since $\phi_{\mathcal{C}}$ is meromorphic,
we now conclude that $\Phi_t:X_t\dashrightarrow \mathcal{Y}$ is a meromorphic map, as claimed.

We denote by $$\Phi|_{X_t}:X_t\dashrightarrow\mathcal{Y}$$ the meromorphic map associated with $\Phi_t$ as just shown.
 Since $\Phi|_{X_t}$ is now meromorphic
so that $\Phi|_{X_t}(X_t)$ is actually a closure of $\Phi|_{X_t}(X_t\setminus S_t)\subset Y_t$ in the
analytic set $Y_t$ (\cite[p. 493]{Griffith}),
we see that $\Phi|_{X_t}(X_t)\subset Y_t$ for every $t\in \Delta^*$.
However, it is not claimed that $\Phi|_{X_t}(X_t)$ equals $Y_t$. For $t=0$, the above argument and conclusion still apply with
minor modifications; we omit the details here.

We shall now see that $Y_t$ is of dimension $n$ for every $t\in \Delta$.
First note that, if $|b'-b|\ll 1$, $d\Phi|_{X_{b'}}$ is of rank $n$ at generic point of $X_{b'}$ since
it is so for $X_b$, hence that $\dim_{\mathbb{C}}Y_b=\dim_{\mathbb{C}} Y_{b'}=n$.
In the algebraic setting, applying the upper semi-continuity of dimension (cf. \cite[Corollary 3 in Section 8 of Chapter 1]{mum}),
one is allowed to conclude that $Y_t$ is of dimension $n$.  Alternatively, an approach suitable in
our analytic setting is described as follows.
Let $$\{H_i\}_{1\le i\le n}\subset\mathbb{P}^N$$
be any hyperplane sections and
$$\tilde H_i:=H_i\times\Delta.$$
The closed subvariety $\tilde H_1\cap \tilde H_2\cap\cdots\cap\tilde H_n\cap \mathcal{Y}$ is projected,
via $\mathcal{Y}\subset \mathbb{P}^N\times\Delta\to\Delta$, down to a
subvariety $W_1\subset\Delta$ by Remmert's proper mapping theorem in complex spaces (= Theorem \ref{remmert}),
which is therefore the whole $\Delta$, as follows from the fact that
$W_1$ must contain a small open subset of $\Delta$ by $\dim_{\mathbb{C}}Y_{b'}=n$.
If $\dim_{\mathbb{C}}Y_{t_0}<n$ for some $t_0\in \Delta$, by choosing
$H_i$ in general positions such that $$H_1\cap H_2\cap\cdots\cap H_n\cap Y_{t_0}=\emptyset$$ with
$Y_{t_0}\subset \mathbb{P}^N$ via identification $ \mathbb{P}^N\cong  \mathbb{P}^N\times \{t_0\}$, then
with these $H_i$, $\tilde H_1\cap \tilde H_2\cap\cdots\cap\tilde H_n\cap \mathcal{Y}$ is projected,
via $\mathcal{Y}\to\Delta$, to a subset of $\Delta$
missing $t_0$, contradicting the preceding $W_1=\Delta$.   Hence $\dim_{\mathbb{C}}Y_t\ge n$  for every $t\in\Delta$ so that
$\dim_{\mathbb{C}}Y_t=n$ since if
$\dim_{\mathbb{C}}Y_{t_1}\ge n+1$ for some $t_1\in\Delta$, it contradicts that
$\mathcal{Y}$ is irreducible and of dimension $n+1$ as
already indicated.

The following lemma has been used in the first half of this step. Notice that the projection $\mathcal{X}\times(\mathbb{P}^N\times\Delta)\to \mathbb{P}^N\times\Delta$ is usually not proper and we try to prove that its restriction to the graph of $\Phi$ is indeed proper.
\begin{lemma}\label{proper} With the notations as above, the projection morphism
$$q_2:\mathcal{G}\,(\subset \mathcal{X}\times(\mathbb{P}^N\times\Delta))\to \mathbb{P}^N\times\Delta$$ is
proper.  As a consequence, $q_2(\mathcal{G})$ is an analytic subvariety of $\mathbb{P}^N\times\Delta$.
\end{lemma}
\begin{proof} Since $\Phi:\mathcal{X}\dashrightarrow \mathbb{P}^N\times\Delta$ is a meromorphic map, $q_1:{\mathcal{G}}\to\mathcal{X}$
is a proper modification by Definition \ref{bimero} and  thus there exist two respective open dense
subsets $$U\subset \mathcal{G}\quad \text{and}\quad V\subset \mathcal{X},$$
which are biholomorphically equivalent under $q_1|_U: U\to V=q_1(U)$, such that $\Phi$ is a morphism on $V$
and $U=\{(x, \Phi(x)) \}_{x\in V}$ (cf. \cite[the remarks preceding Remark 2.3, p. 14]{Ue}).
Let $W_2\subset \mathbb{P}^N\times\Delta$ be a compact subset.  To prove by contradiction,
suppose that $q_2^{-1}(W_2)\subset \mathcal{G}$ is not compact.   Set the projections
$$\pi:\mathcal{X}\to\Delta\quad \text{and}\quad \pi_{\mathbb{P}^N\times\Delta}: \mathbb{P}^N\times\Delta\to \Delta.$$
Then under the projection $q_1:{\mathcal{G}}\to \mathcal{X}$,
$q_1(q_2^{-1}(W_2))$ is closed but not compact since $q_1$ is proper.
This means, since $\pi:\mathcal{X}\to\Delta$ is proper,
that there exists a sequence $t_k\in\Delta$ with $t_k\to \partial\Delta$ and
 $\{t_k\}_k\subset \pi(q_1(q_2^{-1}(W_2)))$.
We first show that $$\Phi(X_t)\subset \mathbb{P}^N_t:=\pi_{\mathbb{P}^N\times\Delta}^{-1}(t)$$ for every $t\in \Delta$.

As $\Phi(T)=\cup_{x\in T}\Phi(x)$ for a set $T\subset\mathcal{X}$,  we need to show that $\Phi(x)\in \mathbb{P}^N_t$
if $x\in X_t$.
For any $(x, y)\in q_1^{-1}(x)\subset \mathcal{G}$,
there exists a sequence
$(x_j, y_j)\in U\subset \mathcal{G}$, $(x_j, y_j)\to (x, y)$ (by that $U$ is dense)
with $x_j\in V=q_1(U)$ and $y_j=\Phi(x_j)$.
Thus, $$\lim_j x_j=\lim_j q_1(x_j, y_j) \to q_1(x, y)=x.$$
By definition $\Phi(x)=q_2(q_1^{-1}(x))$ as in the first paragraph on \cite[p. 14]{Ue}, we have just seen that any point
$y\in \Phi(x)=q_2(q_1^{-1}(x))$
is a limit point of
the form $\Phi(x_j)=y_j$ for some sequence $x_j\to x$ in $\mathcal{X}$ with $x_j\in V$.
In this case, since  $x_j\not\in \mathcal{S}(\Phi)$, i.e., $\Phi$ is a morphism at $x_j$,
one has $\pi(x_j)=\pi_{\mathbb{P}^N\times\Delta}\Phi(x_j)$ by construction of $\Phi$, which is $\pi_{\mathbb{P}^N\times\Delta}(y_j)$.
Write $\pi(x_j)=\pi_{\mathbb{P}^N\times\Delta}(y_j)=t_j$.  So $y_j\in  \mathbb{P}^N_{t_j}$ and if $x\in X_t$,
then $t_j\to t$ since $\pi(x_j)\to \pi(x)$.
In short, for any $x\in X_t$ and any $y\in \Phi(x)$, $y=\lim_jy_j\in \lim_j\mathbb{P}^N_{t_j}$ which is
$\mathbb{P}^N_t$ as $t_j\to t$.
That is $\Phi(X_t)\subset \mathbb{P}^N_t$ for every $t\in \Delta$, as claimed.

The remaining is standard.
Corresponding to every
$t_k\in \{t_k\}_k$
above, there is a $$(x_k, y_k)\in q_2^{-1}(W_2)\subset \mathcal{G}$$
with $t_k=\pi(q_1(x_k, y_k))=\pi(x_k)$ so $x_k\in X_{t_k}$.
By $y_k\in \Phi(x_k)$, $$\pi_{\mathbb{P}^N\times\Delta}(y_k)\in \pi_{\mathbb{P}^N\times\Delta}(\Phi(x_k))\in \pi_{\mathbb{P}^N\times\Delta}(\mathbb{P}^N_{t_k})=t_k$$
by $x_k\in X_{t_k}$ and $\Phi(X_{t_k})\subset \mathbb{P}^N_{t_k}$ above.   In short, $\pi_{\mathbb{P}^N\times\Delta}(y_k)=t_k$.
By $y_k=q_2(x_k, y_k)$ and $q_2(x_k, y_k)\in W_2$, $\pi_{\mathbb{P}^N\times\Delta}(y_k)\in \pi_{\mathbb{P}^N\times\Delta}(W_2)$ which
 is a compact subset in $\Delta$ since $W_2\subset \mathbb{P}^N\times\Delta$ is compact by assumption.
This contradicts  $\pi_{\mathbb{P}^N\times\Delta}(W_2)\ni\pi_{\mathbb{P}^N\times\Delta}(y_k)=t_k\to \partial\Delta$.
As said, the contradiction yields that $q_2:{\mathcal{G}}\to \mathbb{P}^N\times\Delta$ is
a proper morphism.

The second statement of the lemma follows from
Remmert's proper mapping theorem (cf. \cite[p. 395]{Griffith} or Theorem \ref{remmert}).
\end{proof}

\begin{rem}\label{proper-rem} In fact, the above proof works for the following situation.
Let $\pi_{Z_1}:Z_1\to\Delta$ and $\pi_{Z_2}:Z_2\to \Delta$ be proper morphisms where
$Z_1$, $Z_2$ be irreducible (and reduced) complex spaces.  Suppose that $\psi: Z_1\dashrightarrow Z_2$ is a meromorphic map
(in the sense of Remmert)
with $\mathcal{G}(\psi)\subset Z_1\times Z_2$ the irreducible subvariety of the graph of $\psi$.  Let $\emptyset\neq U\subset Z_1$ be an open dense subset such that $\psi$ is a morphism on $U$.  Suppose furthermore that $\pi_{Z_1}|_U=\pi_{Z_2}\circ \psi|_U$.  Then
the projection morphism ${\mathcal{G}}(\psi)\to Z_2$ is proper.
\end{rem}

The remaining proof is devoted to the bimeromorphic problem of $\Phi$.
\begin{step}\label{step 2}
Bimeromorphic embedding of $\Phi$
\end{step}  We shall use the notations
$$\Phi:\mathcal{X}\dashrightarrow\mathbb{P}^N\times\Delta\quad \text{and}\quad \Phi:\mathcal{X}\dashrightarrow\mathcal{Y}=\Phi(\mathcal{X})$$
interchangeably.   Let's start with a desingularization
$\mathcal{R}_{{\mathcal{Y}}}:\tilde{\mathcal{Y}}\to\mathcal{Y}$; it can be chosen as a proper modification.
For a review, see \cite[Theorem 2.12]{Ue} for compact cases and \cite[Theorem 7.13]{pe}
or \cite[Theorem 5.4.2, p. 271]{ahv} for general cases.
Note that if $h$ is a proper modification between complex spaces,
then $h^{-1}$ is still a meromorphic map \cite[p. 34]{st}
and hence $h$ is a bimeromorphism \cite[p. 33]{st}.
Write
$$\Psi:=\mathcal{R}_{{\mathcal{Y}}}^{-1}\circ\Phi:\mathcal{X}\dashrightarrow \mathcal{\tilde Y},$$
 which is still meromorphic (cf. \cite[pp. 16-17]{Ue}).
In algebraic cases,
if $g:X\to Y$ with $Y$ irreducible is a generically finite morphism such that
$g(X)$ is dense in $Y$ or equivalently $g$ is dominant, then there exists an open dense subset $U\subseteq Y$
such that the induced morphism $g^{-1}(U)\to U$ is a finite morphism (e.g. \cite[Exercise 3.7 of Chapter II]{Ht}).
If $g$ is only a generically finite dominant {\it rational} map, by going to its graph and restricting to the open dense subset $V\subset X$ which is the complement of the indeterminacies of $g$, one is reduced to the morphism case and
there is a similar conclusion.

Analytically, let's adopt a similar strategy here.
Write $\mathcal{S}(\Psi)$ for the indeterminacies of $\Psi$,
which is of codimension at least two in $\mathcal{X}$ (\cite[the third paragraph on p. 333]{re}).
Since $\mathcal{X}$ and $\tilde{\mathcal{Y}}$ are smooth and of the same dimension,
the ramification divisor $R_\Psi\subset {\mathcal{X}}$ is
well-defined. Namely, it is first defined outside $\mathcal{S}(\Psi)$ and then extends across it since $\dim_{\mathbb{C}}R_\Psi>\dim_{\mathbb{C}}\mathcal{S}(\Psi)$ by
Remmert--Stein extension theorem, cf. \cite[p. 293]{bis}.
Let
$$\mathcal{G}(\Psi)\subset \mathcal{X}\times\mathcal{\tilde Y}$$
denote
the graph of $\Psi$ with the projections $$p_{\mathcal{X}}:\mathcal{G}(\Psi)\to\mathcal{X}\quad \text{and}\quad
p_{\mathcal{\tilde Y}}: \mathcal{G}(\Psi)\to \mathcal{\tilde Y},$$ respectively.
Having proved Lemma \ref{py-proper} below that $p_{\mathcal{\tilde Y}}: \mathcal{G}(\Psi)\to \mathcal{\tilde Y}$ is proper, one knows that the image of an analytic set under $p_{\mathcal{\tilde Y}}$
is still analytic by the proper mapping theorem of Remmert (= Theorem \ref{remmert}).

Set $$\Psi(\mathcal{S}(\Psi)):
=p_{\tilde{\mathcal{Y}}}(p_{\mathcal{X}}^{-1}(\mathcal{S}(\Psi)))$$ which is a proper  analytic subvariety of $\tilde{\mathcal{Y}}$ and
similarly the subvariety
$$\Psi^{-1}(\Psi(\mathcal{S}(\Psi))):=p_{\mathcal{X}}(p_{\tilde{\mathcal{Y}}}^{-1}(\Psi(\mathcal{S}(\Psi))))\subset\mathcal{X};$$
also subvarieties $\Psi(R_\Psi)$, $\Psi^{-1}(\Psi(R_\Psi))$.
Write $$\Psi':\mathcal{X}\setminus \big(\Psi^{-1}(\Psi(\mathcal{S}(\Psi)))\cup \Psi^{-1}(\Psi(R_\Psi))\big)=:\mathcal{X}'
\to {\tilde{\mathcal{Y}}}':=\tilde{\mathcal{Y}}\setminus \big(\Psi(\mathcal{S}(\Psi))\cup \Psi(R_\Psi)\big)$$ and $\Psi'_t$
for its restriction to (an open part of) $X_t$, more precisely to
$X_{t}':=X_{t}\cap \mathcal{X}'$ with images in $\tilde Y'_{t}:={\tilde Y}_{t}\cap \tilde{\mathcal{Y}}'$
for those $X_{t}'\neq\emptyset$. Here $\tilde Y_t:=\pi_{\tilde{\mathcal{Y}}}^{-1}(t),$
where $$\pi_{\tilde{\mathcal{Y}}}:\tilde{\mathcal{Y}}\to \Delta$$ is the projection via $\tilde{\mathcal{Y}}\to\mathcal{Y}\to\Delta$.
By construction $\Psi'$ is a surjective morphism and since $d\Psi'$ is now of maximal rank everywhere,
$\Psi'$ is a local biholomorphism.
The $\mathcal{X}'$ and ${\tilde{\mathcal{Y}}}'$ can possibly be enlarged.  Suppose that
$x\in \mathcal{X}\setminus \mathcal{S}(\Psi)$ and $\Psi$ is a local
biholomorphism between the open neighborhoods $\mathcal{H}_x\ni x$ and $\mathcal{K}_{\Psi(x)}\ni \Psi(x)$.  Then
$\Psi|_{\mathcal{X}'\cup \mathcal{H}_x}:\mathcal{X}'\cup \mathcal{H}_x\to {\tilde{\mathcal{Y}}}'\cup \mathcal{K}_{\Psi(x)}$ is still surjective and a local biholomorphism.
By enlarging $\mathcal{X}'$ and ${\tilde{\mathcal{Y}}}'$ this way, we can assume that if $\Psi$ is a local biholomorphism
at $x'$, then $x'\in \mathcal{X}'$ and $\Psi(x')\in{\tilde{\mathcal{Y}}}'$.  Here, $\mathcal{X}'\subset\mathcal{X}$ and
${\tilde{\mathcal{Y}}}'\subset{\tilde{\mathcal{Y}}}$ are connected open dense subsets in the sense of ordinary complex topology.

We would like to show that $\Psi'$ is a finite morphism.    First note that since $\Psi'$ is a morphism and surjective,
$$\Psi'(X_t')=\tilde Y_t'=\Psi_t'(X_t')\ \text{and}\ \Psi'^{-1}(\tilde Y_t')=X_t'=\Psi_t'^{-1}(\tilde Y_t').$$
It follows that, if $\Psi'^{-1}(\Psi'(x'))$ is infinite for some $x'\in\mathcal{X}'$ with
$y_\tau=\Psi'(x')\in  \tilde Y_\tau'\subset {\tilde{\mathcal{Y}}'}$,
then ${\Psi'_\tau}^{-1}(y_\tau)=\Psi'^{-1}(y_\tau)\subset X_\tau$ is also infinite.
This cannot occur.
We shall now see that $$\Psi|_{X_\tau}:X_\tau\to  \Psi|_{X_\tau}(X_\tau)\subset \tilde Y_\tau$$ is generically finite,
and that, with $\Psi'_\tau$ defined on $X_\tau':=X_\tau\cap\mathcal{X}'$,
$$C:= {\Psi'_\tau}^{-1}({\Psi'_\tau}(x_{\tau}'))$$ is necessarily finite if $x_{\tau}'\in X_\tau'$.
Here $\Psi|_{X_\tau}$ is a meromorphic map by the same reasoning that $\Phi|_{X_t}$ is meromorphic for every $t\in \Delta$, in Step \ref{step 1}.
This will prove our claim that $\Psi'$ is a finite morphism.

To work on the meromorphic map $\Psi|_{X_\tau}:X_\tau\dashrightarrow \Psi|_{X_\tau}(X_\tau)$ above, we consider the meromorphic map
$$\widehat{\Psi|_{X_\tau}}:=p_2^{-1}\circ \Psi|_{X_\tau}\circ p_1:\hat X_\tau\dashrightarrow \hat Z_{\tau},$$
where
$$p_1:\hat X_\tau\to X_\tau\quad \text{and}\quad p_2:\hat Z_{\tau}\to \Psi|_{X_\tau}(X_\tau)$$
are proper modifications from projective
manifolds $\hat X_\tau$ and $\hat Z_{\tau}$, respectively. Here $\Psi|_{X_\tau}(X_\tau)$ is Moishezon by \cite[Corollary 2.24]{cp}. Now that $\widehat{\Psi|_{X_\tau}}$ is generically finite, as well-known
since $\widehat{\Psi|_{X_\tau}}$ is an algebraic (rational) map and is generically of maximal rank,
one has that $\Psi|_{X_\tau}=p_2\circ \widehat{\Psi|_{X_\tau}}\circ p_1^{-1}$ is also generically finite.
To verify that $C(={\Psi'_\tau}^{-1}({\Psi'_\tau}(x_{\tau}')))$ is finite, suppose otherwise that $C$ is infinite.
We can assume, by further monoidal transformations of $\hat X_\tau$, that
$\widehat{\Psi|_{X_\tau}}:\hat X_\tau\to \hat Z_{\tau}$ is actually a morphism.
By commutativity ${\Psi'_\tau}\circ p_1=p_2\circ \widehat{\Psi|_{X_\tau}}$ where defined,  thus $$\widehat{\Psi|_{X_\tau}}^{-1}(p_2^{-1}(K))\supset
p_1^{-1}({\Psi'_\tau}^{-1}(K))$$ for any set $K\subset \Psi|_{X_\tau}(X_\tau)$, one sees
that $$p_1^{-1}(C)\subset \widehat{\Psi|_{X_\tau}}^{-1}(p_2^{-1}({\Psi'_\tau}(x_{\tau}')))=:\hat C$$
with $\hat C$ being an analytic set in $\hat X_\tau$.  Since $\hat C\subset \hat X_\tau$ has only finitely many
irreducible components, there must exist an irreducible component $\mathcal{J}\subset
\hat X_\tau$ of $\hat C$ such that
$p_1(\mathcal{J})\cap C$ is infinite.  The fact $|p_1(\mathcal{J})|=\infty$ implies that the irreducible analytic set $p_1(\mathcal{J})$ in $X_\tau$
must be of dimension at least one since $X_\tau$ is compact.
Recall $X_\tau'$ and ${\Psi'_\tau}$ above.
This $p_1(\mathcal{J})\cap X_\tau'\supset p_1(\mathcal{J})\cap C$, a nontrivial open subset of $p_1(\mathcal{J})$,
is also of dimension at least one since $p_1(\mathcal{J})$ is irreducible.
By the commutativity ${\Psi'_\tau}\circ p_1=p_2\circ\widehat{\Psi|_{X_\tau}}$,
$${\Psi'_\tau}(p_1(\mathcal{J})\cap X_\tau')\subset p_2(\widehat{\Psi|_{X_\tau}}(\mathcal{J}))
\subset p_2(\widehat{\Psi|_{X_\tau}}(\hat C)) \subset p_2(p_2^{-1}({\Psi'_\tau}(x_{\tau}')))={\Psi'_\tau}(x_{\tau}'),$$ i.e.,
$p_1(\mathcal{J})\cap X_\tau'\subset  {\Psi'_\tau}^{-1}({\Psi'_\tau}(x_{\tau}'))=C$ which by $\dim_{\mathbb{C}}(p_1(\mathcal{J})\cap X_\tau')\ge 1$, gives that
$C\subset X_\tau'$, must also be of dimension at least one.  The facts that ${\Psi'_\tau}(C)$ is a point
and $\dim_{\mathbb{C}} C\ge 1$ contradict that ${\Psi'_\tau}$ on $X_\tau'$ is a local biholomorphism.
As said, this contradiction proves that $\Psi':\mathcal{X}'\to \tilde{\mathcal{Y}}'$ is a finite morphism.

Given that $\Psi':\mathcal{X}'\to \tilde{\mathcal{Y}}'$ is finite, let's recall that in algebraic cases,
a finite surjective morphism between nonsingular varieties
over algebraically closed field is flat (cf. \cite[Exercise 9.3 of Chapter III]{Ht}) and a finite flat morphism $g:X\to Y$
with $Y$ Noetherian gives
that $g_*\mathcal{O}_X$ is a locally free $\mathcal{O}_Y$-module (cf. \cite[Proposition 2 in Section 10 of Chapter 3]{mum}).
Applying these algebraic facts to $\Psi'$,
one sees that
$\Psi'_*\mathcal{O}_{\mathcal{X}'}$ is locally free, say, of rank $r$ on ${\tilde{\mathcal{Y}'}}$.
The local freeness of $\Psi'_*\mathcal{O}_{\mathcal{X}'}$ yields that the cardinality $(=r)$ of
$\Psi'^{-1}(y)$ is independent of $y\in \tilde{\mathcal{Y}}'$.

We shall now prove a similar result as above in the present analytic setting.  This is standard in covering spaces of topology.
For notations and references in the later use, we give some details.
Let $x_0'\in X_{t_0}'$ and $x'\in \mathcal{X}'$ be any point nearby $x_0'$.
Connecting the two points $\Psi'(x_0')=y_0'$ and $\Psi'(x')=y'$ by an analytic curve $\check{C}\subset\mathcal{Y}'$,
one has that $\tilde C:=\Psi'^{-1}(\check{C})$ is an
analytic set with each irreducible component $\tilde C_i$, $1\le i\le k$, being of dimension one and
$\Psi'(\tilde C_i)=\check{C}$ since $\Psi'$ is finite
surjection as remarked above.   Further, these components are pairwise disjoint.  For, if $c\in \tilde C_i\cap \tilde C_j$ with
$\tilde C_i\ne \tilde C_j$, then $\Psi'$ would be seen to be at least ``two-to-one" around $c$, violating the fact
that $\Psi'$ is a local biholomorphism everywhere, as mentioned earlier.   This property of disjointness leads us to arrive at
the fact that each point in
$\Psi'^{-1}(y_0')$ is joined by a unique $\tilde C_i$ to a point in $\Psi'^{-1}(y')$ and vice versa, so that
$\Psi'^{-1}(y_0')$ and $\Psi'^{-1}(y')$ are of the same cardinality $(=k)$.
If the two points $x_0', x' \in \mathcal{X}'$ are not close to each other,
the same result remains valid since $\mathcal{X}'$ is connected.
We conclude that, if with some $t_0$ it holds that $\Psi'_{t_0}$ is injective, then
$\Psi'$ is injective too since ${\Psi_{t_0}'}^{-1}(\tilde Y_{t_0}')=\Psi'^{-1}(\tilde Y_{t_0}')$ as
previously given.

We are going to show that $\Psi'$ is a biholomorphism.  The surjection part is
noted earlier; the injection part is to see that $X_b':=
X_b\cap \mathcal{X}'\ne\emptyset$  (see Step \ref{step 1} for $b\in\Delta^*$). Recall that
$\Phi|_{X_b}$ is bimeromorphic by the assumption on $X_b$,
hence that when $X_b'\ne \emptyset$, $\Psi'_{b}$ is injective,
giving that $\Psi'$ is injective by the preceding paragraph.
Fix a $x_b\in X_b$ such that $\Phi|_{X_b}$ is a morphism
at $x_b$ and that $d(\Phi|_{X_b})(x_b)$ is of rank $n$.
We see that $d\Phi$ is always nonsingular
along the $t$-direction, hence that $d\Phi(x_b)$ is of rank $n+1$.  So
$\Phi:\mathcal{X}\to\mathcal{Y}$ is a local immersion at $x_b$; it is a local biholomorphism at $x_b$
provided that $\mathcal{Y}$ is smooth at $\Phi(x_b)$.   To facilitate our discussion, we make the
claim that
$$\hbox{{\it
$\mathcal{Y}$ is smooth at $\Phi(x_b)$, and thus $\Phi:\mathcal{X}\to \mathcal{Y}$ is  a local biholomorphism at $x_b$.}}$$
We come for a complete discussion of this claim in Step \ref{step 3}.

It is a remarkable fact that the desingularization $\mathcal{R}_{{\mathcal{Y}}}:\tilde{\mathcal{Y}}\to \mathcal{Y}$
can be chosen in such a way that it is an isomorphism away from the singular points of $ \mathcal{Y}$ (\cite[Theorem 5.4.2,
p. 271]{ahv}).
Let's choose such a desingularization in advance.  It follows that $\mathcal{R}_{{\mathcal{Y}}}^{-1}$ is a local
biholomorphism at $\Phi(x_b)$ since ${\mathcal{Y}}$ is assumed smooth there.  So
$\Psi=\mathcal{R}_{{\mathcal{Y}}}^{-1}\circ\Phi$ is a local biholomorphism at $x_b$ and thus $x_b\in \mathcal{X}'$
by construction of $\mathcal{X}'$, giving $x_b\in X_b'$ so $X_b'\ne\emptyset$.
As remarked earlier,  $X_b'\ne\emptyset$ implies that $\Psi'$ is injective
and, in turn, that $\Psi'$ is a biholomorphism.

Now since $\Psi'$ is proved to be biholomorphic, with the fact that
$\Psi$ is meromorphic, we see that $\Psi$ is a bimeromorphic map.   For, the set
$$\widehat{\mathcal{G}(\Psi)}:=\{(y, x)\in \tilde{\mathcal{Y}}\times\mathcal{X}: (x, y)\in \mathcal{G}(\Psi)\}$$ is irreducible and analytic since
$\mathcal{G}(\Psi)$ is so.  The projection morphism
$$\hat p_{\tilde{\mathcal{Y}}}:\widehat{\mathcal{G}(\Psi)}\to \tilde{\mathcal{Y}}$$
is proper as shown in Lemma \ref{py-proper}, and
since $\Psi'$ is biholomorphic and then
$\hat p_{\tilde{\mathcal{Y}}}$ induces $$\mathcal{G}(\Psi'^{-1})\cong \tilde{\mathcal{Y}}'\subset\tilde{\mathcal{Y}},$$
where $\mathcal{G}(\Psi'^{-1})\subset \widehat{\mathcal{G}(\Psi)}$ is open and dense, $\hat p_{\tilde{\mathcal{Y}}}$ is thus a proper modification.
It follows from Definition \ref{modification}, with the fact that
$\hat p_{\tilde{\mathcal{Y}}}$ is a proper modification, that $\Psi^{-1}:\tilde{\mathcal{Y}}\dashrightarrow\mathcal{X}$ is a meromorphic map by Definition \ref{bimero}
and hence a bimeromorphic map since $\Psi\circ \Psi^{-1}={\rm id}_{\tilde{\mathcal{Y}}}$ and $\Psi^{-1}\circ \Psi={\rm id}_{\mathcal{X}}$ (or see \cite[p. 33]{st}).
Having that $\Psi:\mathcal{X}\dashrightarrow\tilde{\mathcal{Y}}$ is a bimeromorphic map, we now know
that $\Phi=\mathcal{R}_{{\mathcal{Y}}}\circ \Psi$ is a bimeromorphic map since the composite of two bimeromorphic maps remains meromorphic
hence bimeromorphic (cf. \cite[pp. 16-17]{Ue} or \cite[2) of Proposition 9]{st}),
as claimed in the second part of the theorem.

\begin{lemma}\label{py-proper}
The map $p_{\mathcal{\tilde Y}}: \mathcal{G}(\Psi)\to \mathcal{\tilde Y}$ is proper.
\end{lemma}
\begin{proof}
  This is equivalent to that $\hat p_{\tilde{\mathcal{Y}}}:\widehat{\mathcal{G}(\Psi)}\to \tilde{\mathcal{Y}}$ is proper
in the preceding paragraph. To justify our assertion above that $\hat p_{\tilde{\mathcal{Y}}}$ is proper, noting that the spaces under consideration are not
compact and the target space $\tilde{\mathcal{Y}}$ may be more general than those in Step \ref{step 1},
one uses Remark \ref{proper-rem}.  Alternatively, let's indicate arguments while dropping most details.
Let $W_3\subset \tilde{\mathcal{Y}}$ be a compact set and suppose that $\hat p_{\tilde{\mathcal{Y}}}^{-1}(W_3)\subset\widehat{\mathcal{G}(\Psi)}$ is not compact.
Then under the projection $\hat p_{\mathcal{X}}:\widehat{\mathcal{G}(\Psi)}\to \mathcal{X}$,
$\hat p_{\mathcal{X}}(\hat p_{\tilde{\mathcal{Y}}}^{-1}(W_3))$ is closed but not compact, so that
there exists a sequence $t_k\in\Delta$ with $t_k\to \partial\Delta$ and
 $\{t_k\}_k\subset \pi(\hat p_{\mathcal{X}}(\hat p_{\tilde{\mathcal{Y}}}^{-1}(W_3)))$.  As in Lemma \ref{proper}, to
show that for every $t\in \Delta$ $$\Psi(X_t)\subset \tilde Y_t=\pi_{\tilde{\mathcal{Y}}}^{-1}(t),$$
where $\pi_{\tilde{\mathcal{Y}}}:\tilde{\mathcal{Y}}\to \Delta$ is the projection via $\tilde{\mathcal{Y}}\to\mathcal{Y}\to\Delta$,
we are reduced to showing that $\Psi(x)\in \tilde{Y_t}=\pi_{\tilde{\mathcal{Y}}}^{-1}(t)$ if $x\in X_t$.
Now that $\hat p_{\mathcal{X}}:\widehat{\mathcal{G}(\Psi)}\to\mathcal{X}$ being a modification, is a biholomorphism
between the dense and open subsets $U$, $V$ of $\mathcal{G}$, $\mathcal{X}$, respectively,
with $V\cap \mathcal{S}(\Psi)=\emptyset$, i.e., $\Psi$ being a morphism on $V$ and $U=\{(\Psi(x), x)\}_{x\in V}$.
In the remaining part, with $\tilde Y_t$ in place of $\mathbb{P}^N_t$ in Lemma \ref{proper},
by exactly the same arguments one can show that
any $(y, x)\in \hat p_{\mathcal{X}}^{-1}(x)\subset \widehat{\mathcal{G}(\Psi)}$ can be approached by
a sequence $(y_j, x_j)\in U$ with $x_j\in V$ and $y_j=\Psi(x_j)$, in such a way that
$y=\lim_j y_j\in \lim_j\tilde{Y}_{t_j}\subset\tilde{Y}_t$.
It implies the similar conclusion $\Psi(X_t)\subset \tilde Y_t$ for every $t\in \Delta$.
As in Lemma \ref{proper},  corresponding to every $t_k$ given above, there is a $(y_k, x_k)\in \hat p_{\tilde{\mathcal{Y}}}^{-1}(W_3)\subset \widehat{\mathcal{G}(\Psi)}$
with $t_k=\pi(\hat p_{\mathcal{X}}(y_k, x_k))=\pi(x_k)$, i.e., $x_k\in X_{t_k}$,
such that $\pi_{\tilde{\mathcal{Y}}}(W_3)\ni\pi_{\tilde{\mathcal{Y}}}(y_k)=t_k\to \partial\Delta$, contradicting that $W_3$ is compact.
\end{proof}

As promised, let's treat the smoothness issue above.
\begin{step}\label{step 3}
Smoothness of $\mathcal{Y}$ at $\Phi(x_b)$
\end{step}
Since $x_b\in X_b$, the idea is to refine the process of choosing
$b\in\Delta$ in Step \ref{step 1} in such a way that $\Phi(X_b)$ is not entirely contained in the set  of
singular points of $\mathcal{Y}$.  If so, it follows by an analogous procedure as before that, one
can choose $x_b\in X_b$ such that $\mathcal{Y}$ is smooth at $\Phi(x_b)$, as desired.    This refinement is as follows.

Again, we start with
a global line bundle $L$ on $\mathcal{X}$; it follows that
 $H^0(X_t, L^{\otimes q(t)}|_{X_t})$ gives a bimeromorphic embedding of $X_t$, for every $t\in \Delta$ together with
a choice of $q(t)\in \mathbb{N}$ that depends on $t$.  Note that if
$$\mathcal{E}:=\{e_1, e_2, \ldots, e_k\}\subset H^0(X_t, L^{\otimes q(t)}|_{X_t})$$ contains a basis,
then the Kodaira map associated with $\mathcal{E}$ is still a bimeromorphic map on $X_t$.
The set $\Delta$ is uncountable while $\mathbb{N}$ is countable;
we easily infer that there exists an uncountable set $\Lambda\subset\Delta$ and some $q\in \mathbb{N}$ such that
$q(t)=q$ for each $t\in \Lambda$.    With this given $q$, $h^0(X_t, L^{\otimes q}|_{X_t})$ is locally constant
outside a proper analytic set $W_4\subset \Delta$ as seen by using Theorem \ref{Upper semi-continuity} or
by \cite[3) of Theorem 1.4, p. 6]{Ue} which
gives in Corollary \ref{gct} the cohomological flatness of $L^{\otimes q}$ in dimension $0$ over $\Delta\setminus W_4$.
By reducing $\Lambda$ while maintaining uncountability,
we can assume that $\Lambda\subset\Delta\setminus W_4$ since $W_4$ in this case can
only be a discrete subset of $\Delta$.
We can further reduce $\Lambda$ and assume that $\Lambda$ is relatively compact in $\Delta\setminus W_4$
(so that $\bar\Lambda\subset \Delta\setminus W_4$) since with $\Lambda=\cup_i(\Lambda\cap O_i)$ for a covering
of $\Delta\setminus W_4$ by countably many  open and relatively compact subsets $O_i$, $\Lambda\cap O_{i_0}$
must be uncountable for some $i_0$.   Also, one sees that there exists a $z\in \Delta\setminus W_4$ such that
for any neighborhood $\mathcal{U}\ni z$, $\mathcal{U}\cap \Lambda$ remains uncountable.  For, by a similar argument as above working
on $O_{i_0}$ there exist $$O_{i_0}=O'_{1}\supset O'_{2}\supset O'_{3}\supset\cdots$$
such that $\Lambda\cap O'_{k}$ is
uncountable for each $k=1, 2, \ldots,$ and $O'_{k}$ converges to some $z\in \bar O_{i_0}\subset \Delta\setminus W_4$.

Having the above uncountable subset $\Lambda$, we are ready to reconstruct the Kodaira
map $\Phi$.  As in Step \ref{step 1}, by Theorem A of Cartan (= Theorem \ref{cartan-a}) we can choose a set $\mathcal{T}$ linearly spanned by sections
$$\{s_0,s_1,\ldots,s_N\}\subset \pi_*L^{\otimes q}(\Delta)$$
such that the germs $(s_0)_z, (s_1)_z,\ldots,(s_N)_z$ at $z$ generate the stalk $(\pi_*L^{\otimes q})_z$,
hence that they generate $(\pi_*L^{\otimes q})_w$ for every $w$ in a neighborhood $U$ of $z$ as a property from
the coherent sheaves (cf. \cite[Lemma 7.1.3]{h90}).  We denote by $\Phi_\mathcal{T}$ the Kodaira map associated with
$\mathcal{T}$.  As said, $\Lambda\cap \mathcal{N}$ is uncountable for any neighborhood $\mathcal{N}\ni z$, so $\Lambda_U:=\Lambda\cap U$ remains uncountable.
By the cohomological flatness above, the natural map
$$\lambda:=\lambda_w: (\pi_*L^{\otimes q})_w\to H^0(X_w, L^{\otimes q}|_{X_w})$$
is surjective for every $w\in U\setminus W_4\neq\emptyset$.
This, together with our preceding construction,
yields that $\Phi_\mathcal{T}$, when restricted to $X_t$,  induces a bimeromorphic embedding of $X_t$ for every
$t\in \Lambda_{U}\subset U\setminus W_4$.  Here, the images
$\{\lambda(s_0), \lambda(s_1), \ldots,\lambda(s_N)\}\subset H^0(X_t, L^{\otimes q}|_{X_t})$ may
not be linearly independent, nevertheless they give a bimeromorphic embedding on $X_t$ as already remarked.

We denote by $\mathcal{Y}_\mathcal{T}$ the image  $\Phi_\mathcal{T}(\mathcal{X})$.
Write $$\mathcal{S}=\mathcal{S}_h\cup \mathcal{S}_v$$ for the set of singular points of
$\mathcal{Y}_\mathcal{T}$, where the vertical part $ \mathcal{S}_v$ consists of those (irreducible) components of $\mathcal{S}$ that are mapped to
points in $\Delta$ via $p_{\mathcal{Y}_\mathcal{T}}:\mathcal{Y}_\mathcal{T}\to\Delta$.   Set
$$\Xi:=p_{\mathcal{Y}_\mathcal{T}}(\mathcal{S}_v)\subset\Delta.$$
There are only countably many components
in $\mathcal{S}_v$, say, by Remmert's proper mapping theorem (= Theorem \ref{remmert}) so that $\Xi$ is discrete, and by that every fiber $X_t$ is compact.  The components in $\mathcal{S}_v$ are fiberwise separated away from one another as $\Xi$ is discrete.
 Since  $\Xi$ is countable and $\Lambda_U$ is uncountable (or $\Xi$ is discrete and
$\Lambda_U$ is not discrete), $\Lambda_U\setminus \Xi\neq\emptyset$.  Now pick some $0\ne o\in\Lambda_U\setminus \Xi$.
This means that the $\Phi_\mathcal{T}$-image of $X_o$ is not contained in $\mathcal{S}_v$ and in turn,
neither in singularities $\mathcal{S}$ of $\mathcal{Y}_\mathcal{T}$.
The smoothness in the beginning is proved, so long as with the choice of sections $E$ in Step \ref{step 1}
changed to $\mathcal{T}$, the $X_b$, $\Phi$ ($=\Phi_E$), $\mathcal{Y}$ ($=\mathcal{Y}_E$) etc. in Step \ref{step 1}
are replaced by $X_o$, $\Phi_\mathcal{T}$, $\mathcal{Y}_\mathcal{T}$ throughout.  Of course, with these $X_o$, $\Phi_\mathcal{T}$, $\mathcal{Y}_\mathcal{T}$ the reasoning for Step \ref{step 2} remains unaltered.  This completes our Step \ref{step 3}.
\end{proof}

\end{CJK*}

\end{document}